\documentclass[11pt]{article}%
\usepackage{amssymb,amsmath,amsfonts,amsthm,array,bm,bbm,color}%
\usepackage{graphicx}
\providecommand{\U}[1]{\protect\rule{.1in}{.1in}}

\numberwithin{equation}{section}

\newtheorem{theorem}{Theorem}[section]
\newtheorem{lemma}[theorem]{Lemma}
\newtheorem{corollary}[theorem]{Corollary}
\newtheorem{proposition}[theorem]{Proposition}
\newtheorem{remark}[theorem]{Remark}

\newtheorem{definition}[theorem]{Definition}

\def\<{\langle}
\def\>{\rangle}
\def\d{{\rm d}}
\def\L{\mathcal{L}}
\def\div{{\rm div}}
\def\E{\mathbb{E}}

\def\N{\mathbb{N}}
\def\P{\mathbb{P}}
\def\R{\mathbb{R}}
\def\T{\mathbb{T}}
\def\Z{\mathbb{Z}}

\def\eps{\varepsilon}

\begin{document}

\title{High mode transport noise improves vorticity blow-up control in 3D Navier-Stokes equations} 
\author{Franco Flandoli\footnote{Email: franco.flandoli@sns.it. Scuola Normale Superiore of Pisa, Piazza dei Cavalieri 7, 56124 Pisa, Italy.}
\quad Dejun Luo\footnote{Email: luodj@amss.ac.cn. Key Laboratory of RCSDS, Academy of Mathematics and Systems Science, Chinese Academy of Sciences, Beijing, China and School of Mathematical Sciences, University of the Chinese Academy of Sciences, Beijing, China.} }

\maketitle

\begin{abstract}
The paper is concerned with the problem of regularization
by noise of 3D Navier-Stokes equations. As opposed to several attempts made with
additive noise which remained inconclusive, we show here that a
suitable multiplicative noise of transport type has a regularizing effect. It
is proven that stochastic transport noise provides a bound on vorticity which gives well posedness, with high probability. The result holds for sufficiently large noise intensity and sufficiently high spectrum of the noise.
\end{abstract}



\section{Introduction}

Well posedness of the 3D incompressible Navier-Stokes equations is a famous
open problem \cite{Feff}. Around this central problem many variants have been
identified which are still very difficult and could contribute to build a
general picture. One of them is the well posedness of \textit{stochastic} 3D
Navier-Stokes equations. In spite of several attempts, see for instance
\cite{FlaRom CKN, DapDeb, MikRoz, FlaRom Markov}, it remained unsolved. The logic behind these
attempts is the known fact that noise sometimes improves the theory of
differential equations, a fact certainly true in finite dimensions \cite{Veret, Kry-Roeck, Davie}
and also true for some infinite dimensional systems, like \cite{Gyongy and co, FGP, FGP2,
DebTsu, DFPR, DFV, BBF, DFRV, GessMau, ButMyt, GassiatGess}, but not for all
examples of PDEs and noise, as shown for different examples related to Euler
equations in \cite{FGP2} and \cite{CFF}. For PDEs of parabolic
type, additive noise was always invoked as the most natural candidate to prove
the above mentioned property of \textit{regularization by noise}.
Multiplicative transport noise was used only in inviscid problems, like
\cite{FGP, FGP2, DFV} devoted to transport, 2D Euler (point
vortices) and 1D Vlasov-Poisson (point charges) equations, respectively. Here
we change perspective and use multiplicative transport noise for the 3D
Navier-Stokes equations. The result is a particular regularization by noise
phenomenon. We prove that stochastic transport noise suppresses vorticity increase
and gives long term well posedness, with high probability. Opposite to all
results mentioned above that hold for any (non-null) intensity of the noise,
the result proved here holds for sufficiently large noise intensity and
sufficiently high spectrum of the noise. In a sense, it is similar to the
stabilization by noise result of \cite{ArnoldCW, Arnold} (see the
acknowledgments at the end of the paper). In devising this result we have been very influenced also by \cite{FlaLuoAoP, Galeati, FGL, Iyer}.

Before we give our main results in Section 1.3, we recall more details on
regularization by noise in Section 1.1 and discuss a partial motivation for
transport noise -- including its main limitation -- in Section 1.2.

\subsection{General remarks on regularization by noise}

The idea that noise may improve the existence and uniqueness theory of 3D
Navier-Stokes equations is a long standing one. In a na\"{\i}ve way it is
based on the analogy with the case of finite dimensional differential
equations%
\[
\d X_{t}=b\left( t,X_{t}\right) \d t+\d W_{t},\quad X_{0}=x\in \mathbb{R}^{d}
\]%
where an additive $d$-dimensional noise $W_{t}$ restores existence and
uniqueness even when the drift $b:\left[ 0,T\right] \times \mathbb{R}%
^{d}\rightarrow \mathbb{R}^{d}$ is just bounded measurable \cite{Veret}.
Such result attracted much attention in finite dimensions, with further
progresses like \cite{Kry-Roeck, Davie} and many others, and it was
extended to infinite dimensions \cite{Gyongy and co, DaPrato Fla, DFPR, DFRV, ButMyt} but covering only one-dimensional
Partial Differential Equations (PDEs) of parabolic type, with nonlinearities
which are irregular, but not in the direction of the irregularity of the
inertial term of Navier-Stokes equations (the drift of the above mentioned
works is for instance a bounded measurable map on a suitable Hilbert space).
There are also heuristic arguments, perhaps less na\"{\i}ve, which may give
the feeling that some kind of noise, or just randomness in the initial
conditions, may exclude the realization of very special dynamical paths with
so strong vortex stretching to lead to a singularity in finite time;\
perhaps the phenomenon discussed by \cite{Tao} is prevented by noise. It is
also the opinion of many experts that during fully develped turbulence
singularities should not appear, maybe opposite to transient-to-turbulence
regimes where a high degree of organization of the motion can still occur
and lead to blow-up; the link between turbulence regime and noisy PDEs is
heuristic, but see the discussion in Section \ref{subsect advection} below.

The problem, whether additive noise ``regularizes'' 3D Navier-Stokes equations
remains open but some contributions have been made. Among others, let us
remind the following ones:

(1) a Caffarelli-Kohn-Nirenberg theory has been developed for stochastic 3D
Navier-Stokes equations \cite{FlaRom CKN}, with the interesting consequence
that, at every time $t$, the random set of spatial singularities $%
S_{t}\left( \omega \right) $ is empty with probability one, namely
\[
\mathbb{P}\left( \omega \in \Omega :S_{t}\left( \omega \right) =\emptyset
\right) =1,
\]
having denoted by $\left( \Omega ,\mathcal{F},\mathbb{P}\right) \ $the
underlying probability space (full absence of singularities would be the
statement%
\[
\mathbb{P}\bigg( \omega \in \Omega :\bigcap\limits_{t\geq 0}S_{t}\left(
\omega \right) =\emptyset \bigg) =1
\]
but this remains open);

(2) The infinite dimensional Kolmogorov equation associated to the stochastic
3D Navier-Stokes equations has been solved \cite{DapDeb}, opening a door for
the application to uniqueness of weak solutions -- not yet reached due to
regularity problems of the solution to the Kolmogorov equation (see \cite{HZZ} for
a recent result on non-uniqueness in law of stochastic 3D Navier-Stokes 
equations via the convex integration method);

(3) Markov selections with the \textit{Strong Feller property} -- elaborating
a preliminary result of \cite{DapDeb} -- have been constructed \cite{FlaRom
Markov}, proving a continous dependence result on initial conditions, due to
noise, which has no counterpart in the deterministic theory of 3D
Navier-Stokes equations.

\subsection{Transport and advection noise\label{subsect advection}}

The previous results, in analogy with the finite dimensional case, have been
obtained by an additive noise with suitable non-degeneracy properties. It is
the first noise that is natural to investigate, used for instance in
numerical simulations to accelerate transition to turbulence \cite{Vincent
Meneguzzi}. But from the physical viewpoint the justification is weak. On
the heuristic ground, on the contrary, a multiplicative noise of advection
type is more motivated, by the idea of separating large and small scales and
model the small ones by noise, corresponding to some intuition of
turbulence. The 3D Navier-Stokes equations perturbed by such advection noise
have the form
\begin{equation}\label{stoch NS}
\partial _{t}\xi +\mathcal{L}_{u}\xi +\mathcal{L}_{\circ \eta }\xi =\Delta \xi
\end{equation}%
where the notations will be defined during the next arguments.

Let us discuss this issue of noise approximation of small scales for the 3D
Navier-Stokes equations written in vorticity form, say on the torus $\mathbb{%
T}^{3}=\mathbb{R}^{3}/\mathbb{Z}^{3}$ with periodic boundary conditions:%
\begin{equation} \label{det NS}
\partial _{t}\xi +\mathcal{L}_{u}\xi =\Delta \xi ,\quad \xi |_{t=0}=\xi _{0}
\end{equation}%
where $u$ is the velocity field, satisfying also ${\rm div}\, u=0$, $\xi =%
{\rm curl}\, u$ is the vorticity field, the viscosity is set to one to avoid
too many parameters below, and $\mathcal{L}_{u}\xi $ is the Lie derivative
\[
\mathcal{L}_{u}\xi =u\cdot \nabla \xi -\xi \cdot \nabla u.
\]%
In various functions spaces the link between $u$ and $\xi $ is uniquely
inverted by a Biot-Savart operator $B$, so that we write $u=B\xi $. Assume
\[
\xi _{0}=\xi _{0,L}+\xi _{0,S}
\]%
(the subscript $L$ stands for ``Large'' scale part, $S$ for ``Small'' part)
and assume we can solve the system%
\begin{eqnarray*}
\partial _{t}\xi _{L}+\mathcal{L}_{u}\xi _{L} &=&\Delta \xi _{L},\quad \xi
_{L}|_{t=0}=\xi _{0,L} \\
\partial _{t}\xi _{S}+\mathcal{L}_{u}\xi _{S} &=&\Delta \xi _{S},\quad \xi
_{S}|_{t=0}=\xi _{0,S}
\end{eqnarray*}%
where $u=B\left( \xi _{L} +\xi _{S}\right) $. The sum $\xi =\xi_{L} +\xi_{S}$
solves (\ref{det NS}). Very heuristically, we could think that in some limit
and in a regime of turbulent fluid the small component $\xi_{S}$ varies in
time very rapidly compared to the larger one $\xi_{L}$ (unfortunately
such a separation of scales has never been proved to hold so strictly) so
that
\[
\eta :=u_{S}
\]
can be considered as an approximation of white noise. The equation for $\xi
_{L}$ is%
\[
\partial _{t}\xi_{L} +\mathcal{L}_{u_L} \xi_{L} +\mathcal{L}_{\eta } \xi_{L}
= \Delta \xi_{L},\quad \xi_{L}|_{t=0} = \xi_{0,L}
\]%
which has precisely the form (\ref{stoch NS}). Above we have used the more
precise notation $\mathcal{L}_{\circ \eta }\xi $ to anticipate the fact that
we work with Stratonovich stochastic integrals, the correct ones -- when one
can prove a Wong-Zakai result -- as limit of regular approximations of white
noise.

A key issue which emerges from the previous heuristics is that the
multiplicative structure of the noise is related to the Lie derivative $%
\mathcal{L}_{\circ \eta }\xi $, because that is the form of the inertial
term. We call \textit{advection} term the expression $\mathcal{L}_{\circ
\eta }\xi $. It is composed of the \textit{transport term} $\mathcal{T}%
_{\circ \eta }\xi $ and the \textit{vortex stretching term }$\mathcal{S}%
_{\circ \eta }\xi $ defined respectively as
\begin{eqnarray*}
\mathcal{T}_{\circ \eta }\xi  =\eta \cdot \nabla \xi, \quad
\mathcal{S}_{\circ \eta }\xi  =\xi \cdot \nabla \eta
\end{eqnarray*}%
(with suitable Stratonovich interpretation). The advection structure of the
noise was stressed also by the geometric approach of \cite{Holm}.\ The
effect on fluid dynamics of an advection noise is to stretch vorticity in a
relatively strong way. In the case of full advection noise $\mathcal{L}_{\circ\eta}\xi
=\mathcal{T}_{\circ\eta}\xi - \mathcal{S}_{\circ\eta}\xi$ we meet an
intermediate but unlucky situation. On one side, the Stratonovich-It\^{o}
corrector is again a multiple of the Laplacian (see Proposition \ref{prop-key-identity}), which
goes in the right direction. But on the other side certain main estimates
blow-up in the scaling limit $N\rightarrow\infty$ considered below, due to the additional stretching
introduced by the noise, and thus we cannot prove convergence of the
approximating scheme. Details are given in Section 6. Therefore, unfortunately, we are
unable to prove our result for the advection noise described so far; instead, we
restrict ourselves to the transport noise $\mathcal{T}_{\circ \eta }\xi $,
which has only the effect of an additional background motion of the fluid,
without stretching of the vector quantities. What we prove is that such
random background motion has a regularizing effect;\ or more precisely, as
stated in the title, in the limit of high modes, this transport noise
improves vorticity blow-up control. In Section \ref{appendix-3} we make an effort to justify a model based on pure transport noise. The justification is incomplete but may suggest new ideas.

As discussed in this section, we are thus aware of the limitation, from a
physical viewpoint, of our choice of the noise. However, a number of reasons
suggest to consider at least this initial case: (i) additive noise is not
carefully motivated as well, since body forces in real fluids are usually
extremely smooth; hence the transport noise $\mathcal{T}_{\circ\eta}\xi$ is at
least in a similar speculative line of research; (ii) it is the first noise
discovered to improve the theory of 3D Navier-Stokes equations; (iii) the proof
given below, especially for what concerns the Stratonovich-It\^o correction term 
$S_{\theta}(\xi)$ defined in \eqref{I-theta},
is highly non-trivial and may constitute in the future a
building block for the investigation of more difficult and realistic cases.
Last but not least, knowing that such noise has a property of vorticity depletion, the
intriguing question arises whether it is possible to implement technologically
a similar mechanism.

\subsection{Main results}

In view of the discussions above, we consider the 3D Navier-Stokes equations on the torus $\mathbb{T}^{3}=\mathbb{R}^{3}/\mathbb{Z}^{3}$ in vorticity form perturbed by a \textit{transport noise}:
  $$\partial_t \xi + \L_u \xi = \Delta \xi + \Pi( \eta\cdot\nabla \xi),$$
where we apply the Leray projection operator $\Pi$ to the noise part to make it divergence free. This is a central element of our model and analysis that we now briefly comment. Without the projection the model is not meaningful: if $\xi$ is a solution, the equation becomes an identity between three divergence free terms and a non-divergence free one (the transport noise term), which hence should be equal to zero. Thus the projection is strictly necessary; but the consequence is that computations below require a much greater effort (see e.g. Section \ref{sect-5}). We also mention that, since the vorticity $\xi$ and the noise $\eta$ are divergence free, we have
  $$\<\xi, \Pi( \eta\cdot\nabla \xi)\>_{L^2}= \<\xi, \eta\cdot \nabla\xi\>_{L^2}= 0.$$
This implies that the above equation has the same a priori $L^2$-estimate as the deterministic equation \eqref{det NS}; as a result, it is globally well-posed for small initial values and enjoys the usual estimate on the blow-up time for large ones. Therefore, at first glance, the multiplicative transport noise has no regularizing effect on the 3D Navier-Stokes equations. However, by taking a suitable scaling limit, we will show the phenomenon of dissipation enhancement (cf. \cite{Constantin, Iyer} and references therein) which implies that, for given large initial data, the above equation admits a pathwise unique global solution, with large probability.

The space-time noise $\eta$ used in this paper has the explicit form:
  $$\eta(t,x)= \frac{C_\nu}{\|\theta\|_{\ell^2}} \sum_{k\in \Z^3_0} \sum_{\alpha=1}^2 \theta_k\, \sigma_{k,\alpha}(x)\, \dot W^{k,\alpha}_t, $$
where, for some constant $\nu>0$, $C_\nu= \sqrt{3\nu/2}$ is the noise intensity and the coefficient $3/2$ is chosen to simplify some of the equations below; $\Z^3_0$ is the nonzero lattice points and $\theta \in \ell^2= \ell^2(\Z^3_0)$, the usual space of square summable sequences indexed by $\Z^3_0$. In the following we will mainly consider those $\theta$ with only finitely many non-zero components. The family $\{\sigma_{k,\alpha}: k\in \Z^3_0, \alpha =1,2\}$ of complex divergence free vector fields (see the next section for explicit definitions) is a CONS of the space
  $$H_{\mathbb C} = \bigg\{v\in L^2(\T^3, \mathbb C^3): \int_{\T^3} v \,\d x=0,\, {\rm div}\, v=0 \bigg\}.$$
Finally, $\{W^{k,\alpha}: k\in \Z^3_0, \alpha=1,2\}$ are independent complex-valued Brownian motions defined on some filtered probability space $(\Omega,\mathcal{F},(\mathcal{F}_{t}),\mathbb{P})$. Thus, the equation studied in the paper can be written more precisely as below:
  \begin{equation}\label{SNSE-vort}
  \d \xi + \L_u \xi\,\d t = \Delta \xi\,\d t + \frac{C_\nu}{\|\theta\|_{\ell^2}} \sum_{k\in \Z^3_0} \sum_{\alpha=1}^2 \theta_k \Pi( \sigma_{k,\alpha}\cdot \nabla \xi) \circ \d W^{k,\alpha}_t.
  \end{equation}
To save notations, we shall simply write $\sum_{k,\alpha}$ for $\sum_{k\in \Z^3_0} \sum_{\alpha =1}^2$.

We need some more notations in order to introduce the definition of solution. As usual, we write $H$ for the real subspace of $H_{\mathbb C}$. Denote by $\<\cdot, \cdot\>_{L^2}$ the $L^2$-inner product in $H$, and $V$ the intersection of $H$ with the first order Sobolev space $H^1(\T^3,\R^3)$. In the following we write $\L_{u}^\ast$ for the adjoint operator of the Lie derivative: for any vector fields $X, Y\in V$, $\<\L_{u} X, Y\>_{L^2} = -\<X, \L_{u}^\ast Y\>_{L^2}$. Since $u$ is divergence free, one has $\L_{u}^\ast Y =u\cdot \nabla Y + (\nabla u)^\ast Y$, where for $i=1,2,3$, $((\nabla u)^\ast Y)_i= \sum_{j=1}^3 Y_j\partial_i u_j$.

\begin{definition}\label{def-1}
\label{def sol Strat}Given a filtered probability space $(\Omega
,\mathcal{F},(\mathcal{F}_{t}),\mathbb{P})$ and a family of independent
$(\mathcal{F}_{t})$-complex Brownian motions $\{W^{k,\alpha}:k\in
\mathbb{Z}_{0}^{3},\alpha=1,2\}$ defined on $\Omega$, we say that an
$(\mathcal{F}_{t})$-progressively measurable process $\xi$ is a strong
solution of the Stratonovich equation (\ref{SNSE-vort}) if it has trajectories
of class $L^{\infty}(0,T;H)\cap L^{2}( 0,T;V)  $ and in $C\left(
[0,T],H^{-\delta}\right)  $ and, for any divergence free vector field $v\in C^{\infty
}(\mathbb{T}^{3},\mathbb{R}^{3})$, the process $\left\langle
\xi_{t}, \Pi(\sigma_{k,\alpha}\cdot \nabla v) \right\rangle _{L^{2}}$ is an
$(\mathcal{F}_{t})$-continuous semimartingale and $\mathbb{P}$-a.s. the
following identity holds for all $t\in [0,T]$:
\[ \aligned
\left\langle \xi_{t},v\right\rangle _{L^{2}} =&\, \left\langle \xi_{0}%
,v\right\rangle _{L^{2}}+\int_{0}^{t}\left\langle \xi_{s},\mathcal{L}_{u_{s}%
}^{\ast}v\right\rangle _{L^{2}}\,\mathrm{d}s+\int_{0}^{t}\left\langle \xi
_{s},\Delta v\right\rangle _{L^{2}}\,\mathrm{d}s \\
&\, -\frac{C_{\nu}}{\Vert \theta\Vert_{\ell^{2}}}\sum_{k,\alpha}\theta_{k}\int_{0}^{t}\left\langle
\xi_{s}, \Pi(\sigma_{k,\alpha}\cdot \nabla v) \right\rangle_{L^{2}}
\circ \mathrm{d}W_{s}^{k,\alpha}.
\endaligned \]
\end{definition}

Recall that Stratonovich integrals are well defined when the integrand is an
$(\mathcal{F}_{t})$-continuous semimartingale, see \cite{Kunita} for the
definition and theory used here. The rules of stochastic calculus give us%
\begin{align*}
&\, \int_{0}^{t}\left\langle \xi_{s}, \Pi(\sigma_{k,\alpha}\cdot \nabla v) \right\rangle_{L^{2}} \circ\mathrm{d}W_{s}^{k,\alpha} \\
= &\, \int_{0}^{t}\left\langle \xi_{s}, \Pi(\sigma_{k,\alpha}\cdot \nabla v) \right\rangle_{L^{2}} \,\mathrm{d}W_{s}^{k,\alpha} +\frac{1}{2} \Big[ \left\langle \xi_\cdot, \Pi(\sigma_{k,\alpha}\cdot \nabla v) \right\rangle_{L^{2}} ,W_{\cdot}^{k,\alpha}\Big]_{t},
\end{align*}
where the last term is the joint quadratic variation. The identity in Definition \ref{def sol Strat} with $\Pi( \sigma_{k,\alpha}%
\cdot\nabla v)$ replacing $v$ gives us%
\[
\left\langle \xi_{t},\Pi (\sigma_{k,\alpha}\cdot \nabla v)\right\rangle _{L^{2}%
}=V_{t} - \frac{C_{\nu}}{\Vert\theta\Vert_{\ell^{2}}} \sum_{l,\beta} \theta_{l}
\int_{0}^{t} \left\langle \xi_{s},
\Pi\big[ \sigma_{l,\beta}\cdot\nabla \Pi (\sigma_{k,\alpha}\cdot \nabla v) \big] \right\rangle _{L^{2}}%
\circ\mathrm{d}W_{s}^{l,\beta},
\]
where $V_{t}$ has bounded variation. Hence, by \eqref{qudratic-var} below, the only term which has non-zero
joint quadratic variation with $W_{\cdot}^{k,\alpha}$ is
\[
- \frac{C_{\nu}}{\Vert\theta\Vert_{\ell^{2}}} \theta_{-k}\int_{0}^{t}\left\langle \xi_{s},
\Pi\big[ \sigma_{-k,\alpha} \cdot\nabla \Pi (\sigma_{k,\alpha}\cdot \nabla v) \big] \right\rangle _{L^{2}} \circ \mathrm{d} W_{s}^{-k,\alpha},
\]
giving rise to
\begin{align*}
\frac{1}{2} \Big[ \left\langle \xi_\cdot, \Pi(\sigma_{k,\alpha}\cdot \nabla v) \right\rangle_{L^{2}} ,W_{\cdot}^{k,\alpha}\Big]_{t}
& = - \frac{C_{\nu}}{\Vert\theta\Vert_{\ell^{2}}}\theta_{-k}%
\int_{0}^{t}\left\langle \xi_{s},
\Pi\big[ \sigma_{-k,\alpha} \cdot\nabla \Pi (\sigma_{k,\alpha}\cdot \nabla v) \big] \right\rangle_{L^{2}} \,\d s.
\end{align*}
We have proved one implication of the following proposition. The proof of the
other is similar.

\begin{proposition}\label{1-prop}
An $(\mathcal{F}_{t})$-progressively measurable process $\xi$ with
paths of class $L^{\infty}(0,T;H) \cap L^{2}( 0,T;V)  $ and
in $C\left(  [0,T],H^{-\delta}\right)  $ is a strong solution of the
Stratonovich equation (\ref{SNSE-vort}) if and only if for any divergence free vector field
$v\in C^{\infty}(\mathbb{T}^{3},\mathbb{R}^{3})$, $\mathbb{P}$-a.s., the
following identity holds for all $t\in\lbrack0,T]$:
\begin{equation}\label{prop-def}
\aligned
\left\langle \xi_{t},v\right\rangle _{L^{2}}  =& \left\langle \xi
_{0},v\right\rangle _{L^{2}}+\int_{0}^{t}\left\langle \xi_{s},\mathcal{L}%
_{u_{s}}^{\ast}v\right\rangle _{L^{2}}\,\mathrm{d}s+\int_{0}^{t}\left\langle
\xi_{s},\Delta v\right\rangle _{L^{2}}\,\mathrm{d}s\\
& -\frac{C_{\nu}}{\Vert\theta\Vert_{\ell^{2}}} \sum_{k,\alpha}\theta_{k}%
\int_{0}^{t}\left\langle \xi_{s} ,
\Pi(\sigma_{k,\alpha}\cdot\nabla v) \right\rangle _{L^{2}}\,\mathrm{d}W_{s}^{k,\alpha}\\
&+ \frac{C_{\nu}^2}{\Vert\theta\Vert_{\ell^{2}}^2} \sum_{k,\alpha} \theta_{k} \theta_{-k}
\int_{0}^{t}\left\langle \xi_{s},
\Pi\big[ \sigma_{-k,\alpha} \cdot\nabla \Pi (\sigma_{k,\alpha}\cdot \nabla v) \big] \right\rangle_{L^{2}} \,\d s.
\endaligned
\end{equation}
\end{proposition}

We shall always assume that $\theta\in \ell^2$ is radially symmetric, i.e.,
  \begin{equation}\label{theta-sym}
  \theta_k = \theta_l \quad \mbox{whenever} \quad |k|=|l|.
  \end{equation}
The above equation \eqref{prop-def} can be written in the weak form as
  $$\d \xi + \L_u \xi\,\d t = \Delta \xi\,\d t + \frac{C_\nu}{\|\theta\|_{\ell^2}} \sum_{k,\alpha} \theta_k \Pi(\sigma_{k,\alpha} \cdot \nabla \xi) \, \d W^{k,\alpha}_t + \frac{C_\nu^2}{\|\theta\|_{\ell^2}^2} \sum_{k, \alpha} \theta_k^2 \, \Pi\big[ \sigma_{k,\alpha} \cdot \nabla \Pi(\sigma_{-k,\alpha}\cdot \nabla \xi) \big] \d t.$$
To simplify the notation, we denote the Stratonovich-It\^o correction term by
  \begin{equation}\label{I-theta}
  S_\theta(\xi)= \frac{C_\nu^2}{\|\theta \|_{\ell^2}^2}\sum_{k,\alpha} \theta_k^2\, \Pi\big[ \sigma_{k,\alpha}\cdot \nabla \Pi (\sigma_{-k,\alpha}\cdot\nabla \xi) \big],
  \end{equation}
which, like the Laplacian, is a symmetric operator with respect to the $L^2$-inner product of divergence free vector fields. This term looks a little complicated, but we can show that, if $\xi$ is smooth, then it has a simple limit when taking a special sequence of $\{\theta^N\}_{N\geq 1}$, see \eqref{limit-I-theta-N} below.  Summarizing these discussions, we can rewrite the above equation as
  \begin{equation}\label{SNSE-vort-Ito}
  \d \xi + \L_u \xi\,\d t = \big[\Delta \xi + S_\theta(\xi)\big] \,\d t  + \frac{C_\nu}{\|\theta\|_{\ell^2}} \sum_{k,\alpha} \theta_k \Pi(\sigma_{k,\alpha}\cdot \nabla \xi) \, \d W^{k,\alpha}_t.
  \end{equation}

The equation \eqref{SNSE-vort-Ito}, due to the presence of the nonlinear part $\L_u \xi$, has only local solutions for general initial data $\xi_0\in H$, hence we need the cut-off technique. For $R>0$, let $f_R:\R_+ \to [0,1]$ be a smooth non-increasing function taking the value 1 on $[0,R]$ and 0 on $[R+1, \infty)$. Fix a parameter $\delta\in (0,1/2)$. We consider
  \begin{equation}\label{SNSE-cut-off}
  \d \xi + f_R(\|\xi \|_{-\delta})\L_u \xi\,\d t =  \big[ \Delta \xi + S_\theta(\xi)\big] \,\d t+ \frac{C_\nu}{\|\theta\|_{\ell^2}} \sum_{k,\alpha} \theta_k \Pi(\sigma_{k,\alpha}\cdot \nabla \xi) \, \d W^{k,\alpha}_t ,
  \end{equation}
where $\|\cdot \|_{s} $ is the norm of the Sobolev space $H^{s}(\T^3,\R^3),\, s\in \R$. Note that, although we are concerned with $H$-valued solutions, here, due to technical reasons, we use a cut-off on negative Sobolev norms. Thanks to the cut-off, we can prove the global existence of pathwise unique strong solution to \eqref{SNSE-cut-off}.

\begin{theorem} \label{thm-existence}
Assume $\xi_0\in H$, $T>0$ and $\theta\in \ell^2$ verifies the symmetry property \eqref{theta-sym}, then there exists a pathwise unique strong solution to \eqref{SNSE-cut-off} on the interval $[0,T]$. More precisely, given a filtered probability space $(\Omega, \mathcal F, (\mathcal F_t), \P)$ and a family of independent $(\mathcal F_t)$-complex Brownian motions $\{W^{k, \alpha}: k\in \Z^3_0, \alpha=1,2\}$ defined on $\Omega$, there exists a pathwise unique $(\mathcal F_t)$-progressively measurable process $\xi$ with trajectories of class $L^\infty(0,T; H) \cap L^2(0,T;V)$ and in $ C\big([0,T], H^{-\delta} \big)$, such that for any divergence free test vector field $v \in C^\infty(\T^3,\R^3)$, one has $\P$-a.s. for all $t\in [0,T]$,
  $$\aligned
  \<\xi_t, v\>_{L^2} &= \<\xi_0, v\>_{L^2} + \int_0^t f_R\big(\| \xi_s \|_{-\delta} \big) \big\< \xi_s, \L_{u_s}^\ast v\big\>_{L^2} \,\d s + \int_0^t \<\xi_s,\Delta v \>_{L^2} \,\d s \\
  &\quad + \int_0^t \<\xi_s, S_\theta(v) \>_{L^2} \,\d s - \frac{C_\nu}{\|\theta\|_{\ell^2}} \sum_{k, \alpha} \theta_k \int_0^t \big\< \xi_s,  \sigma_{k,\alpha} \cdot\nabla v\big\>_{L^2}\, \d W^{k,\alpha}_s.
  \endaligned $$
Moreover, there is a constant $C_{\|\xi_0\|_{L^2}, \delta, R,T}>0$, independent of $\nu>0$ and $\theta\in \ell^2$, such that
  \begin{equation}\label{solution-property}
  \P \mbox{-a.s.}, \quad \|\xi \|_{L^\infty(0,T; H)} \vee \|\xi \|_{L^2(0,T; V)} \leq C_{\|\xi_0\|_{L^2}, \delta, R,T}.
  \end{equation}
\end{theorem}

We shall first prove the existence of global weak solutions to \eqref{SNSE-cut-off}, and then prove that \eqref{SNSE-cut-off} enjoys the pathwise uniqueness property, which, together with Yamada-Watanabe type argument \cite[Theorem 3.14]{Kurtz}, gives us the desired assertion, see Section \ref{sect-existence}. By stopping the solution given by this theorem at the random time $\tau_{R}=\inf \{  t\geq0:\left\Vert \xi_{t}\right\Vert _{-\delta}\geq R \}  $ (equal to $+\infty$ if the set is empty) we get a local solution of the original equation \eqref{SNSE-vort-Ito} without cut-off.

Next, we take a special sequence $\{\theta^N\}_{N\geq 1} \subset \ell^2$ as follows: for some $\gamma>0$,
  \begin{equation}\label{theta-N-def}
  \theta^N_k = \frac1{|k|^\gamma} {\bf 1}_{\{N\leq |k|\leq 2N\}},\quad k\in \Z^3_0,\, N\geq 1.
  \end{equation}
One can take more general sequences $\{\theta^N\}_{N\geq 1}$, but we do not pursue such generality here, see Remark \ref{rem-generic-theta} for a short discussion. It is easy to show that
  \begin{equation}\label{theta-N}
  \lim_{N\to \infty} \frac{\|\theta^N \|_{\ell^\infty} }{\|\theta^N \|_{\ell^2}} =0.
  \end{equation}
Moreover, we shall prove in Theorem \ref{prop-extra-term} that for any smooth divergence free vector field $v$,
  \begin{equation}\label{limit-I-theta-N}
  \lim_{N\to \infty} S_{\theta^N}(v) = \frac35 \nu \Delta v
  \end{equation}
which is independent of $\gamma>0$. We fix $R_0>0$ and let $B_H(R_0)$ be the closed ball in $H$, centered at the origin with radius $R_0$. Consider the sequence of stochastic 3D Navier-Stokes equations with cut-off:
  \begin{equation}\label{SNSE-cut-off-N}
  \aligned \d \xi^N =& - f_R\big(\|\xi^N \|_{-\delta} \big)\L_{u^N} \xi^N\,\d t + \big[ \Delta \xi^N + S_{\theta^N} \big(\xi^N \big) \big] \,\d t \\
  & + \frac{C_\nu}{\|\theta^N \|_{\ell^2}} \sum_{k,\alpha} \theta^N_k \Pi\big(\sigma_{k,\alpha}\cdot \nabla \xi^N \big) \,\d W^{k,\alpha}_t
  \endaligned
  \end{equation}
with initial condition $ \xi^N|_{t=0}= \xi^N_0\in B_{H}(R_0)$. For every $N\geq 1$, Theorem \ref{thm-existence} implies that the equation \eqref{SNSE-cut-off-N} has a pathwise unique global solution $\xi^N$ with the property
  \begin{equation}\label{solution-property-1}
  \P \mbox{-a.s.}, \quad \big\|\xi^N \big\|_{L^\infty(0,T; H)} \vee \big\|\xi^N \big\|_{L^2(0,T; V)} \leq C_{R_0, \delta, R,T},
  \end{equation}
where $C_{R_0, \delta, R,T}$ is independent of $\nu>0$ and $N\in \N$.

We want to take limit $N\to \infty$ in the above equation. Thanks to \eqref{theta-N} and the bounds \eqref{solution-property-1} on the solutions $\xi^N$, one can show that the martingale part in \eqref{SNSE-cut-off-N} will vanish. Next, due to \eqref{limit-I-theta-N}, the viscosity coefficient in the limit equation will be $1+\frac35 \nu$. Now we can state our main result.

\begin{theorem}[Scaling limit] \label{main-thm}
Fix $R_0>0,\, T>0$ and assume that $\big\{\xi^N_0 \big\}_{N\geq 1} \subset B_H(R_0)$ converges weakly in $H$ to some $\xi_0$. Then there exist $\nu>0$ and $R>0$ big enough, such that the pathwise unique strong solution $\xi^N$ of \eqref{SNSE-cut-off-N} converges weakly to the global strong solution of the deterministic 3D Navier-Stokes equations
  \begin{equation}\label{determ-NSE}
  \partial_t \xi + \L_u \xi= \Big(1+\frac35 \nu\Big) \Delta\xi, \quad \xi|_{t=0} = \xi_0.
  \end{equation}
Moreover, let $\mathcal X=L^2(0,T; H) \cap C\big([0,T], H^{-\delta}\big)$ and $Q^N_{\xi_0}$ be the law of the solution $\xi^N$ to \eqref{SNSE-cut-off-N} with $\xi^N|_{t=0} =\xi_0 \in B_{H}(R_0)$, $N\in \N$; then for any $\eps>0$,
  $$\lim_{N\to \infty} \sup_{\xi_0\in B_H(R_0)} Q^N_{\xi_0} \Big(\varphi\in \mathcal X: \|\varphi -\xi_\cdot (\xi_0) \|_{L^2(0,T; H)} \vee \|\varphi -\xi_\cdot (\xi_0) \|_{C([0,T], H^{-\delta})} >\eps \Big)=0. $$
Here we write $\xi_\cdot (\xi_0)$ to emphasize that it is the unique solution of \eqref{determ-NSE} starting from $\xi_0$.
\end{theorem}

It is well known that, for any initial vorticity $\xi_0\in B_H(R_0)$ (equivalently, the velocity field $u_0$ belongs to some ball in $V$), if the viscosity $\nu$ is big enough, then the deterministic 3D Navier-Stokes equations \eqref{determ-NSE} have a unique global strong solution, with explicit estimate on the time evolution of the norm $\|\xi_t\|_{L^2}$, see for instance Lemma \ref{lem-determ-3DNSE} below. In the following we want to take advantage of this fact and derive some consequences on the stochastic approximating equations \eqref{SNSE-cut-off-N}. The idea is similar to \cite{FlaMah} but here it is based on the noise, it is a regularization by noise property, as opposed to \cite{FlaMah} where it is due to a deterministic mechanism of fast rotation, in spite of the presence of the noise.

We fix $R_0>0$, and $\nu,\, R$ big enough (see Corollary \ref{cor-determ-NSE} for estimates on their values). For any $\xi_0\in B_H(R_0)$, denote by $\xi= \xi_\cdot (\xi_0)$ the unique global solution to \eqref{determ-NSE}; we can assume
  \begin{equation}\label{sect-1-bound}
  \|\xi \|_{C([0,T], H^{-\delta})} \leq R-1.
  \end{equation}
Now we consider the approximating equations \eqref{SNSE-cut-off-N}, but with the same initial condition $\xi_0$ as in \eqref{determ-NSE}. Given $T>0$ and arbitrary small $\eps>0$, Theorem \ref{main-thm} implies that there exists $N_0= N_0(R_0, \nu, R, T, \eps)\in \N$ such that for all $N\geq N_0$, the pathwise unique strong solution $\xi^N$ of \eqref{SNSE-cut-off-N} satisfies
  \begin{equation}\label{sect-1-bound.1}
  \P \Big( \big\|\xi^N -\xi \big\|_{L^2(0,T; H)} \vee \big\|\xi^N -\xi \big\|_{C([0,T], H^{-\delta})} \leq \eps \Big) \geq 1- \eps.
  \end{equation}
Combining this with \eqref{sect-1-bound}, we deduce
  $$\P \Big( \big\|\xi^N \big\|_{C([0,T], H^{-\delta})} < R \Big) \geq 1- \eps.$$
Thus, if we define the stopping time $\tau^N_R= \inf\big\{t>0: \big\|\xi^N_t \big\|_{-\delta} >R\big\}$ ($\inf\emptyset =T$), then
  \begin{equation}\label{stopping-time}
   \P \big( \tau^N_R \geq T \big) \geq 1- \eps.
   \end{equation}
For any $t\leq \tau^N_R$, it holds $f_R\big(\|\xi^N_t \|_{-\delta} \big) =1$, thus $\big(\xi^N_t \big)_{t\leq \tau^N_R}$ is a solution to the following equation without cut-off:
  \begin{equation*}
  \d \xi^N + \L_{u^N} \xi^N\,\d t = \big[ \Delta \xi^N + S_{\theta^N} \big(\xi^N \big) \big]\,\d t + \frac{C_\nu}{\|\theta^N \|_{\ell^2}} \sum_{k,\alpha} \theta^N_k \Pi\big( \sigma_{k,\alpha}\cdot \nabla \xi^N \big) \,\d W^{k,\alpha}_t.
  \end{equation*}
Therefore we have proved

\begin{corollary} \label{thm-uniqueness}
Corresponding to $R_{0}>0$, $T>0$ and $\eps>0$, choose $\nu>0$ and $R>0$ satisfying \eqref{sect-1-bound}, and choose $N_{0}$ as above. Then for all $N>N_{0}$, for any initial value $\xi_0\in B_H(R_0)$, the equation
  \begin{equation}\label{SNSE-N}
  \d \xi + \L_{u} \xi\,\d t = \big[ \Delta \xi + S_{\theta^N}(\xi) \big]\,\d t + \frac{C_\nu}{\|\theta^N \|_{\ell^2}} \sum_{k,\alpha} \theta^N_k \Pi( \sigma_{k,\alpha}\cdot \nabla \xi ) \,\d W^{k,\alpha}_t
  \end{equation}
admits a unique strong solution up to time $T$ with a probability no less than $1- \eps$.
\end{corollary}

Recall that, by Proposition \ref{1-prop} and the subsequent arguments, equation \eqref{SNSE-N} is equivalent to equation \eqref{SNSE-vort} with $\theta =\theta^N$, which has unitary viscosity and Stratonovich noise. Hence this corollary proves well posedness on a large time interval for large initial conditions but only unitary viscosity; the result is a form of regularization by noise.

A natural question is whether one can extend further the life time of the strong solution. Recall that, for any $\xi_0\in B_H(R_0)$, the $L^2$-norm of the unique solution to the deterministic 3D Navier-Stokes equations \eqref{determ-NSE} decreases exponentially fast. Combining this fact with \eqref{sect-1-bound.1}, we will show that the life time of the pathwise unique solution can actually be extended to $\infty$, with large probability.

\begin{theorem}[Long term well posedness] \label{main-result-thm}
Given $R_0>0$, take $\nu\geq \frac53\big[C_0R_0/(2\pi^2)^{1/4} -1 \big]$ and $ R>0$ big enough, where $C_0$ is a constant coming from some Sobolev embedding inequality. Then for any $0< \eps \leq (2\pi^2)^{1/4} / (2C_0)$, for all $T>1$ such that
  $$2R_0\, e^{-2\pi^2 \nu_1(T-1)} \leq \eps,$$
where $\nu_1= 1+\frac35 \nu$, there exists $N_0 = N_0(R_0,\nu, R,\eps, T)$ chosen as above such that for all $N\geq N_0$, for all $\xi_0\in B_H(R_0)$, the equation \eqref{SNSE-N} has a pathwise unique solution with infinite life time with probability greater than $1-\eps$.
\end{theorem}

This paper is organized as follows. In the next section we give explicit definitions of the vector fields $\{\sigma_{k,\alpha} \}_{k,\alpha}$ used above, and prove the key identity \eqref{proof.0}; a heuristic proof of \eqref{limit-I-theta-N} is provided in a special case, in order to facilitate the reader's understanding. Then we prove Theorem \ref{thm-existence} in Section \ref{sect-existence}, i.e. the global existence of pathwise unique solution to the equation \eqref{SNSE-cut-off} with cut-off. Section 4 contains the proofs of Theorems \ref{main-thm} and \ref{main-result-thm}, while Section 5 is devoted to the proof of the limit \eqref{limit-I-theta-N}. We provide in Section 6 a discussion of the reason why we cannot deal with the advection noise, and some heuristic arguments in Section \ref{appendix-3} with an attempt to justify the pure transport noise.

\section{Notations, proof of \eqref{proof.0} and heuristic discussions}\label{sect-prelim}

In this section we first define the vector fields $\{\sigma_{k,\alpha} \}_{k,\alpha}$ appeared in the last section and prove the identity \eqref{proof.0}. Similar results hold also in high dimensions, see e.g. \cite[Section 2]{Galeati}. Then we provide some heuristic discussions on the noise used in this paper, as well as a preliminary justification of the limit \eqref{limit-I-theta-N}.

Recall that $\Z^3_0= \Z^3\setminus \{0\}$ is the nonzero lattice points. Let $\Z^3_0= \Z^3_+ \cup \Z^3_-$ be a partition of $\Z^3_0$ such that
  $$\Z^3_+ \cap \Z^3_-= \emptyset, \quad \Z^3_+ = - \Z^3_-.$$
Let $L^2_0(\T^3, \mathbb C)$ be the space of complex valued square integrable functions on $\T^3$ with zero average. It has the CONS:
  $$e_k(x)= e^{2\pi {\rm i} k\cdot x}, \quad x\in \T^3,\, k\in \Z^3_0,$$
where ${\rm i}$ is the imaginary unit. For any $k\in \Z^3_+$, let $\{a_{k,1}, a_{k,2}\}$ be an orthonormal basis of $k^\perp := \{x\in \R^3: k\cdot x=0\}$ such that $\{a_{k,1}, a_{k,2}, \frac{k}{|k|}\}$ is right-handed. The choice of $\{a_{k,1}, a_{k,2}\}$ is not unique. For $k\in \Z^3_-$, we define $a_{k,\alpha} = a_{-k,\alpha}$, $\alpha=1,2$. Now we can define the divergence free vector fields:
  \begin{equation}\label{vector-fields}
  \sigma_{k,\alpha}(x) = a_{k,\alpha} e_k(x), \quad x\in \T^3,\,  k\in \Z^3_0,\, \alpha=1,2.
  \end{equation}
Then $\{\sigma_{k,1}, \sigma_{k,2}:  k\in \Z^3_0\}$ is a CONS of the subspace $H_{\mathbb C} \subset L^2_0(\T^3,\mathbb C^3)$ of square integrable and divergence free vector fields with zero mean. A vector field
  $$v= \sum_{k,\alpha} v_{k,\alpha} \sigma_{k,\alpha} \in H_{\mathbb C}$$
has real components if and only if $\overline{v_{k,\alpha}} = v_{-k, \alpha}$.

Next we introduce the family $\{W^{k,\alpha}: k\in \Z^3_0, \alpha=1,2\}$ of complex Brownian motions. Let
  $$\big\{B^{k,\alpha}: k\in \Z^3_0,\, \alpha=1,2 \big\}$$
be a family of independent standard real Brownian motions; then the complex Brownian motions can be defined as
  $$W^{k,\alpha} = \begin{cases}
  B^{k,\alpha} + {\rm i} B^{-k,\alpha}, & k\in  \Z^3_+;\\
  B^{-k,\alpha} - {\rm i} B^{k,\alpha}, & k\in  \Z^3_-.
  \end{cases}$$
Note that $\overline{W^{k,\alpha}}= W^{-k,\alpha}\, (k\in\Z^3_0, \alpha=1,2)$, and they have the following quadratic covariation:
  \begin{equation}\label{qudratic-var}
  \big[W^{k,\alpha}, W^{l,\beta} \big]_t= 2\, t \delta_{k,-l} \delta_{\alpha, \beta},\quad k,l\in \Z^3_0,\, \alpha, \beta\in \{1,2\} .
  \end{equation}

Take a $\theta\in \ell^2$ which is radially symmetric, namely, it satisfies \eqref{theta-sym}. We want to prove the following key equality:
  \begin{equation}\label{proof.0}
  \sum_{k,\alpha} \theta_k^2 (\sigma_{k,\alpha} \otimes \sigma_{-k,\alpha}) = \frac23 \|\theta \|_{\ell^2}^2 I_3,
  \end{equation}
where $I_3$ is the identity matrix of order 3. First, one has
  $$\aligned
  \sum_{k,\alpha} \theta_k^2 (\sigma_{k,\alpha} \otimes \sigma_{-k,\alpha}) &= \sum_{k,\alpha} \theta_k^2 (a_{k,\alpha}\otimes a_{k,\alpha}) = \sum_{k} \theta_k^2 (a_{k,1}\otimes a_{k,1} + a_{k,2}\otimes a_{k,2}) \\
  &= \sum_{k} \theta_k^2 \bigg(I_3 - \frac{k\otimes k}{|k|^2} \bigg),
  \endaligned $$
where in the last step we have used the fact that $\big\{\frac{k}{|k|}, a_{k,1}, a_{k,2} \big\}$ is an ONS of $\R^3$. It remains to compute the last series. Fix any $i, j\in \{1,2,3\}$. If $i \neq j$, without loss of generality, suppose $i=1$ and $j=2$, then
  $$\sum_{k,\alpha} \theta_k^2 \sigma_{k,\alpha}^1 \sigma_{-k,\alpha}^2 = - \sum_{k\in \Z^3_0}\theta_k^2 \frac{ k_1 k_2}{|k|^2} =0,$$
due to the symmetry property \eqref{theta-sym} and the fact that the sum involving the four points $(k_1, k_2, k_3)$, $(k_1, -k_2, k_3)$, $(-k_1, k_2, k_3)$, $(-k_1, -k_2, k_3)$ cancel. If $i = j$, then
  $$\sum_{k,\alpha} \theta_k^2 \sigma_{k,\alpha}^i \sigma_{-k,\alpha}^i = \sum_{k\in \Z^3_0} \theta_k^2 \bigg( 1- \frac{ k_i^2}{|k|^2} \bigg) = \sum_{k\in \Z^3_0}\theta_k^2 \frac{|k|^2 - k_i^2 }{|k|^2}.$$
Next, using the mapping $\psi:\Z^3_0 \to \Z^3_0$, $(k_1, k_2, k_3) \mapsto ( k_2, k_1, k_3)$, one can show that
  $$\sum_{k\in \Z^3_0}\theta_k^2 \frac{k_1^2 + k_3^2 }{|k|^2} = \sum_{k\in \Z^3_0}\theta_k^2 \frac{k_2^2 + k_3^2 }{|k|^2}. $$
In the same way,
  $$\sum_{k\in \Z^3_0}\theta_k^2 \frac{k_1^2 + k_3^2 }{|k|^2} =\sum_{k\in \Z^3_0}\theta_k^2 \frac{k_1^2 + k_2^2 }{|k|^2} =\sum_{k\in \Z^3_0}\theta_k^2 \frac{k_2^2 + k_3^2 }{|k|^2} , $$
and thus each of them is equal to
  $$\frac13  \sum_{k\in \Z^3_0} \theta_k^2 \frac{2\big( k_1^2 + k_2^2 + k_3^2 \big)}{|k|^2}= \frac23 \|\theta \|_{\ell^2}^2.$$
The proof of \eqref{proof.0} is complete.

Next, we note that $\sigma_{k,\alpha}\cdot \nabla \sigma_{-k,\alpha} \equiv -2\pi {\rm i}\, (a_{k,\alpha}\cdot k)a_{k,\alpha} = 0$ and $C_\nu=\sqrt{3\nu /2}$, thus for any smooth divergence free vector field $\xi$, the equality \eqref{proof.0} implies
  \begin{equation}\label{key-identity}
  \frac{C_\nu^2}{\|\theta\|_{\ell^2}^2} \sum_{k, \alpha} \theta_k^2\, \Pi\big[ \sigma_{k,\alpha}\cdot \nabla (\sigma_{-k,\alpha}\cdot\nabla \xi) \big] = \Pi[\nu \Delta \xi]= \nu \Delta \xi.
  \end{equation}

It may be helpful for the reader to rewrite some of the previous concepts and formulae
with different notations. The space-dependent vector valued Brownian motion
used here is%
\[
\eta( t,x)  :=\frac{C_{\nu}}{\left\Vert \theta\right\Vert _{\ell^{2}}%
}\sum_{k,\alpha}\theta_{k}\sigma_{k,\alpha} ( x)  W_{t}^{k,\alpha}%
\]
and its incremental covariance matrix-function is
\[
Q(x,y)  := \frac12 \mathbb{E}\left[  \eta( 1,x)  \otimes \eta(1,y)  \right]  =\frac{C_{\nu}^{2}}{\left\Vert \theta\right\Vert _{\ell^{2}%
}^{2}}\sum_{k,\alpha}\theta_{k}^{2}\, \sigma_{k,\alpha}(x) \otimes\sigma_{-k,\alpha}(y) =Q(x-y),
\]
where we add the coefficient $1/2$ due to \eqref{qudratic-var}, and
\[
Q(z)  :=\frac{C_{\nu}^{2}}{\left\Vert \theta\right\Vert _{\ell^{2}}^{2}} \sum_{k,\alpha}\theta_{k}^{2}\, a_{k,\alpha}\otimes a_{k,\alpha}\, e^{2\pi {\rm i}k\cdot z}
\]
(the random field $W\left(  t,x\right)  $ is invariant by translations in
$x$). The above computations, in particular (2.3), give us%
\[
Q(  0)  =\nu I_{3}.
\]

Moreover, recall the Stratonovich-It\^{o} corrector
\[
S_\theta(\xi) =\frac{C_{\nu}^{2}}{\left\Vert \theta\right\Vert _{\ell^{2}}^{2}%
}\sum_{k,\alpha} \theta_{k}^{2} \, \Pi\left[  \sigma_{k,\alpha}\cdot\nabla
\Pi\left(  \sigma_{-k,\alpha}\cdot\nabla\xi\right)  \right]
\]
and introduce the associated quadratic form%
\begin{align*}
a\left(  \xi,v\right)    & :=-\left\langle S_\theta(\xi), v\right\rangle =\frac{C_{\nu}^{2}}{\left\Vert \theta\right\Vert _{\ell^{2}}^{2}} \sum_{k,\alpha}\theta_{k}^{2}\left\langle \Pi (  \sigma_{-k,\alpha}\cdot \nabla\xi )  ,\sigma_{k,\alpha}\cdot\nabla v\right\rangle_{L^2}
\end{align*}
on divergence free smooth vector fields $\xi,v$ (we have used $\Pi^{\ast}=\Pi
$, $\Pi v=v$ and $\operatorname{div}\sigma_{k,\alpha}=0$ in the integration by
parts). Let $\Pi^\perp$ be the projection operator  orthogonal to $\Pi$: for any vector field $v\in L^2(\T^3, \R^3)$, $\Pi^\perp v$ is the gradient part of $v$. Then
  \begin{equation}\label{decompositions}
  \Pi(\sigma_{-k,\alpha}\cdot \nabla \xi) = \sigma_{-k,\alpha}\cdot\nabla \xi - \Pi^\perp (\sigma_{-k,\alpha}\cdot\nabla \xi).
  \end{equation}
Therefore, we get%
\[
a\left(  \xi,v\right)  =a_{0}\left(  \xi,v\right)  -a_{1}\left(  \xi,v\right),
\]
where
\begin{align*}
a_{0}\left(  \xi,v\right)    & :=\frac{C_{\nu}^{2}}{\left\Vert \theta
\right\Vert _{\ell^{2}}^{2}}\sum_{k,\alpha}\theta_{k}^{2}\left\langle
\sigma_{-k,\alpha}\cdot\nabla\xi,\sigma_{k,\alpha}\cdot\nabla v\right\rangle_{L^2},
\\
a_{1}\left(\xi,v\right)    & :=\frac{C_{\nu}^{2}}{\left\Vert \theta
\right\Vert _{\ell^{2}}^{2}}\sum_{k,\alpha}\theta_{k}^{2} \big\langle \Pi^{\perp
}\left(  \sigma_{-k,\alpha}\cdot\nabla\xi\right)  ,\Pi^{\perp} (
\sigma_{k,\alpha}\cdot\nabla v)  \big\rangle_{L^2} .
\end{align*}
We have, with the definitions above,
\[
a_{0}\left(  \xi,v\right)  =\left\langle Q( 0)  \,\nabla
\xi,\nabla v\right\rangle_{L^2} =-\nu\left\langle \Delta\xi,v\right\rangle_{L^2}
\]
which explains \eqref{key-identity} and clarifies that its structure is quite general.

Finally, for a particular choice of $\xi$ and $v$, we show in a heuristic way that
$a_{1} (\xi,v)$ converges to $-\frac{2}{5}\nu \langle \Delta\xi,v \rangle_{L^2} $ in the special scaling limit considered in the last subsection. This is to help the reader with understanding the limit \eqref{limit-I-theta-N}, since the rigorous proof of the general case is quite long, see Section \ref{sect-5}. To this end, we introduce the new operator
  \begin{equation}\label{new-operator}
  S_\theta^\perp (\xi)= \frac{C_{\nu}^{2}}{\left\Vert \theta
  \right\Vert _{\ell^{2}}^{2}}\sum_{k,\alpha}\theta_{k}^{2}\, \Pi\big[ \sigma_{k,\alpha}\cdot\nabla \Pi^{\perp} (\sigma_{-k,\alpha}\cdot\nabla\xi ) \big];
  \end{equation}
then $a_{1} (\xi,v)= -\big\< S_\theta^\perp(\xi), v \big\>_{L^2}$ and, by \eqref{key-identity} and \eqref{decompositions}, it is clear that
  \begin{equation} \label{decomp}
  S_\theta(v)= \nu \Delta v -S_\theta^\perp(v).
  \end{equation}
We fix $l\in \Z^3_0$ and take complex vector fields
  $$\xi = v = \sigma_{l,1} + \sigma_{l,2} = (a_{l,1} + a_{l,2}) e_l. $$
Recall the sequence $\theta^N$ defined in \eqref{theta-N-def}; by Corollary \ref{cor-extra}, we have
  $$\aligned
  S_{\theta^N}^\perp (v) &= -\frac{6\pi^2 \nu}{\|\theta^N \|_{\ell^2}^2} \sum_{\beta=1}^2 |l|^2 \Pi\bigg\{ \bigg[ \sum_{k} \big(\theta^N_k \big)^2 \sin^2(\angle_{k,l}) (a_{l,\beta}\cdot (k-l)) \frac{k-l}{|k-l|^2} \bigg] e_l \bigg\} \\
  &\sim -\frac{6\pi^2 \nu}{\|\theta^N \|_{\ell^2}^2} \sum_{\beta=1}^2 |l|^2 \Pi\bigg\{ \bigg[ \sum_{k} \big(\theta^N_k \big)^2 \sin^2(\angle_{k,l}) (a_{l,\beta}\cdot k) \frac{k}{|k|^2} \bigg] e_l \bigg\},
  \endaligned$$
where $\angle_{k,l}$ is the angle between the vectors $k$ and $l$, and $\sim$ means the difference between the two quantities vanishes as $N\to \infty$. The complex conjugate $\bar v$ of $v$ is divergence free, hence
  $$\aligned
  \big\< S_{\theta^N}^\perp (v), \bar v\big\>_{L^2} &\sim -\frac{6\pi^2 \nu}{\|\theta^N \|_{\ell^2}^2} \sum_{\beta=1}^2 |l|^2 \bigg\< \bigg[ \sum_{k} \big(\theta^N_k \big)^2 \sin^2(\angle_{k,l}) (a_{l,\beta}\cdot k) \frac{k}{|k|^2} \bigg] e_l, (a_{l,1} + a_{l,2}) e_{-l} \bigg\>_{L^2} \\
  &= -\frac{6\pi^2 \nu}{\|\theta^N \|_{\ell^2}^2} |l|^2 \sum_{\beta, \beta' =1}^2 \sum_{k} \big(\theta^N_k \big)^2 \sin^2(\angle_{k,l}) \frac{(a_{l,\beta}\cdot k) (a_{l,\beta'}\cdot k)}{|k|^2} .
  \endaligned $$
Recall that $\{a_{l,1}, a_{l,2} , \frac{l}{|l|} \}$ is an ONS of $\R^3$. By symmetry, the terms with $\beta \neq \beta'$ vanish, thus
  $$\aligned
  \big\< S_{\theta^N}^\perp (v), \bar v\big\>_{L^2} &\sim -\frac{6\pi^2 \nu}{\|\theta^N \|_{\ell^2}^2} |l|^2 \sum_{\beta =1}^2 \sum_{k} \big(\theta^N_k \big)^2 \sin^2(\angle_{k,l}) \frac{(a_{l,\beta}\cdot k)^2}{|k|^2}\\
  &= -\frac{6\pi^2 \nu}{\|\theta^N \|_{\ell^2}^2} |l|^2 \sum_{k} \big(\theta^N_k \big)^2 \sin^4(\angle_{k,l}),
  \endaligned $$
where we have used
  $$\sum_{\beta =1}^2 \frac{(a_{l,\beta}\cdot k)^2}{|k|^2} = 1- \bigg(\frac{k}{|k|} \cdot \frac{l}{|l|}\bigg)^2 = \sin^2(\angle_{k,l}). $$
Now, approximating the sums by integrals and changing to spherical variables yield
  $$\aligned \frac1{\|\theta^N \|_{\ell^2}^2} \sum_{k} \big(\theta^N_k \big)^2 \sin^4(\angle_{k,l}) &\sim \frac{\int_{\{ N\leq |x|\leq 2N\}} \frac{\sin^4(\angle_{x,l})}{|x|^{2\gamma}}\,\d x}{\int_{\{N\leq |x|\leq 2N\}} \frac{1}{|x|^{2\gamma}}\,\d x} = \frac{\int_N^{2N} \frac{\d r}{r^{2\gamma-2}} \int_0^\pi \sin^5 \psi\,\d\psi \int_0^{2\pi} \d\varphi}{\int_N^{2N} \frac{\d r}{r^{2\gamma-2}} \int_0^\pi \sin \psi\,\d\psi \int_0^{2\pi} \d\varphi} \\
  &= \frac12 \int_0^\pi \sin^5 \psi\,\d\psi = \frac8{15}.
  \endaligned $$
Thus, as $N\to \infty$,
  $$\big\< S_{\theta^N}^\perp (v), \bar v \big\>_{L^2} \to - 6\pi^2 \nu |l|^2 \cdot \frac8{15} = - \frac{16}5 \pi^2 \nu |l|^2 = \frac25 \nu \<\Delta v, \bar v\>_{L^2}, $$
since $\Delta v= -4\pi^2|l|^2 v = -4\pi^2|l|^2 ( \sigma_{l,1} + \sigma_{l,2})$.

\section{Existence of pathwise unique global solution to \eqref{SNSE-cut-off}} \label{sect-existence}

In this section we fix $\nu>0$, $R>0$, $\delta\in (0,1/2)$ and $\theta\in \ell^2$ satisfying \eqref{theta-sym}. Consider the equation \eqref{SNSE-cut-off} that we recall here:
  \begin{equation}\label{SNSE-cut-off-1}
  \d \xi + f_R(\|\xi \|_{-\delta})\L_u \xi\,\d t = \big[ \Delta \xi + S_\theta(\xi) \big]\,\d t + \frac{C_\nu}{\|\theta\|_{\ell^2}} \sum_{k,\alpha} \theta_k \Pi(\sigma_{k,\alpha}\cdot \nabla \xi) \, \d W^{k,\alpha}_t.
  \end{equation}
Our purpose is to show the global existence of pathwise unique solutions to the above equation.

We consider the Galerkin approximations of the equation \eqref{SNSE-cut-off-1}. Recall that $H$ is the subspace of $L^2_0(\T^3, \R^3)$ consisting of divergence free vector fields with zero average, and $V$ the intersection of $H$ with the first order Sobolev space $H^1(\T^3, \R^3)$. The norms in $H$ and $V$ are $\|\cdot \|_{H} = \|\cdot \|_{L^2}$ and $\|\cdot \|_V$, respectively. For $N\geq 1$, let
  $$H_N= {\rm span}\{\sigma_{k,\alpha}: |k|\leq N, \alpha =1,2\}$$
be a finite dimensional subspace of $H$ and $\Pi_N: H\to H_N$ the orthogonal projection. We define
  $$b_N(\xi_N)= f_R\big(\|\xi_N \|_{-\delta} \big)\, \Pi_N \big(\L_{u_N} \xi_N \big), \quad \xi_N\in H_N,$$
where $u_N= B(\xi_N)$ and $B$ is the Biot-Savart kernel on $\T^3$. Next,  by \eqref{decomp} and the expression \eqref{lem-extra-term.1} of $S_\theta^\perp(v)$, we see that if $v\in H_N$, then $S_\theta(v) \in H_N$. Thus we consider the following SDEs on $H_N$:
  \begin{equation}\label{SNSE-finite-dim}
  \left\{ \aligned
  \d \xi_N &= \big[- b_N (\xi_N ) + \Delta\xi_N + S_\theta(\xi_N)\big]\,\d t + \frac{C_\nu}{\|\theta\|_{\ell^2}} \sum_{k, \alpha} \theta_k \Pi_N \big(\sigma_{k,\alpha} \cdot\nabla \xi_N \big) \, \d W^{k,\alpha}_t,\\
  \xi_N(0) &= \Pi_N \xi_0\in H_N.
  \endaligned \right.
  \end{equation}
We have the following simple result.

\begin{lemma}\label{lem-apriori}
$\P$-a.s., for all $t>0$,
  \begin{equation}\label{lem-apriori.1}
  \|\xi_N(t) \|_{L^2}^2 + \int_0^t \|\nabla \xi_N(s)\|_{L^2}^2 \,\d s \leq \|\xi_0 \|_{L^2}^2 + C_{\delta, R}\, t,
  \end{equation}
where $C_{\delta, R} >0$ is a constant depending on $\delta$ and $R$, but independent of $N,\, \nu$ and $\theta\in \ell^2$.
\end{lemma}

\begin{proof}
For simplicity of notation we omit the time variable $t$. By the It\^o formula,
  $$\aligned
  \d \|\xi_N \|_{L^2}^2 &= -2\<\xi_N, b_N(\xi_N)\>_{L^2}\,\d t + 2\<\xi_N, \Delta\xi_N\>_{L^2} \,\d t+ 2\<\xi_N, S_\theta(\xi_N)\>_{L^2}\,\d t\\
  &\quad + \frac{2 C_\nu}{\|\theta\|_{\ell^2}} \sum_{k,\alpha} \theta_k \big\<\xi_N, \Pi_N( \sigma_{k,\alpha} \cdot\nabla \xi_N )\big\>_{L^2} \, \d W^{k,\alpha}_t + \frac{2C_\nu^2}{\|\theta \|_{\ell^2}^2} \sum_{k,\alpha} \theta_k^2 \big\| \Pi_N( \sigma_{k,\alpha} \cdot\nabla \xi_N ) \big\|_{L^2}^2 \,\d t,
  \endaligned $$
where the quadratic variation term follows from \eqref{qudratic-var}. As $\sigma_{k,\alpha}$ is divergence free, we have
  $$\big\<\xi_N, \Pi_N( \sigma_{k,\alpha} \cdot\nabla \xi_N )\big\>_{L^2} = \big\<\xi_N, \sigma_{k,\alpha} \cdot\nabla \xi_N \big\>_{L^2}=0, $$
thus the martingale part vanishes. Since $C_\nu =\sqrt{3\nu/2}$ and $\Pi_N:H\to H_N$ is an orthogonal projection,
  $$\frac{2C_\nu^2}{\|\theta \|_{\ell^2}^2} \sum_{k,\alpha} \theta_k^2 \big\| \Pi_N( \sigma_{k,\alpha} \cdot\nabla \xi_N ) \big\|_{L^2}^2 \leq \frac{3\nu}{\|\theta \|_{\ell^2}^2} \sum_{k,\alpha} \theta_k^2 \|\Pi( \sigma_{k,\alpha} \cdot\nabla \xi_N) \|_{L^2}^2. $$
Moreover, by the definition \eqref{I-theta} of $S_\theta$ and integration by parts,
  $$\aligned
  2\<\xi_N, S_\theta(\xi_N)\>_{L^2} &= -\frac{2C_\nu^2}{\|\theta \|_{\ell^2}^2} \sum_{k,\alpha} \theta_k^2 \big\< \sigma_{k,\alpha} \cdot\nabla \xi_N, \Pi (\sigma_{-k,\alpha} \cdot\nabla \xi_N)\big\>_{L^2}\\
  &= -\frac{3\nu}{\|\theta \|_{\ell^2}^2} \sum_{k,\alpha} \theta_k^2 \|\Pi( \sigma_{k,\alpha} \cdot\nabla \xi_N) \|_{L^2}^2 .
  \endaligned $$
Summarizing the above discussions we obtain
  \begin{equation}\label{lem-apriori.2}
  \d \|\xi_N \|_{L^2}^2 \leq -2\<\xi_N, b_N(\xi_N)\>_{L^2}\,\d t - 2\|\nabla\xi_N\|_{L^2}^2 \,\d t.
  \end{equation}

Next, we deal with the more difficult nonlinear term:
  $$\<\xi_N, b_N(\xi_N)\>_{L^2}= f_R\big(\|\xi_N \|_{-\delta} \big) \big\<\xi_N, \L_{u_N} \xi_N \big\>_{L^2}= - f_R\big(\|\xi_N \|_{-\delta} \big) \big\<\xi_N, \xi_N\cdot \nabla u_N \big\>_{L^2},$$
since $\big\<\xi_N, u_N\cdot \nabla \xi_N \big\>_{L^2}=0$. H\"older's inequality leads to
  $$|\<\xi_N, b_N(\xi_N)\>_{L^2}| \leq f_R\big(\|\xi_N \|_{-\delta} \big) \|\xi_N \|_{L^3}^2 \|\nabla u_N\|_{L^3} \leq C f_R\big(\|\xi_N \|_{-\delta} \big) \|\xi_N \|_{L^3}^3. $$
Here and below, $C>0$ is a generic constant which may change from line to line. Using the Sobolev imbedding $H^{1/2}(\T^3)\hookrightarrow L^3(\T^3)$ we obtain
  $$\aligned
  |\<\xi_N, b_N(\xi_N)\>_{L^2}| &\leq C f_R\big(\|\xi_N \|_{-\delta} \big) \|\xi_N \|_{1/2}^3 .
  \endaligned $$
Next, we need the interpolation inequality
  $$\|\phi \|_{1/2} \leq \|\phi \|_{-\delta}^{1/2(1+\delta)} \|\phi \|_{1}^{(1+2\delta)/2(1+\delta)}, \quad \phi\in H^1(\T^3) ,$$
which gives us
  $$\aligned
  |\<\xi_N, b_N(\xi_N)\>_{L^2}| &\leq C f_R\big(\|\xi_N \|_{-\delta} \big) \|\xi_N \|_{-\delta}^{3/2(1+\delta)} \|\xi_N \|_{1}^{3(1+2\delta)/2(1+\delta)} \\
  &\leq C_\delta f_R\big(\|\xi_N \|_{-\delta} \big) \|\xi_N \|_{-\delta}^{3/2(1+\delta)} \|\nabla \xi_N \|_{L^2}^{3(1+2\delta)/2(1+\delta)} .
  \endaligned $$
Since $\delta<1/2$, it holds $3(1+2\delta)/2(1+\delta) <2$. Then, Young's inequality leads to
  \begin{equation}\label{estim-nonlinear}
  \aligned
  |\<\xi_N, b_N(\xi_N)\>_{L^2}| &\leq f_R\big(\|\xi_N \|_{-\delta} \big) \bigg(\frac12 \|\nabla\xi_N \|_{L^2}^{2} + C'_\delta \|\xi_N \|_{-\delta}^{6/(1-2\delta)} \bigg)\\
  &\leq \frac12 \|\nabla\xi_N \|_{L^2}^{2} + C'_\delta (R+1)^{6/(1-2\delta)}.
  \endaligned
  \end{equation}
where we have used the property of the cut-off function $f_R$ mentioned above \eqref{SNSE-cut-off}.

Finally, combining \eqref{lem-apriori.2} and \eqref{estim-nonlinear}, we get
  $$\d \|\xi_N \|_{L^2}^2 \leq - \|\nabla \xi_N\|_{L^2}^2 \,\d t + 2C'_\delta (R+1)^{6/(1-2\delta)}\,\d t,$$
which immediately yields the desired result.
\end{proof}

Lemma \ref{lem-apriori} implies that the sequence $\{\xi_N \}_{N\geq 1}$ is bounded both in $L^\infty \big(\Omega, L^\infty(0,T; H)\big)$ and in $L^2 \big(\Omega, L^2(0,T; V)\big)$. As a result, there exists a subsequence $\{\xi_{N_i} \}_{i\geq 1}$ which converge weakly-$\ast$ in $L^\infty \big(\Omega, L^\infty(0,T; H)\big)$ and weakly in $L^2 \big(\Omega, L^2(0,T; V)\big)$.

In order to pass to the limit in the nonlinear term, we need stronger convergence of $\{\xi_{N_i}\}_{i\geq 1}$. For this purpose, let $\eta_N$ be the law of $\xi_N$, $N\geq 1$; we shall prove that the family $\{\eta_N\}_{N\geq 1}$ is tight on
  $$L^2(0,T; H) \cap C\big([0,T], H^{-\delta} \big). $$
Indeed, with slightly more work, we can show the tightness in $C([0,T], H^{-} )$ where $H^{-}= \cap_{s<0} H^s$. We shall use Simon's compactness results in \cite{Simon} which involve the time fractional Sobolev spaces. For $\gamma\in (0,1)$, $p>1$ and a normed linear space $(Y, \|\cdot \|_Y)$, the fractional Sobolev space $W^{\gamma, p}(0,T; Y)$ consists of those functions $\varphi\in L^p(0,T; Y)$ such that
  $$\int_0^T\!\int_0^T \frac{\|\varphi(t)- \varphi(s)\|_{Y}^p}{|t-s|^{1+\gamma p}}\,\d t\d s <\infty. $$
In the following we take $Y= H^{-6}$, a choice which will become clear in view of the calculations in Corollary \ref{cor-tight}.

\begin{theorem}\label{thm-Simon}
\begin{itemize}
\item[\rm(i)] For any $\gamma\in (0, 1/2)$,
  $$L^2(0,T; V) \cap W^{\gamma,2} \big(0,T; H^{-6} \big) \subset L^2(0,T; H)$$
is a compact imbedding.
\item[\rm(ii)] If $p> 12(6 -\delta)/\delta$, then
  $$L^p(0,T; H) \cap W^{1/3,4} \big(0,T; H^{-6} \big) \subset C\big([0,T], H^{-\delta} \big) $$
is a compact imbedding.
\end{itemize}
\end{theorem}

\begin{proof}
The first assertion follows directly from \cite[p.86, Corollary 5]{Simon}. For assertion (ii), we shall apply \cite[p.90, Corollary 9]{Simon}. In this case, we use the interpolation inequality
  $$\|\varphi \|_{-\delta} \leq \|\varphi \|_{L^2}^{1-\kappa} \|\varphi \|_{-6}^\kappa, \quad \varphi\in H,$$
where $\kappa = \delta/ 6$, playing the role of the parameter $\theta$ in \cite[p.90, Corollary 9]{Simon}. We have $s_0=0, r_0=p$ and $s_1=1/3, r_1=4$, hence, $s_\kappa= (1-\kappa)s_0 + \kappa s_1 = \kappa/3$ and
  $$\frac1{r_\kappa} = \frac{1-\kappa}{r_0} + \frac{\kappa}{r_1} = \frac{1-\kappa}{p} + \frac{\kappa}{4}. $$
For $p$ given above, it is clear that $s_\kappa > 1/ r_\kappa$, thus we deduce the second assertion from Corollary 9 in \cite{Simon}.
\end{proof}

Recall that $\eta_N$ is the law of $\xi_N$, $N\geq 1$. We have the following immediate consequences.

\begin{corollary}\label{cor-Simon}
\begin{itemize}
\item[\rm(i)] If there is $C>0$ such that
  \begin{equation}\label{bound-estima}
  \E\int_0^T \|\xi_N(t) \|_{V}^2\,\d t + \E \int_0^T\!\int_0^T \frac{\|\xi_N(t)- \xi_N(s) \|_{-6}^2}{|t-s|^{1+2\gamma}} \,\d t\d s \leq C \quad \mbox{for all } N\in \N,
  \end{equation}
then $\{\eta_N \}_{N\in \N}$ is tight on $L^2(0,T; H)$.
\item[\rm(ii)] If there is $C>0$ such that
  \begin{equation}\label{bound-estima.2}
  \E\int_0^T \|\xi_N(t) \|_{L^2}^p \,\d t + \E \int_0^T\!\int_0^T \frac{\|\xi_N(t)- \xi_N(s) \|_{-6}^4 }{|t-s|^{7/3}} \,\d t\d s \leq C \quad \mbox{for all } N\in \N,
  \end{equation}
then $\{\eta_N \}_{N\in \N}$ is tight on $C\big([0,T], H^{-\delta} \big)$.
\end{itemize}
\end{corollary}

To apply these tightness criteria, by Lemma \ref{lem-apriori}, it is sufficient to estimate the expectations involving double time integrals. For this aim, we need the next estimate.

\begin{lemma}\label{lem-moment-estim}
There is a constant $C= C( \|\xi_0 \|_{L^2},\nu, \delta, R, T)>0$, independent of $N$ and $\theta$, such that for any $0\leq s<t\leq T$ and $l\in \Z_0^3,\, j =1,2$, one has
  $$\E \Big( \big|\big\<\xi_N(t)- \xi_N(s), \sigma_{l,j}\big\>_{L^2} \big|^4 \Big)\leq C|l|^8 |t-s|^2.$$
\end{lemma}

\begin{proof}
It is enough to consider $|l|\leq N$. Since $\xi_N$ satisfies the equation \eqref{SNSE-finite-dim}, we have
  $$\aligned
  &\ \big\<\xi_N(t)- \xi_N(s), \sigma_{l,j}\big\>_{L^2}\\
  = &\, - \int_s^t \big\<b_N(\xi_N(r)), \sigma_{l,j} \big\>_{L^2} \,\d r + \int_s^t \big\<\xi_N(r), \Delta \sigma_{l,j} \big\>_{L^2} \,\d r \\
  &\, +\int_s^t \big\<\xi_N(r), S_\theta( \sigma_{l,j})\big\>_{L^2} \,\d r - \frac{C_\nu}{\|\theta\|_{\ell^2}} \sum_{k,\alpha} \theta_k \int_s^t \big\<\xi_N(r), \sigma_{k,\alpha} \cdot\nabla \sigma_{l,j} \big\>_{L^2} \, \d W^{k,\alpha}_r.
  \endaligned $$
We omit the time variable $r$ in the sequel when there is no confusion. We have
  $$\aligned
  \big\<b_N(\xi_N), \sigma_{l,j} \big\>_{L^2} &=  f_R\big(\|\xi_N \|_{-\delta} \big) \Big( \big\< u_N\cdot\nabla \xi_N, \sigma_{l,j} \big\>_{L^2} -\big\< \xi_N\cdot\nabla u_N, \sigma_{l,j} \big\>_{L^2} \Big) \\
  &= -f_R\big(\|\xi_N \|_{-\delta} \big) \Big( \big\< \xi_N, u_N\cdot\nabla \sigma_{l,j} \big\>_{L^2} +\big\< \xi_N\cdot\nabla u_N, \sigma_{l,j} \big\>_{L^2} \Big).
  \endaligned $$
Lemma \ref{lem-apriori} and the fact $\nabla \sigma_{l,j} = 2\pi {\rm i}\, \sigma_{l,j}\otimes l$ imply
  $$ \big|\big\< \xi_N, u_N\cdot\nabla \sigma_{l,j} \big\>_{L^2} \big| \leq C |l|\, \|\xi_N \|_{L^2} \|u_N \|_{L^2} \leq C_{\xi_0, \delta, R, T} |l|,  $$
where $C_{\xi_0, \delta,R, T} = C( \|\xi_0 \|_{L^2}, \delta, R, T)$ is a constant. Similarly,
  $$\big|\big\< \xi_N\cdot\nabla u_N, \sigma_{l,j} \big\>_{L^2} \big|\leq C  \|\xi_N \|_{L^2} \|\nabla u_N \|_{L^2} \leq C  \|\xi_N \|_{L^2}^2 \leq C_{\xi_0,\delta, R, T}.  $$
Therefore,
  $$\E\, \bigg| \int_s^t \big\<b_N(\xi_N ), \sigma_{l,j} \big\>_{L^2} \,\d r \bigg|^4 \leq C_{\xi_0, \delta,R, T}|l|^4 |t-s|^4. $$
Next, since $\Delta \sigma_{l,j}= -4\pi^2|l|^2 \sigma_{l,j}$, we have
  $$\big|\big\<\xi_N, \Delta \sigma_{l,j} \big\>_{L^2} \big| \leq C\|\xi_N \|_{L^2} |l|^2 \leq C_{\xi_0,\delta, R, T} |l|^2.$$
As a consequence,
  $$\E\, \bigg|\int_s^t \big\<\xi_N, \Delta \sigma_{l,j} \big\>_{L^2} \,\d r \bigg|^4 \leq C_{\xi_0, \delta, R, T} |l|^8 |t-s|^4. $$
By \eqref{decomp} and \eqref{lem-extra-term.1}, we have
  $$S_\theta(\sigma_{l,j}) = \nu\Delta \sigma_{l,j} +\frac{6\pi^2 \nu}{\|\theta \|_{\ell^2}^2} \Pi\bigg\{ \bigg[ \sum_{k,\alpha} \theta_k^2 (a_{k,\alpha} \cdot l)^2 (a_{l,j}\cdot (k-l)) \frac{k-l}{|k-l|^2} \bigg] e_l \bigg\}. $$
Therefore,
  $$\aligned \|S_\theta(\sigma_{l,j}) \|_{L^2} &\leq 4\pi^2 \nu |l|^2+ \frac{6\pi^2 \nu}{\|\theta \|_{\ell^2}^2} \bigg| \sum_{k,\alpha} \theta_k^2 (a_{k,\alpha} \cdot l)^2 (a_{l,j}\cdot (k-l)) \frac{k-l}{|k-l|^2} \bigg| \\
  & \leq 4\pi^2 \nu |l|^2+ \frac{6\pi^2 \nu}{\|\theta \|_{\ell^2}^2} \sum_{k,\alpha} \theta_k^2 (a_{k,\alpha} \cdot l)^2 \leq 10\pi^2 \nu |l|^2 .
  \endaligned $$
This implies
  $$\E\, \bigg| \int_s^t \big\<\xi_N, S_\theta( \sigma_{l,j})\big\>_{L^2} \,\d r \bigg|^4 \leq \E\, \bigg( \int_s^t C|l|^2 \|\xi_N \|_{L^2}  \,\d r \bigg)^4 \leq C_{\xi_0,\nu,\delta, R, T} |l|^8 |t-s|^4.$$

Finally, by the Burkholder-Davis-Gundy inequality,
  $$\aligned
  &\, \E\, \bigg| \frac{C_\nu}{\|\theta\|_{\ell^2}} \sum_{k,\alpha} \theta_k \int_s^t \big\<\xi_N,\sigma_{k,\alpha} \cdot\nabla \sigma_{l,j}\big\>_{L^2} \, \d W^{k,\alpha}_r \bigg|^4 \\
  \leq&\,  \frac{C_\nu^4}{\|\theta \|_{\ell^2}^4} \E \bigg(\sum_{k,\alpha} \theta_k^2 \int_s^t \big| \big\<\xi_N, \sigma_{k,\alpha}\cdot\nabla \sigma_{l,j} \big\>_{L^2} \big|^2 \,\d r \bigg)^2 .
  \endaligned $$
We have
  $$ \aligned
  \sum_{k,\alpha} \theta_k^2 \big| \big\<\xi_N, \sigma_{k,\alpha}\cdot\nabla \sigma_{l,j} \big\>_{L^2} \big|^2 &\leq \sum_{k,\alpha} \theta_k^2 \|\xi_N \|_{L^2}^2 \|\sigma_{k,\alpha} \cdot\nabla \sigma_{l,j}\|_{L^2}^2  \leq C_{\xi_0,\delta, R, T} \|\theta\|_{\ell^2}^2 |l|^2.
  \endaligned $$
Substituting this estimate into the above inequality yields
  $$\E\, \bigg|\frac{C_\nu}{\|\theta\|_{\ell^2}} \sum_{k,\alpha} \theta_k \int_s^t \big\<\xi_N, \sigma_{k,\alpha}\cdot\nabla \sigma_{l,j} \big\>_{L^2} \, \d W^{k,\alpha}_r \bigg|^4 \leq C_{\xi_0,\nu, \delta, R, T} |l|^4 |t-s|^2.  $$
Summarizing the above estimates we complete the proof.
\end{proof}

As a consequence, we have

\begin{corollary}\label{cor-tight}
The family $\{\eta_N\}_{N\geq 1}$ is tight both on $L^2(0,T; H)$ and on $C\big([0,T], H^{-\delta} \big)$.
\end{corollary}

\begin{proof}
We first check the uniform boundedness of the second expectation in \eqref{bound-estima.2}. By Cauchy's inequality and Lemma \ref{lem-moment-estim},
  $$\aligned \E\big( \|\xi_N(t)- \xi_N(s)\|_{-6}^4 \big) & = \E \Bigg(\sum_{l,j} \frac{ \big| \<\xi_N(t)- \xi_N(s), \sigma_{l,j} \>_{L^2} \big|^2}{|l|^{12 }} \Bigg)^2 \\
  &\leq \bigg(\sum_{l,j} \frac{ 1}{|l|^{12 }} \bigg)\, \E \Bigg( \sum_{l,j} \frac{\big| \<\xi_N(t)- \xi_N(s), \sigma_{l,j} \>_{L^2} \big|^4} {|l|^{12 }} \Bigg) \\
  &\leq C \sum_{l,j} \frac{C |l|^8 |t-s|^2}{|l|^{12}} \leq C' |t-s|^2
  \endaligned $$
since $12- 8>3$. Therefore,
  $$\E \int_0^T\!\int_0^T \frac{\|\xi_N(t)- \xi_N(s)\|_{-6}^4}{|t-s|^{7/3}}\,\d t\d s \leq \int_0^T\!\int_0^T \frac{C' |t-s|^2}{|t-s|^{7/3}}\,\d t\d s <+\infty .$$
Thus we have proved the estimate \eqref{bound-estima.2}. Now we can apply (ii) in Corollary \ref{cor-Simon} to get the tightness of $\{\eta_N\}_{N\geq 1}$ on $C\big([0,T], H^{-\delta} \big)$.

In the same way, we can check the uniform boundedness of the second expectation in \eqref{bound-estima}, using the facts that $\gamma \in (0,1/2)$ and
  $$\E \big| \<\xi_N(t)- \xi_N(s), \sigma_{l,j} \>_{L^2} \big|^2 \leq \Big[ \E \big|\<\xi_N(t)- \xi_N(s), \sigma_{l,j} \>_{L^2} \big|^4 \Big]^{1/2} \leq C |l|^4 |t-s|. $$
The proof is complete.
\end{proof}

Based on the above results, we can apply the Prohorov theorem (see \cite[p.59, Theorem 5.1]{Billingsley}) to deduce that there exists a subsequence $\{\eta_{N_i} \}_{i\geq 1}$ which converge weakly to some probability measure $\eta$ supported on $L^2(0,T; H)$ and on $C\big([0,T], H^{-\delta} \big)$. Moreover, by Skorokhod's representation theorem (\cite[p.70, Theorem 6.7]{Billingsley}), there exist a new probability space $\big(\tilde \Omega, \tilde{\mathcal F}, \tilde\P\big)$ and a sequence of random variables $\big\{\tilde \xi_{N_i} \big\}_{i\geq 1}$ and $\tilde \xi$ defined on this space, such that $\tilde \xi_{N_i} \stackrel{d}{\sim} \eta_{N_i}\, (i\geq 1),\, \tilde \xi \stackrel{d}{\sim} \eta$, and
  \begin{equation}\label{limit-as}
  \tilde \P \mbox{-a.s.}, \quad \lim_{i\to \infty} \tilde \xi_{N_i} = \tilde \xi \quad \mbox{in the topologies of } L^2(0,T; H) \mbox{ and } C\big([0,T], H^{-\delta} \big).
  \end{equation}

\begin{remark}\label{rem-BM}
We can also consider $\{\xi_N \}_{N\geq 1}$ together with the family of complex Brownian motions $W:= \big\{W^{k,\alpha}: k\in \Z^3_0, \alpha =1,2 \big\}$ to get tightness of their joint laws. Here, for simplicity, we write $W$ for the whole family of Brownian motions. Namely, for each $i\in \N$, there exist a family $\tilde W^{N_i}:= \big\{\tilde W^{N_i, k,\alpha}:k\in \Z^3_0, \alpha=1,2 \big\}$ of independent complex Brownian motions defined on $\big(\tilde \Omega, \tilde{\mathcal F}, \tilde\P\big)$ such that
\begin{itemize}
\item[\rm(1)] for any $i\in \N$, $(\xi_{N_i}, W)$ and $\big(\tilde \xi_{N_i}, \tilde W^{N_i} \big)$ have the same joint law;
\item[\rm(2)] in addition to \eqref{limit-as}, we have, for all $k\in \Z^3_0$ and $\alpha \in \{1,2\}$, $\tilde W^{N_i, k,\alpha}$ converge $\tilde\P$-a.s. in $C([0,T], \mathbb C)$ to a complex Brownian motion $\tilde W^{k,\alpha}$.
\end{itemize}
Furthermore, the family $\tilde W:= \big\{\tilde W^{k,\alpha}:k\in \Z^3_0, \alpha=1,2 \big\}$ of Brownian motions are mutually independent. See for instance the discussions above (3.7) of \cite{FGL} for details. These additional facts will be useful in the proof of existence of weak solutions.
\end{remark}

The next bounds on the limit $\tilde\xi$ are important for us to prove the scaling limit in the next section, where we will need them to show tightness.

\begin{corollary}\label{cor-bound}
There exists a constant $C_{\|\xi_0\|_{L^2},\delta, R,T}$ independent of $\theta\in \ell^2$ such that $\tilde \P$-a.s.,
  \begin{equation}\label{cor-bound.0}
  \big\|\tilde \xi \big\|_{L^\infty(0,T;H)} \leq C_{\|\xi_0 \|_{L^2},\delta, R,T}
  \end{equation}
and
  \begin{equation}\label{cor-bound.1}
  \big\|\tilde \xi \big\|_{L^2(0,T;V)} \leq C_{\|\xi_0 \|_{L^2},\delta, R,T}.
  \end{equation}
\end{corollary}

\begin{proof}
Thanks to \eqref{lem-apriori.1}, the proof of the first assertion is similar to that of \cite[Lemma 3.4]{FGL}, hence we omit it here. Next we focus on \eqref{cor-bound.1}. For any $N\geq 1$, $\xi_N$ satisfies the bound \eqref{lem-apriori.1} almost surely. Hence, there is a constant $C_{\|\xi_0\|_{L^2},\delta, R,T} >0$, independent of $\theta\in \ell^2$, such that, $\P$-a.s.,
  $$\| \xi_N \|_{L^2(0,T;V)} \leq C_{\|\xi_0\|_{L^2},\delta, R,T} .$$
Since $\tilde \xi_{N_i}$ has the same law as $\xi_{N_i}$, thus it enjoys the same bound: $\tilde \P$-a.s.,
  \begin{equation*}
  \big\|\tilde \xi_{N_i} \big\|_{L^2(0,T;V)} \leq C_{\|\xi_0\|_{L^2},\delta, R,T} .
  \end{equation*}
Note that the bound is independent of $i\geq 1$. This implies that there is an event $\tilde\Omega_0 \subset \tilde\Omega$ of full probability such that for every $\tilde\omega \in \tilde\Omega_0$, one has
  \begin{equation} \label{cor-bound.2}
  \sup_{i\geq 1} \big\|\tilde \xi_{N_i}(\tilde\omega, \cdot) \big\|_{L^2(0,T;V)} \leq C_{\|\xi_0\|_{L^2},\delta, R,T} .
  \end{equation}
Therefore, up to a subsequence, $\tilde \xi_{N_i}(\tilde\omega, \cdot)$ converge weakly in $L^2(0,T;V)$ to some limit $\hat \xi(\tilde\omega, \cdot)$. This also means that $\tilde \xi_{N_i}(\tilde\omega, \cdot)$ converge weakly in $L^2(0,T;H)$ to $\hat \xi(\tilde\omega, \cdot)$. Combining this fact with \eqref{limit-as}, we conclude that $\hat \xi(\tilde\omega, \cdot) = \tilde \xi(\tilde\omega, \cdot)$. This holds for all $\tilde \omega\in \tilde\Omega_0$, the event of full probability. As a consequence, the limit process $\tilde \xi$ obtained above actually has trajectories in $L^2(0,T;V)$, and  by the property of weak convergence, one has, for any $\tilde\omega\in \tilde\Omega_0$,
  $$\big\| \tilde\xi(\tilde\omega, \cdot) \big\|_{L^2(0,T; V)} \leq \liminf_{i\to \infty} \big\| \tilde\xi_{N_i}(\tilde\omega, \cdot) \big\|_{L^2(0,T; V)} \leq C_{\|\xi_0\|_{L^2},\delta, R,T} .$$
This completes the proof.
\end{proof}

Now we can prove the existence of pathwise unique strong solutions to the stochastic NSEs \eqref{SNSE-cut-off-1} with cut-off.

\begin{proof}[Proof of Theorem \ref{thm-existence}]
The proof is quite long and is divided into two steps. In \emph{Step 1}, we prove that the limit process $\tilde\xi$ obtained above is a weak solution to the equation \eqref{SNSE-cut-off-1}, while in \emph{Step 2} we prove that the pathwise uniqueness holds for \eqref{SNSE-cut-off-1}, thus the existence of a unique strong solution follows from the Yamada-Watanabe type result \cite[Theorem 3.14]{Kurtz}.

\emph{Step 1: weak existence}. Let $v$ be any divergence free test vector field. Recall that, by Remark \ref{rem-BM}, we have the sequences of complex Brownian motions $\tilde W^{N_i}:= \big\{\tilde W^{N_i, k,\alpha}: k\in \Z^3_0, \alpha=1,2 \big\}$, such that for each $i\in \N$, $\big(\tilde \xi_{N_i}, \tilde W^{N_i}\big) $ has the same law as the pair $\big(\xi_{N_i}, W \big) $ defined on the original probability space $(\Omega, \mathcal F, \P)$, and the latter pair satisfies the equation \eqref{SNSE-finite-dim} with $N_i$ in place of $N$. Thus $\tilde \xi_{N_i}$ verifies the following stochastic integral equation:
  $$\aligned
  \big\<\tilde \xi_{N_i}(t), v\big\>_{L^2} &= \big\<\Pi_{N_i} \xi_0, v\big\>_{L^2} + \int_0^t f_R\big(\|\tilde \xi_{N_i}(s)\|_{-\delta} \big) \Big\<\tilde \xi_{N_i}(s), \L_{\tilde u_{N_i}(s)}^\ast (\Pi_{N_i} v) \Big\>_{L^2} \,\d s \\
  &\quad + \int_0^t \big\<\tilde \xi_{N_i}(s),\Delta v\big\>_{L^2} \,\d s + \int_0^t \big\<\tilde \xi_{N_i}(s), S_\theta(v) \big\>_{L^2} \,\d s \\
  &\quad - \frac{C_\nu}{\|\theta\|_{\ell^2}} \sum_{k, \alpha} \theta_k \int_0^t \big\<\tilde \xi_{N_i}(s),  \sigma_{k,\alpha} \cdot\nabla \Pi_{N_i} v\big\>_{L^2}\, \d \tilde W^{N_i, k,\alpha}_s ,
  \endaligned $$
where $\tilde u_{N_i}= B\big(\tilde \xi_{N_i} \big)$ is the velocity field on the new probability space $\tilde \Omega$ and $\alpha$ ranges in $\{1,2 \}$. Recall that $B$ is the Biot-Savart operator and $\L^\ast$ is the adjoint operator of the Lie derivative, see its formula above Theorem \ref{thm-existence}. By \eqref{lem-extra-term.1}, $S_\theta(v)$ is also a smooth divergence free vector field. Thanks to the above preparations, it is standard to show that all the terms, except the nonlinear one, converge to the corresponding limits, see for instance the proof of Theorem 2.2 at the end of \cite[Section 3]{FGL}. In the following we concentrate on the convergence of the nonlinear term and denote $\tilde\E$ the expectation on the new probability space $\big(\tilde \Omega, \tilde{\mathcal F}, \tilde\P\big)$.

We omit the time variable $s$ in the integrals to save space. We have, by triangle inequality,
  $$\aligned
  &\, \tilde\E \bigg[\sup_{t\in [0,T]} \bigg| \int_0^t f_R\big(\|\tilde \xi_{N_i}\|_{-\delta} \big) \Big\<\tilde \xi_{N_i}, \L_{\tilde u_{N_i}}^\ast (\Pi_{N_i} v) \Big\>_{L^2} \,\d s - \int_0^t f_R\big(\|\tilde \xi\|_{-\delta} \big) \big\<\tilde \xi, \L_{\tilde u}^\ast v \big\>_{L^2} \,\d s \bigg| \bigg] \\
  \leq &\, \tilde\E \bigg[ \int_0^T f_R\big(\|\tilde \xi_{N_i} \|_{-\delta} \big) \Big| \Big\<\tilde \xi_{N_i}, \L_{\tilde u_{N_i}}^\ast (\Pi_{N_i} v) \Big\>_{L^2} - \big\<\tilde \xi, \L_{\tilde u}^\ast v \big\>_{L^2} \Big| \,\d s \bigg] \\
  &+ \tilde\E \bigg[ \int_0^T \Big| f_R\big(\|\tilde \xi_{N_i} \|_{-\delta} \big) - f_R\big(\|\tilde \xi\|_{-\delta} \big)\Big| \cdot \Big| \big\<\tilde \xi, \L_{\tilde u}^\ast v \big\>_{L^2} \Big| \,\d s \bigg].
  \endaligned $$
Denote the two expectations on the right hand side by $I_1$ and $I_2$ respectively. First,
  \begin{equation}\label{proof-weak.1}
  I_1 \leq \tilde\E \bigg[ \int_0^T \Big| \Big\<\tilde \xi_{N_i}, \L_{\tilde u_{N_i}}^\ast (\Pi_{N_i} v) \Big\>_{L^2} - \big\<\tilde \xi, \L_{\tilde u}^\ast v \big\>_{L^2} \Big| \,\d s \bigg].
  \end{equation}
Recall the $\tilde\P$-a.s. convergence stated in \eqref{limit-as}; we deduce that $\tilde u_{N_i}= B\big(\tilde \xi_{N_i} \big)$ converge a.s. in the strong topology of $L^2(0,T; V)$ to the velocity field $\tilde u= B\big(\tilde \xi \big)$. Moreover, the uniform bounds \eqref{cor-bound.1} and \eqref{cor-bound.2} imply that
  $$\tilde \P \mbox{-a.s.}, \quad \|\tilde u \|_{L^2(0,T; V)} \vee \bigg(\sup_{i\geq 1} \|\tilde u_{N_i} \|_{L^2(0,T; V)} \bigg) \leq C_{\|\xi_0\|_{L^2}, \delta, R,T}.$$
Finally, since $v$ is smooth, $\Pi_{N_i}v$ converge to $v$ in $C^1(\T^3,\R^3)$. Using these facts, it is easy to show that the right hand side of \eqref{proof-weak.1} tends to 0 as $i\to \infty$.

It remains to prove that $I_2$ also vanishes as $i\to \infty$. First, one can easily show that, $\tilde\P$-a.s.,
  $$\Big| \big\<\tilde \xi, \L_{\tilde u}^\ast v \big\>_{L^2} \Big| \leq \big\| \tilde \xi \big\|_{L^2} \big\| \L_{\tilde u}^\ast v \big\|_{L^2} \leq \big\| \tilde \xi \big\|_{L^2} \| \tilde u \|_V \|v\|_{C^1} \leq C\big\| \tilde \xi \big\|_{L^2}^2 \|v\|_{C^1} \leq C_{\|\xi_0\|_{L^2}, \delta, R,T}^2 \|v\|_{C^1} .$$
Moreover, \eqref{limit-as} implies that, $\tilde\P$-a.s.,
  $$\lim_{i\to \infty} \sup_{t\in [0,T]} \big\|\tilde \xi_{N_i}(t)- \tilde \xi(t) \big\|_{-\delta} =0 .$$
Since $f_R$ is bounded and continuous, we apply the dominated convergence theorem to conclude that $I_2$ vanishes as $i\to \infty$.

Summarizing the above arguments we conclude that the limit process $\tilde\xi$ satisfies
  $$\aligned
  \big\<\tilde \xi(t), v\big\>_{L^2} &= \<\xi_0, v\>_{L^2} + \int_0^t f_R\big(\|\tilde \xi(s)\|_{-\delta} \big) \Big\<\tilde \xi(s), \L_{\tilde u(s)}^\ast v\Big\>_{L^2} \,\d s \\
  &\quad + \int_0^t \big\<\tilde \xi(s),\Delta v\big\>_{L^2} \,\d s + \int_0^t \big\<\tilde \xi(s), S_\theta(v)\big\>_{L^2} \,\d s \\
  &\quad - \frac{C_\nu}{\|\theta\|_{\ell^2}} \sum_{k, \alpha} \theta_k \int_0^t \big\<\tilde \xi(s),  \sigma_{k,\alpha} \cdot\nabla v\big\>_{L^2}\, \d \tilde W^{k,\alpha}_s,
  \endaligned $$
Thus $\tilde \xi$ is a global weak solution to the stochastic 3D Navier-Stokes equations \eqref{SNSE-cut-off-1}, i.e. \eqref{SNSE-cut-off}.

\emph{Step 2: pathwise uniqueness of \eqref{SNSE-cut-off-1}}. Assume that on a probability space $(\Omega, \mathcal F, (\mathcal F_t), \P)$ there are two solutions $\xi_i\, (i=1,2)$ of \eqref{SNSE-cut-off-1} with the same complex Brownian motions $\big\{W^{k,\alpha}: k\in \Z^3_0, \alpha=1,2 \big\}$ and the same initial condition, satisfying the bounds
  \begin{equation}\label{solution-property-uniq}
  \P \mbox{-a.s.}, \quad \|\xi_i \|_{L^\infty(0,T; H)} \vee \|\xi_i \|_{L^2(0,T; V)} \leq C_{\|\xi_0\|_{L^2}, \delta, R,T}, \quad i=1,2.
  \end{equation}
Let $\xi=\xi_1 -\xi_2$, then
  $$\aligned
  \d\xi=&\, - \big[f_R(\|\xi_1 \|_{-\delta})\L_{u_1} \xi_1- f_R(\|\xi_2 \|_{-\delta})\L_{u_2} \xi_2\big]\,\d t+ \big[ \Delta \xi + S_\theta(\xi)\big] \,\d t \\
  &\, + \frac{C_\nu}{\|\theta\|_{\ell^2}} \sum_{k,\alpha} \theta_k \Pi(\sigma_{k,\alpha}\cdot \nabla \xi) \, \d W^{k,\alpha}_t.
  \endaligned $$
Note that the vorticity $\xi$ is divergence free. Thanks to \eqref{solution-property-uniq}, one can check that the assumptions of \cite[p.72, Theorem 2.13]{RL} are verified, thus by the It\^o formula \cite[(2.5.3)]{RL} and the definition \eqref{I-theta} of $S_{\theta}(\xi)$, we have
  $$\aligned
  \d \|\xi\|_{L^2}^2=&\, - 2 \Big\<\xi, f_R(\|\xi_1 \|_{-\delta})\L_{u_1} \xi_1- f_R(\|\xi_2 \|_{-\delta})\L_{u_2} \xi_2 \Big\>_{L^2} \,\d t - 2 \|\nabla \xi\|_{L^2}^2 \,\d t \\
  &\, - \frac{2C_\nu^2}{\|\theta \|_{\ell^2}^2 } \sum_{k,\alpha} \theta_k^2 \| \Pi(\sigma_{k,\alpha}\cdot \nabla \xi ) \|_{L^2}^2 \,\d t + \frac{2 C_\nu}{\|\theta\|_{\ell^2}} \sum_{k,\alpha} \theta_k \big\<\xi, \sigma_{k,\alpha}\cdot \nabla \xi\big\> \, \d W^{k,\alpha}_t \\
  &\, + \frac{2C_\nu^2}{\|\theta\|_{\ell^2}^2} \sum_{k,\alpha} \theta_k^2 \big\| \Pi(\sigma_{k,\alpha}\cdot \nabla \xi) \big\|_{L^2}^2 \, \d t,
  \endaligned $$
where the last quadratic variation term follows from \eqref{qudratic-var}. The martingale part vanishes, since all the vector fields $\sigma_{k,\alpha}$ are divergence free; therefore,
  \begin{equation}\label{pathwise-uniq-1}
  \d \|\xi\|_{L^2}^2= - 2 \Big\<\xi, f_R(\|\xi_1 \|_{-\delta})\L_{u_1} \xi_1- f_R(\|\xi_2 \|_{-\delta})\L_{u_2} \xi_2 \Big\>_{L^2} \,\d t - 2 \|\nabla \xi\|_{L^2}^2 \,\d t.
  \end{equation}

Now we treat the difficult terms involving Lie derivatives:
  \begin{equation}\label{pathwise-uniq-2}
  \aligned
  J:=&\ \Big\<\xi, f_R(\|\xi_1 \|_{-\delta})\L_{u_1} \xi_1- f_R(\|\xi_2 \|_{-\delta})\L_{u_2} \xi_2 \Big\>_{L^2} \\
  =&\ \big[f_R(\|\xi_1 \|_{-\delta}) - f_R(\|\xi_2 \|_{-\delta})\big] \big\<\xi, \L_{u_1} \xi_1 \big\>_{L^2} + f_R(\|\xi_2 \|_{-\delta}) \big\<\xi, \L_{u_1} \xi_1- \L_{u_2} \xi_2 \big\>_{L^2} \\
  =:&\ J_1+J_2.
  \endaligned
  \end{equation}
We start with $J_1$:
  $$\aligned
  |J_1| &\leq \|f'_R\|_\infty\, \|\xi_1 -\xi_2 \|_{-\delta} \big|\big\<\xi, \L_{u_1} \xi_1 \big\>_{L^2} \big| \\
  &\leq C \|\xi \|_{L^2} \Big( \big|\big\<\xi, u_1\cdot\nabla \xi_1 \big\>_{L^2} \big| + \big|\big\<\xi, \xi_1\cdot\nabla u_1 \big\>_{L^2} \big|\Big) =: J_{1,1} + J_{1,2}.
  \endaligned $$
By H\"older's inequality with exponent $\frac13 +\frac16 + \frac12 =1$, we have
  \begin{equation*}
  J_{1,1} \leq C \|\xi \|_{L^2} \|\xi \|_{L^3} \|u_1\|_{L^6} \|\nabla\xi_1 \|_{L^2} \leq C \|\xi \|_{L^2} \|\xi \|_{1/2} \|u_1\|_{1} \|\nabla\xi_1 \|_{L^2} ,
  \end{equation*}
where we have used the Sobolev embeddings $H^{1/2}(\T^3) \hookrightarrow L^3(\T^3)$ and $H^{1}(\T^3) \hookrightarrow L^6(\T^3)$. Moreover, applying the interpolation inequality,
  \begin{equation}\label{pathwise-uniq-3}
  J_{1,1} \leq C \|\xi \|_1^{1/2} \|\xi \|_{L^2}^{3/2} \|\xi_1 \|_{L^2} \|\nabla\xi_1 \|_{L^2} \leq C \|\nabla \xi \|_{L^2}^{1/2} \|\xi \|_{L^2}^{3/2} \|\xi_1 \|_{L^2} \|\nabla\xi_1 \|_{L^2}.
  \end{equation}
Young's inequality with exponent $\frac14 + \frac34 =1$ implies
  $$J_{1,1} \leq \frac15 \|\nabla \xi \|_{L^2}^2 + C \|\xi \|_{L^2}^{2} \|\xi_1 \|_{L^2}^{4/3} \|\nabla\xi_1 \|_{L^2}^{4/3} \leq \frac15 \|\nabla \xi \|_{L^2}^2 + C_1 \|\xi \|_{L^2}^{2} \|\nabla\xi_1 \|_{L^2}^{4/3}, $$
since, by \eqref{solution-property-uniq}, $\xi_1$ is a.s. bounded in $ L^\infty(0,T; H)$. Next we turn to estimate $J_{1,2}$. By H\"older's inequality with exponent $\frac13 + \frac13 +\frac13 =1$,
  $$\aligned
  J_{1,2} &\leq C \|\xi \|_{L^2} \|\xi \|_{L^3} \|\xi_1 \|_{L^3} \|\nabla u_1\|_{L^3} \leq C \|\xi \|_{L^2} \|\xi \|_{1/2} \|\xi_1 \|_{1/2}^2 \\
  &\leq C \|\xi \|_1^{1/2} \|\xi \|_{L^2}^{3/2} \|\xi_1 \|_{L^2} \|\xi_1 \|_{1} \leq C \|\nabla \xi \|_{L^2}^{1/2} \|\xi \|_{L^2}^{3/2} \|\xi_1 \|_{L^2} \|\nabla \xi_1 \|_{L^2} ,
  \endaligned $$
which is the same as the right hand side of \eqref{pathwise-uniq-3}. Thus, similarly as above, we have
  $$J_{1,2} \leq \frac15 \|\nabla \xi \|_{L^2}^2 + C_1 \|\xi \|_{L^2}^{2} \|\nabla\xi_1 \|_{L^2}^{4/3}. $$
Summarizing the above arguments, we obtain
  \begin{equation}\label{J-1}
  J_1\leq \frac25 \|\nabla \xi \|_{L^2}^2 + C \|\xi \|_{L^2}^{2} \|\nabla\xi_1 \|_{L^2}^{4/3}.
  \end{equation}

Next we estimate $J_2$ defined in \eqref{pathwise-uniq-2} which can be done in the same way as for $J_1$. Since $0\leq f_R\leq 1$,
  \begin{equation}\label{J-2}
  |J_2| \leq \big|\big\<\xi, \L_{u_1} \xi_1- \L_{u_2} \xi_2 \big\>_{L^2} \big| \leq \big|\big\<\xi, \L_{u_1} \xi \big\>_{L^2} \big| + \big|\big\<\xi, \L_{u} \xi_2 \big\>_{L^2} \big|= J_{2,1} + J_{2,2},
  \end{equation}
where $u=u_1- u_2$. We have
  $$J_{2,1} = \big|\<\xi, u_1\cdot\nabla \xi \>_{L^2} - \<\xi, \xi\cdot\nabla u_1 \>_{L^2}\big| = \big| \<\xi, \xi\cdot\nabla u_1 \>_{L^2} \big|,$$
since $u_1$ is divergence free. By H\"older's inequality and the Sobolev embedding $H^{1/2}(\T^3) \hookrightarrow L^3(\T^3)$,
  $$\aligned
  J_{2,1} & \leq \|\xi \|_{L^3}^2 \|\nabla u_1 \|_{L^3} \leq C\|\xi \|_{1/2}^2 \|\nabla u_1 \|_{1/2} \\
  & \leq C\|\xi \|_{L^2} \|\xi \|_1  \|\nabla u_1 \|_1 \leq C \|\nabla \xi \|_{L^2} \|\xi \|_{L^2} \|\nabla \xi_1 \|_{L^2} .
  \endaligned $$
where the third step follows from the interpolation inequality and $\|\nabla u_1 \|_{1/2} \leq \|\nabla u_1 \|_{1}$, respectively. We deduce that
  \begin{equation}\label{proof-uniq.1}
  J_{2,1} \leq \frac15 \|\nabla \xi \|_{L^2}^2 + C \|\xi \|_{L^2}^2 \|\nabla \xi_1 \|_{L^2}^2 .
  \end{equation}

Next, we consider $J_{2,2}$ in \eqref{J-2}:
  \begin{equation}\label{proof-uniq.1.5}
  J_{2,2} \leq |\<\xi, u\cdot\nabla \xi_2\>_{L^2}| + |\<\xi, \xi_2\cdot\nabla u\>_{L^2} | .
  \end{equation}
Using the H\"older inequality with exponents $\frac12 + \frac13 + \frac16 =1$,
  $$  |\<\xi, u\cdot\nabla \xi_2\>_{L^2}| \leq \|\xi \|_{L^3} \|u \|_{L^6} \|\nabla \xi_2 \|_{L^2} \leq C\|\xi \|_{1/2} \|u \|_{1} \|\nabla \xi_2 \|_{L^2} ,$$
where we have also used the Sobolev embedding $H^{1}(\T^3) \hookrightarrow L^6(\T^3)$. Now, by the interpolation inequality,
  $$|\<\xi, u\cdot\nabla \xi_2\>_{L^2}| \leq C\|\xi \|_{L^2}^{1/2} \|\xi \|_{1}^{1/2}  \|\xi \|_{L^2} \|\nabla \xi_2 \|_{L^2} \leq C\|\nabla \xi \|_{L^2}^{1/2}  \|\xi \|_{L^2}^{3/2} \|\nabla \xi_2\|_{L^2}.  $$
Thus, by Young's inequality with exponents $\frac14 + \frac34 =1$,
  \begin{equation}\label{proof-uniq.2}
  |\<\xi, u\cdot\nabla \xi_2\>_{L^2}| \leq \frac15 \|\nabla \xi \|_{L^2}^2 + C \|\xi \|_{L^2}^2 \|\nabla \xi_2\|_{L^2}^{4/3} .
  \end{equation}
It remains to estimate the last term $|\<\xi, \xi_2\cdot\nabla u\>_{L^2}| $ in \eqref{proof-uniq.1.5}. We have
  $$\aligned
  |\<\xi, \xi_2\cdot\nabla u\>_{L^2}| &\leq \|\xi \|_{L^3} \|\xi_2 \|_{L^3}\|\nabla u \|_{L^3} \leq C \|\xi \|_{1/2}^2 \|\xi_2 \|_{1/2} \\
  &\leq C \|\xi \|_{L^2} \|\xi \|_{1} \|\xi_2 \|_1 \leq C \|\nabla\xi \|_{L^2} \|\xi \|_{L^2} \|\nabla\xi_2 \|_{L^2}.
  \endaligned $$
Finally, we get
  $$|\<\xi, \xi_2\cdot\nabla u\>_{L^2}|\leq \frac15 \|\nabla \xi \|_{L^2}^2 + C \|\xi \|_{L^2}^2 \|\nabla\xi_2 \|_{L^2}^2. $$
This estimate together with \eqref{J-2}--\eqref{proof-uniq.2} implies
  $$J_{2} \leq \frac35 \|\nabla \xi \|_{L^2}^2 + C \|\xi \|_{L^2}^2 \big(\|\nabla \xi_1 \|_{L^2}^2 + \|\nabla \xi_2\|_{L^2}^{4/3} + \|\nabla\xi_2 \|_{L^2}^2 \big) .$$

Now we combine the above estimate with \eqref{pathwise-uniq-2} and \eqref{J-1} to deduce
  $$|J| \leq \|\nabla \xi \|_{L^2}^2 + C \|\xi \|_{L^2}^2 \big(\|\nabla\xi_1 \|_{L^2}^{4/3}+ \|\nabla \xi_1 \|_{L^2}^2 + \|\nabla \xi_2\|_{L^2}^{4/3} + \|\nabla\xi_2 \|_{L^2}^2 \big). $$
Substituting this result into \eqref{pathwise-uniq-1}, we conclude that, $\P$-a.s. for all $t\in [0,T]$,
  $$\d\|\xi \|_{L^2}^2 \leq C \|\xi \|_{L^2}^2 \big(\|\nabla\xi_1 \|_{L^2}^{4/3}+ \|\nabla \xi_1 \|_{L^2}^2 + \|\nabla \xi_2\|_{L^2}^{4/3} + \|\nabla\xi_2 \|_{L^2}^2 \big) \,\d t. $$
By the regularity properties \eqref{solution-property-uniq} on the two solutions $\xi_1$ and $\xi_2$, the quantity in the brackets on the right hand side is integrable. Thus Gronwall's inequality give us $\|\xi(t) \|_{L^2}^2=0$ $\P$-a.s. for all $t\leq T$. The proof is complete.
\end{proof}

\section{The scaling limit and its consequences}

In this part, we take the sequence $\{\theta^N \}_{N\in \N}$ defined in \eqref{theta-N-def}, which satisfies
  \begin{equation}\label{theta-limit}
  \lim_{N\to \infty} \frac{\|\theta^N \|_{\ell^\infty}}{\|\theta^N \|_{\ell^2}}=0.
  \end{equation}
For any $N\geq 1$, we consider the stochastic 3D Navier-Stokes equations \eqref{SNSE-cut-off-N} with cut-off, namely,
  \begin{equation}\label{SNSE-vort-3}
  \d \xi^N + f_R\big(\|\xi^N \|_{-\delta} \big)\, \L_{u^N} \xi^N\,\d t = \big[ \Delta \xi^N + S_{\theta^N}\big(\xi^N \big) \big] \,\d t + \frac{C_\nu}{\|\theta^N \|_{\ell^2}} \sum_{k,\alpha} \theta^N_k \Pi\big(\sigma_{k,\alpha}\cdot \nabla \xi^N \big) \, \d W^{k,\alpha}_t
  \end{equation}
with $\xi^N|_{t=0} = \xi^N_0\in B_H(R_0)$. Recall that $u^N$ is related to $\xi^N$ via the Biot-Savart law: $u^N= B\big(\xi^N \big)$. By Theorem \ref{thm-existence}, the system \eqref{SNSE-vort-3} has a pathwise unique solution $\xi^N$ on the interval $[0,T]$, satisfying the bounds below:
  \begin{equation}\label{bound-weak-solu.1}
  \P\mbox{-a.s.},\quad  \big\|\xi^N \big\|_{L^\infty(0,T;H)} \vee \big\|\xi^N \big\|_{L^2(0,T;V)} \leq C_{R_0,\delta, R,T},
  \end{equation}
where $C_{R_0,\delta, R,T}$ is independent of $\nu$ and $N$; moreover, for any divergence free test vector field $v$, one has $\P$-a.s. for all $t\in [0,T]$,
  \begin{equation}\label{SNSE-weak-1}
  \aligned
  \big\< \xi^N(t), v\big\>_{L^2} &= \big\< \xi^N_0, v\big\>_{L^2} + \int_0^t f_R\big(\| \xi^N(s)\|_{-\delta} \big) \Big\< \xi^N(s), \L_{u^N(s)}^\ast v\Big\>_{L^2} \,\d s\\
  &\quad+ \int_0^t \big\< \xi^N(s),\Delta v\big\>_{L^2} \,\d s+ \int_0^t \big\< \xi^N(s),S_{\theta^N} (v)\big\>_{L^2} \,\d s\\
  &\quad - \frac{C_\nu}{\|\theta^N\|_{\ell^2}} \sum_{k,\alpha} \theta^N_k \int_0^t \big\< \xi^N(s),  \sigma_{k,\alpha} \cdot\nabla v\big\>_{L^2}\, \d W^{k,\alpha}_s.
  \endaligned
  \end{equation}

Using the uniform bounds \eqref{bound-weak-solu.1} and the equation \eqref{SNSE-weak-1}, one can show that the assertion of Lemma \ref{lem-moment-estim} still holds for $\xi^N,\, N\in \N$. Thus, let $Q^N$ be the law of $\xi^N,\, N\in \N$, as in Corollary \ref{cor-tight}, we can prove

\begin{lemma}\label{lem-tightness}
The family $\{Q^N \}_{N\in\N}$ is tight on $L^2(0,T; H)$ and on $C\big([0,T], H^{-\delta} \big)$.
\end{lemma}

Next, repeating the arguments below Corollary \ref{cor-tight}, we can find a subsequence $\{Q^{N_i} \}_{i\in\N}$ converging weakly to some probability measure $Q$ which is supported on $L^2(0,T; H)$ and on $C\big([0,T], H^{-\delta} \big)$. Moreover, there is a new probability space $\big( \tilde\Omega, \tilde{\mathcal F}, \tilde \P\big)$ and a sequence of processes $\big\{ \big(\tilde \xi^{N_i}, \tilde W^{N_i} \big) \big\}_{i\in \N}$ and $\big(\tilde \xi, \tilde W \big)$ defined on $\tilde\Omega$, such that
\begin{itemize}
\item[(a)] for each $i\in \N$, $\big(\tilde \xi^{N_i}, \tilde W^{N_i}\big) \stackrel{d}{\sim} (\xi^{N_i}, W)$; in particular, $\tilde W^{N_i}$ and $\tilde W$ are families of complex Brownian motions;
\item[(b)] $\tilde \P$-a.s., as $i\to \infty$, $\tilde \xi^{N_i}$ converge to $\tilde \xi$ strongly in $L^2(0,T; H)$ and in $C\big([0,T], H^{-\delta} \big)$, and $\tilde W^{N_i,k,\alpha}$ converge in $C([0,T],\mathbb C)$ to $\tilde W^{k,\alpha}$ for all $k\in \Z^3_0$ and $\alpha =1, 2$.
\end{itemize}

Now we prove the following intermediate result.

\begin{proposition}\label{prop-NSE}
Assume that $\xi^N_0$ converge weakly in $H$ to $\xi_0\in B_H(R_0)$. Then the limit process $\tilde \xi$ solves the deterministic 3D Navier-Stokes equations with cut-off: for any divergence free test vector field $v\in C^\infty(\T^3, \R^3)$,
  \begin{equation}\label{prop-NSE.1}
  \big\< \tilde\xi(t), v\big\>_{L^2} = \< \xi_0, v\>_{L^2} + \Big(1+\frac{3}{5}\nu \Big) \int_0^t \big\<\tilde\xi(s),\Delta v\big\>_{L^2} \,\d s + \int_0^t f_R\big(\| \tilde\xi(s)\|_{-\delta} \big) \big\< \tilde\xi(s), \L_{\tilde u(s)}^\ast v\big\>_{L^2} \,\d s,
  \end{equation}
where $\tilde u=  B\big(\tilde\xi \big)$ is the corresponding velocity field.
\end{proposition}

\begin{proof}
By the above assertion (a) and \eqref{SNSE-weak-1}, the process $\tilde \xi^{N_i}$ on the new probability space $\big( \tilde\Omega, \tilde{\mathcal F}, \tilde \P\big)$ satisfies the equation below:
  $$   \aligned
  \big\< \tilde\xi^{N_i}(t), v\big\>_{L^2} &= \big\< \xi^{N_i}_0, v\big\>_{L^2} + \int_0^t \big\< \tilde\xi^{N_i}(s),\Delta v\big\>_{L^2} \,\d s + \int_0^t \big\< \tilde\xi^{N_i}(s), S_{\theta^{N_i}} (v)\big\>_{L^2} \,\d s\\
  &\quad + \int_0^t f_R\big(\| \tilde\xi^{N_i}(s)\|_{-\delta} \big) \Big\< \tilde\xi^{N_i}(s), \L_{\tilde u^{N_i}(s)}^\ast v\Big\>_{L^2} \,\d s\\
  &\quad - \frac{C_\nu}{\|\theta^{N_i} \|_{\ell^2}} \sum_{k, \alpha} \theta^{N_i}_k \int_0^t \big\< \tilde\xi^{N_i}(s),  \sigma_{k,\alpha} \cdot\nabla v\big\>_{L^2}\, \d \tilde W^{N_i, k,\alpha}_s,
  \endaligned $$
where $\tilde u^{N_i}= B\big(\tilde \xi^{N_i} \big)$ is the velocity field on the new probability space $\tilde\Omega$. We want to take limit $i\to \infty$ in the above equation. The proof is similar to the existence part of Theorem \ref{thm-existence}, with two main differences: (1) by Theorem \ref{prop-extra-term}, $S_{\theta^{N_i}} (v)$ converge in $L^2(\T^3,\R^3)$ to $\frac{3\nu}5 \Delta v$, thus the assertion (b) implies
  $$  \tilde\P \mbox{-a.s.}, \quad \lim_{i\to \infty} \sup_{t\in [0,T]} \bigg| \int_0^t \big\< \tilde\xi^{N_i}(s), S_{\theta^{N_i}} (v)\big\>_{L^2} \,\d s - \frac{3\nu}5 \int_0^t \big\< \tilde\xi(s), \Delta v \big\>_{L^2} \,\d s \bigg| =0.$$
(2) the martingale part vanishes in the mean square sense. We conclude from these two facts and the weak convergence of $\xi^{N_i}_0$ to $\xi_0$ that the limit equation is \eqref{prop-NSE.1}.

It remains to prove the assertion (2). We denote $\tilde\E$ the expectation on the probability space $\big( \tilde\Omega, \tilde{\mathcal F}, \tilde \P\big)$ and recall that $C_\nu= \sqrt{3\nu/2}$. By the It\^o isometry and \eqref{qudratic-var},
  $$\aligned &\, \tilde \E \bigg( \frac{C_\nu}{\|\theta^{N_i} \|_{\ell^2}} \sum_{k,\alpha} \theta^{N_i}_k \int_0^t \big\< \tilde\xi^{N_i}(s),  \sigma_{k,\alpha} \cdot\nabla v\big\>_{L^2}\, \d \tilde W^{N_i, k,\alpha}_s \bigg)^2 \\
  =&\, \frac{3\nu}{\|\theta^{N_i} \|_{\ell^2}^2} \sum_{k,\alpha} \big( \theta^{N_i}_k\big)^2\, \tilde\E \int_0^t \big| \big\< \tilde\xi^{N_i}(s),  \sigma_{k,\alpha} \cdot\nabla v\big\>_{L^2} \big|^2\, \d s \\
  \leq &\, 3\nu\frac{\|\theta^{N_i} \|_{\ell^\infty}^2}{\|\theta^{N_i} \|_{\ell^2}^2}\, \tilde\E \int_0^t \sum_{k,\alpha} \big| \big\< \tilde\xi^{N_i}(s),  \sigma_{k,\alpha} \cdot\nabla v\big\>_{L^2} \big|^2\, \d s.
  \endaligned $$
Using the fact that $\{\sigma_{k,\alpha}: k\in \Z^3_0, \alpha =1,2\}$ is an orthonormal system,
  $$\aligned
  \sum_{k,\alpha} \big| \big\< \tilde\xi^{N_i}(s),  \sigma_{k,\alpha} \cdot\nabla v\big\>_{L^2} \big|^2 &= \sum_{k,\alpha} \big| \big\<(\nabla v)^\ast\, \tilde\xi^{N_i}(s), \sigma_{k,\alpha} \big\>_{L^2} \big|^2 \\
  &\leq \big\|(\nabla v)^\ast\, \tilde\xi^{N_i}(s) \big\|_{L^2}^2 \leq \|\nabla v\|_\infty^2 \big\|\tilde\xi^{N_i}(s) \big\|_{L^2}^2.
  \endaligned $$
Recall that $\tilde\xi^{N_i}$ has the same law as $\xi^{N_i}$, the latter satisfying the uniform bound \eqref{bound-weak-solu.1}. Thus,
  $$\aligned &\, \tilde \E \bigg( \frac{C_\nu}{\|\theta^{N_i} \|_{\ell^2}} \sum_{k,\alpha} \theta^{N_i}_k \int_0^t \big\< \xi^{N_i}(s),  \sigma_{k,\alpha} \cdot\nabla v\big\>_{L^2}\, \d \tilde W^{N_i, k,\alpha}_s \bigg)^2 \\
  \leq &\, 3\nu\frac{\|\theta^{N_i} \|_{\ell^\infty}^2}{\|\theta^{N_i} \|_{\ell^2}^2}\, \tilde\E \int_0^t \|\nabla v\|_\infty^2 \big\|\tilde \xi^{N_i}(s) \big\|_{L^2}^2 \, \d s \leq C_{R_0,\nu, \delta, R,T} \|\nabla v\|_\infty^2 \frac{\|\theta^{N_i} \|_{\ell^\infty}^2}{\|\theta^{N_i} \|_{\ell^2}^2},
  \endaligned $$
which, by \eqref{theta-limit}, tends to 0 as $i\to \infty$. Thus the limit $\tilde\xi$ satisfies the equation \eqref{prop-NSE.1}.
\end{proof}

Now we need the following classical estimate for deterministic 3D Navier-Stokes equations.

\begin{lemma}\label{lem-determ-3DNSE}
Let $\xi_0\in H$ be fixed. If $\nu_1> \frac{C_0}{\sqrt{2\pi}} \|\xi_0 \|_{L^2}$, where $C_0$ is a dimensional constant comes from some Sobolev embedding inequality, then the deterministic 3D Navier-Stokes equations
  \begin{equation}\label{lem-determ-3DNSE.0}
  \partial_t \xi + \L_{u}\xi= \nu_1\Delta \xi, \quad \xi|_{t=0} = \xi_0
  \end{equation}
have a unique global strong solution satisfying
  \begin{equation}\label{lem-determ-3DNSE.1}
  \|\xi_t \|_{L^2} \leq \bigg[\bigg(\frac1{\|\xi_0 \|_{L^2}^4} - \frac{C_0^4}{4\pi^2\nu_1^4} \bigg) e^{8\pi^2\nu_1 t} + \frac{C_0^4}{4\pi^2\nu_1^4} \bigg]^{-1/4} \leq \frac{\sqrt{2\pi} \nu_1}{C_0}.
  \end{equation}
\end{lemma}

\begin{proof}
We only recall some steps of proving the estimate \eqref{lem-determ-3DNSE.1}, see e.g. \cite[p.20]{Temam} where it was done for the velocity field. The proof of the following inequality is easier than that of \eqref{estim-nonlinear}; we have (cf. \cite[(3.26)]{Temam})
  $$\frac{\d}{\d t} \|\xi \|_{L^2}^2 \leq - 4\pi^2 \nu_1 \|\xi \|_{L^2}^2 + \frac{C_0^4}{\nu_1^3} \|\xi \|_{L^2}^6,$$
where $4\pi^2$ comes from the application of the Poincar\'e inequality on $\T^3$. Letting $y(t)= \|\xi_t \|_{L^2}^2$ yields the differential inequality $y' \leq - 4\pi^2\nu_1 y + \frac{C_0^4}{\nu_1^3} y^3$ with $y(0) = \|\xi_0 \|_{L^2}^2$. The latter can be solved explicitly to yield
  $$ y(t) \leq \bigg[\bigg(\frac1{y(0)^2} - \frac{C_0^4}{4\pi^2\nu_1^4} \bigg) e^{8\pi^2\nu_1 t} + \frac{C_0^4}{4\pi^2\nu_1^4} \bigg]^{-1/2}, $$
which implies the estimate \eqref{lem-determ-3DNSE.1}.
\end{proof}

As a consequence, we have

\begin{corollary}\label{cor-determ-NSE}
Given $R_0>0$,  if we choose $\nu> \frac53 \big(\frac{C_0}{\sqrt{2\pi}} R_0 -1 \big)$ and
  $$R> \frac{\sqrt{2\pi} }{C_0} \Big(1+\frac35 \nu \Big) , $$
then for any $\xi_0\in B_H(R_0)$, the equation \eqref{prop-NSE.1} reduces to the deterministic 3D Navier-Stokes equations without cut-off.
\end{corollary}

\begin{proof}
Indeed, applying Lemma \ref{lem-determ-3DNSE} with $\nu_1= 1+\frac35 \nu$, we have
  $$\|\tilde\xi(t) \|_{-\delta} \leq \|\tilde\xi(t) \|_{L^2} \leq \frac{\sqrt{2\pi} }{C_0} \Big(1+\frac35 \nu \Big) <R.$$
Therefore, the cut-off part in the equation \eqref{prop-NSE.1} is identically equal to 1, i.e., \eqref{prop-NSE.1} reduces to
  \begin{equation}\label{cor-determ-NSE.1}
  \big\< \tilde\xi(t), v\big\>_{L^2} = \< \xi_0, v\>_{L^2} + \Big(1+\frac35 \nu \Big) \int_0^t \big\<\tilde\xi(s),\Delta v\big\>_{L^2} \,\d s + \int_0^t \big\< \tilde\xi(s), \L_{\tilde u(s)}^\ast v\big\>_{L^2} \,\d s,
  \end{equation}
which is the weak formulation of vorticity form of the deterministic 3D Navier-Stokes equations.
\end{proof}

Now we are ready to prove the first main result of this paper.

\begin{proof}[Proof of Theorem \ref{main-thm}]
We take the parameters $\nu$ and $R$ as in Corollary \ref{cor-determ-NSE}. In the above we have already shown that any weakly convergent subsequence of $\{Q^N \}_{N\geq 1}$ converge weakly to the Dirac measure $\delta_{\xi}$, where $\xi$ is the unique global solution of the deterministic 3D Navier-Stokes equations.  Lemma \ref{lem-tightness} implies that the family $\{Q^N \}_{N\geq 1}$ is tight on both $L^2(0,T; H)$ and on $C\big([0,T], H^{-\delta} \big)$. Therefore the whole sequence $\{Q^N \}_{N\geq 1}$ converge weakly to the Dirac measure $\delta_{\xi}$.

It remains to prove the second assertion of Theorem \ref{main-thm}. We argue by contradiction. Suppose there exists $\eps_0>0$ small enough such that
  $$ \limsup_{N\to \infty} \sup_{\xi_0\in B_H(R_0)} Q^N_{\xi_0} \Big(\varphi\in \mathcal X: \|\varphi -\xi_\cdot (\xi_0) \|_{\mathcal X} >\eps_0 \Big)>0, $$
where we have denoted by $\|\cdot \|_{\mathcal X}= \|\cdot \|_{L^2(0,T; H)} \vee \|\cdot \|_{C([0,T], H^{-\delta})}$. Recall that $Q^N_{\xi_0}$ is the law of the pathwise unique solution $\xi^N$ to \eqref{SNSE-vort-3} with initial condition $\xi^N|_{t=0} =\xi_0\in B_H(R_0)$ and $\xi_\cdot (\xi_0)$ is the unique global solution of the deterministic 3D Navier-Stokes equations \eqref{cor-determ-NSE.1} with initial condition $\xi_0$. Then we can find a subsequence of integers $\{N_i\}_{i\geq 1}$, and $\xi^{N_i}_0\in B_H(R_0)$, $Q^{N_i} := Q^{N_i}_{ \xi^{N_i}_0} ,\, i\geq 1$, such that (choose $\eps_0$ even smaller if necessary)
  \begin{equation}\label{proof-scaling-1}
  Q^{N_i} \Big(\varphi\in \mathcal X: \big\|\varphi -\xi_\cdot \big(\xi^{N_i}_0\big) \big\|_{\mathcal X} >\eps_0 \Big) \geq \eps_0 >0.
  \end{equation}
For each $i\geq 1$, let $\xi^{N_i}$ be the pathwise unique solution of \eqref{SNSE-vort-3} (replacing $N$ by $N_i$) with the initial value $\xi^{N_i}|_{t=0}= \xi^{N_i}_0$; then $\xi^{N_i}$ has the law $Q^{N_i}$. Since $\big\{ \xi^{N_i}_0 \big\}_{i\geq 1}$ is contained in the ball $B_H(R_0)$, there exists a subsequence of $\big\{\xi^{N_i}_0 \big\}_{i\geq 1}$ which converges weakly to some $\xi_0\in B_H(R_0)$. For simplification of notations, we assume the sequence $\big\{ \xi^{N_i}_0 \big\}_{i\geq 1}$ itself converges weakly (without taking a subsequence).

We can show as in Lemma \ref{lem-tightness} that the family $\{Q^{N_i} \}_{i\geq 1}$ is tight on $\mathcal X= L^2(0,T; H) \cap C\big([0,T], H^{-\delta} \big)$, hence, up to a subsequence, $Q^{N_i}$ converge weakly to some probability measure $Q$ supported on $\mathcal X$. The rest of the arguments are similar to those below Lemma \ref{lem-tightness}. Namely, by Skorokhod's representation theorem, we can find a new probability space $\big( \tilde\Omega, \tilde{\mathcal F}, \tilde \P\big)$ and a sequence of processes $\big\{\tilde \xi^{N_i} \big\}_{i\in \N}$ defined on $\tilde\Omega$, such that for each $i\in \N$, $\tilde \xi^{N_i}$ has the same law $Q^{N_i}$ as $\xi^{N_i}$, and $\tilde \P$-a.s., $\tilde \xi^{N_i}$ converge to some $\tilde\xi$ strongly in $L^2(0,T; H)$ and in $C\big([0,T], H^{-\delta} \big)$. As before, the limit $\tilde\xi$ solves the deterministic 3D Navier-Stokes equations \eqref{cor-determ-NSE.1} with initial condition $\xi_0$. From this we conclude that $\tilde\xi = \xi_\cdot(\xi_0)$, and thus, as $i\to \infty$, $\tilde \xi^{N_i}$ converge in $\mathcal X$ to $\xi_\cdot(\xi_0)$ in probability, i.e., for any $\eps>0$,
  \begin{equation}\label{proof-scaling-2}
  \lim_{N\to \infty} \tilde\P \Big(\big\|\tilde\xi^{N_i} -\xi_\cdot(\xi_0) \big\|_{\mathcal X} >\eps \Big) =0.
  \end{equation}
Note that $\tilde\xi^{N_i} \stackrel{d}{\sim} Q^{N_i}$, \eqref{proof-scaling-1} implies
  \begin{equation}\label{proof-scaling-3}
  \tilde \P\Big(\big\|\tilde\xi^{N_i} -\xi_\cdot \big(\xi^{N_i}_0\big) \big\|_{\mathcal X} >\eps_0 \Big) \geq \eps_0 >0.
  \end{equation}
We have the triangle inequality:
  \begin{equation}\label{proof-scaling-4}
  \big\|\tilde\xi^{N_i} -\xi_\cdot \big(\xi^{N_i}_0\big) \big\|_{\mathcal X} \leq \big\|\tilde\xi^{N_i} -\xi_\cdot (\xi_0) \big\|_{\mathcal X} + \big\|\xi_\cdot (\xi_0) -\xi_\cdot \big(\xi^{N_i}_0\big) \big\|_{\mathcal X}.
  \end{equation}
Using the estimate in Lemma \ref{lem-determ-3DNSE} and also the deterministic 3D Navier-Stokes equations, one can easily show that the family $\big\{ \xi_\cdot \big(\xi^{N_i}_0\big) \big\}_{i\geq 1}$ is bounded in $L^\infty(0,T; H) \cap L^2(0,T; V)$. The boundedness in $W^{1/3,4}\big(0,T; H^{-6} \big) \cap W^{\gamma,2}\big(0,T; H^{-6} \big)$ (here $\gamma$ is any number in $(0,1/2)$) can be proven by following the arguments in Lemma \ref{lem-moment-estim} and Corollary \ref{cor-tight}, without taking expectation. Then Theorem \ref{thm-Simon} implies the family $\big\{ \xi_\cdot \big(\xi^{N_i}_0\big) \big\}_{i\geq 1}$ is sequentially compact in $\mathcal X= L^2(0,T; H) \cap C\big([0,T], H^{-\delta} \big)$. Therefore, up to a subsequence, $\xi_\cdot \big(\xi^{N_i}_0\big)$ converge in $\mathcal X$ to some $\bar \xi$, which can be shown to solve \eqref{cor-determ-NSE.1} since $\xi^{N_i}_0$ converge weakly to $\xi_0$. In other words, $\bar \xi= \xi_\cdot (\xi_0)$ and $\big\|\xi_\cdot \big(\xi^{N_i}_0\big) -\xi_\cdot (\xi_0) \big\|_{\mathcal X} \to 0$ as $i\to \infty$. Combining this result with \eqref{proof-scaling-2}--\eqref{proof-scaling-4}, we get a contradiction.
\end{proof}

The rest of this section is devoted to the proof of Theorem \ref{main-result-thm}. We start with the following elementary result.

\begin{proposition}\label{4-prop}
Let $r_0 =(2\pi^2)^{1/4} / C_0$, where $C_0$ is a dimensional constant coming from some Sobolev embedding inequality. Then for all $\nu>0$ and $\theta\in \ell^2$, the stochastic 3D Navier-Stokes equations
  \begin{equation}\label{4-prop.1}
  \d \xi + \L_u \xi\,\d t = \Delta \xi\,\d t + \frac{C_\nu}{\|\theta\|_{\ell^2}} \sum_{k,\alpha} \theta_k \Pi( \sigma_{k,\alpha}\cdot \nabla \xi) \circ \d W^{k,\alpha}_t
  \end{equation}
have a pathwise unique global solution for any $\xi_0\in B_H(r_0)$.
\end{proposition}

\begin{proof}
Here we do not provide the complete proof which is similar to that in the deterministic theory, instead we only give some heuristic arguments. First we prove some a priori estimates on the solutions: if $\|\xi_0 \|_{L^2} \leq r_0 $, then $\P$-a.s. for all $t>0$,
  $$\|\xi_t\|_{L^2} \leq 2^{1/4} \|\xi_0\|_{L^2}\, e^{-2\pi^2 t}$$
and ($C_1$ is some positive constant independent of $t$)
  $$\int_0^t \|\nabla\xi_s \|_{L^2}^2 \,\d s \leq \|\xi_0\|_{L^2}^2 + C_1.$$

Indeed, using the Stratonovich calculus and the fact that $\div(\sigma_{k,\alpha}) =0$,
  $$\aligned
  \d \|\xi\|_{L^2}^2 &= -2 \<\xi, \L_u \xi\>_{L^2} \,\d t + 2\<\xi, \Delta \xi\>_{L^2} \,\d t + \frac{2 C_\nu}{\|\theta\|_{\ell^2}} \sum_{k, \alpha} \theta_k \big\<\xi, \Pi( \sigma_{k,\alpha}\cdot \nabla \xi) \big\>_{L^2} \circ \d W^{k,\alpha}_t \\
  &= 2\<\xi, \xi\cdot \nabla u\>_{L^2} \,\d t - 2\|\nabla\xi \|_{L^2}^2 \,\d t .
  \endaligned$$
The rest of the computations are the same as the deterministic case, see e.g. Lemma 4.3 above (taking $\nu_1=1$). Then we get
  $$\aligned  \|\xi_t \|_{L^2} &\leq \bigg[\bigg(\frac1{\|\xi_0 \|_{L^2}^4} - \frac{C_0^4}{4\pi^2} \bigg) e^{8\pi^2 t} + \frac{C_0^4}{4\pi^2} \bigg]^{-1/4} \\
  &\leq \bigg[\bigg(\frac1{\|\xi_0 \|_{L^2}^4} - \frac{C_0^4}{4\pi^2} \bigg) e^{8\pi^2 t}  \bigg]^{-1/4} \\
  &\leq \bigg[\frac1{2\|\xi_0 \|_{L^2}^4} e^{8\pi^2 t}  \bigg]^{-1/4} = 2^{1/4} \|\xi_0\|_{L^2}\, e^{-2\pi^2 t},
  \endaligned $$
where the third inequality follows from the condition $\|\xi_0 \|_{L^2} \leq r_0 = (2\pi^2)^{1/4}/ C_0$, which implies
  $$\frac1{\|\xi_0 \|_{L^2}^4} - \frac{C_0^4}{4\pi^2} \geq \frac1{2\|\xi_0 \|_{L^2}^4}. $$
This gives us the first estimate.

Next using the inequality
  $$\aligned
  \d \|\xi\|_{L^2}^2 &= 2\<\xi, \xi\cdot \nabla u\>_{L^2} \,\d t - 2\|\nabla\xi \|_{L^2}^2 \,\d t \\
  &\leq - \|\nabla\xi \|_{L^2}^2 \,\d t + C_1 \|\xi_t \|_{L^2}^6 \,\d t,
  \endaligned$$
we obtain from the above estimate of $\|\xi_t \|_{L^2}$ that
  $$\aligned
  \int_0^t \|\nabla\xi_s \|_{L^2}^2 \,\d s &\leq \|\xi_0\|_{L^2}^2 + C_1 \int_0^t \|\xi_s \|_{L^2}^6 \,\d s \\
  &\leq \|\xi_0\|_{L^2}^2 +C_1 \int_0^t 2^{3/2} \|\xi_0\|_{L^2}^6 e^{-12\pi^2 s} \,\d s \leq \|\xi_0\|_{L^2}^2 +C'_1.
  \endaligned $$

Thanks to the above a priori estimates, we can repeat the arguments in Section 3 to show the existence of weak solutions to \eqref{4-prop.1}. In Section 3, the existence is proven on any finite interval $[0,T]$, but it can be extended to $[0,\infty)$.

Next, for two weak solutions (defined on the same probability space) with a priori bounds as above, we can prove as in the second part of Theorem \ref{thm-existence} that the pathwise uniqueness holds for \eqref{4-prop.1}. Thus we obtain a global pathwise unique solution by the Yamada-Watanabe type result (see \cite[Theorem 3.14]{Kurtz}).
\end{proof}

Now we are ready to prove Theorem \ref{main-result-thm}.

\begin{proof}[Proof of Theorem \ref{main-result-thm}]
Recall the choices of the parameters $R_0, \nu, R, \eps,T$ and $N_0(R_0, \nu, R, \eps,T)$; let $N> N_0(R_0, \nu, R, \eps,T)$. By Corollary \ref{thm-uniqueness}, for any $\xi_0\in B_H(R_0)$, the stochastic 3D Navier-Stokes equations \eqref{SNSE-N} have a pathwise unique strong solution $\xi^N$ with initial condition $\xi_0$, which exists up to time $T$ with a probability greater than $1-\eps$; moreover, by \eqref{sect-1-bound.1},
  \begin{equation}\label{proof-cor.1}
  \P \Big( \big\|\xi^N -\xi \big\|_{L^2(0,T; H)} \vee \big\|\xi^N -\xi \big\|_{C([0,T], H^{-\delta})} \leq \eps \Big) \geq 1- \eps,
  \end{equation}
where $\xi$ is the unique global solution of the deterministic 3D Navier-Stokes equations \eqref{lem-determ-3DNSE.0} with $\nu_1 = 1+\frac35 \nu$. By \eqref{lem-determ-3DNSE.1}, for all $t>0$,
  $$\|\xi_t \|_{L^2} \leq \bigg[\bigg(\frac1{\|\xi_0 \|_{L^2}^4} - \frac{C_0^4}{4\pi^2\nu_1^4} \bigg) e^{8\pi^2\nu_1 t} \bigg]^{-1/4} \leq \bigg[\bigg(\frac1{R_0^4} - \frac{C_0^4}{4\pi^2\nu_1^4} \bigg) e^{8\pi^2\nu_1 t} \bigg]^{-1/4}. $$
The choice of $\nu$ implies that $2\pi^2 \nu_1^4 \geq C_0^4 R_0^4$, hence, for $t>0$,
  $$\|\xi_t \|_{L^2} \leq \bigg[ \frac1{2R_0^4}  e^{8\pi^2\nu_1 t} \bigg]^{-1/4} \leq 2R_0 e^{-2\pi^2\nu_1 t}. $$
As a consequence,
  $$\|\xi \|_{L^2(T-1, T; H)} = \bigg[\int_{T-1}^T \|\xi_t \|_{L^2}^2 \,\d t\bigg]^{1/2} \leq 2R_0 e^{-2\pi^2\nu_1 (T-1)} \leq \eps, $$
where the last step follows from the choice of $T$.

Now we consider the event
  $$\Omega_\eps= \Big\{\big\| \xi^N - \xi\big\|_{L^2(0,T; H)} \vee \big\| \xi^N - \xi\big\|_{C([0,T], H^{-\delta})} \leq \eps \Big\}. $$
Then \eqref{proof-cor.1} implies $\P(\Omega_\eps) \geq 1-\eps$. On the event $\Omega_\eps$, the triangle inequality yields
  \begin{equation}\label{eq-2}
  \big\| \xi^N \big\|_{L^2(T-1,T; H)} \leq \big\| \xi^N - \xi\big\|_{L^2(T-1,T; H)} + \| \xi\|_{L^2(T-1,T; H)} \leq \eps + \eps =2 \eps. \end{equation}
Recall that $\eps \leq (2\pi^2)^{1/4}/(2 C_0)$, which, together with \eqref{eq-2}, implies
  $$\big\| \xi^N \big\|_{L^2(T-1,T; H)} \leq \frac{(2\pi^2)^{1/4}}{C_0}. $$
This inequality holds for all $\omega\in \Omega_\eps$. As a result, for any $\omega\in \Omega_\eps$, there exists $t= t(\omega) \in [T-1, T]$ such that
  $$ \big\| \xi^N_{t(\omega)}(\omega) \big\|_{L^2} \leq \frac{(2\pi^2)^{1/4}}{C_0} .$$
Finally, applying Proposition \ref{4-prop} above to the equation \eqref{SNSE-N} with the initial condition $\xi^N_{t(\omega)}(\omega)$, we can conclude that the solution extends to all $t> t(\omega)$ for every $\omega\in \Omega_\eps$.
\end{proof}

\section{Appendix 1: convergence of $S_{\theta^N}(v)$} \label{sect-5}

Recall the definition \eqref{I-theta} of $S_\theta(v)$ in the introduction. The purpose of this section is to prove

\begin{theorem}\label{prop-extra-term}
Assume $\theta^N$ is given as in \eqref{theta-N-def}. Then for any smooth divergence free vector field $v:\T^3 \to \R^3$, the following limit holds in $L^2(\T^3,\R^3)$:
  $$\lim_{N\to \infty} S_{\theta^N}(v) = \frac{3\nu}5 \Delta v.$$
\end{theorem}

First, thanks to the equality \eqref{decomp}, it is sufficient to prove that, under the conditions of Theorem \ref{prop-extra-term},
  \begin{equation}\label{alternative-limit}
  \lim_{N\to \infty} S_{\theta^N}^\perp (v) = \frac{2\nu}5 \Delta v \quad \mbox{holds in } L^2(\T^3,\R^3),
  \end{equation}
where the operator $S_{\theta^N}^\perp$ is defined in \eqref{new-operator} (replacing $\theta$ by $\theta^N$). The reason for turning to the new quantity $S_\theta^\perp(v)$ is that we have simpler formulae for the operator $\Pi^\perp$ which is orthogonal to the Leray projection $\Pi$. If $X$ is a general vector field, then, formally,
  \begin{equation}\label{Leray-proj-1}
  \Pi^\perp X = \nabla \Delta^{-1} \div(X).
  \end{equation}
On the other hand, if $X= \sum_{l\in \Z^3_0} X_l e_l$, $X_l\in \mathbb C^3$, then
  \begin{equation}\label{Leray-proj-2}
  \Pi^\perp X= \sum_l \frac{l\cdot X_l}{|l|^2} l e_l = \nabla\bigg[ \frac1{2\pi {\rm i}} \sum_l \frac{l\cdot X_l}{|l|^2} e_l \bigg].
  \end{equation}

Now we assume the divergence free vector field $v$ has the Fourier expansion
  $$v= \sum_{l,\beta} v_{l,\beta} \sigma_{l,\beta}. $$
The coefficients $\{v_{l,\beta}: l\in \Z^3_0, \beta=1,2 \} \subset \mathbb C$ satisfy $\overline{v_{l,\beta}}= v_{-l,\beta}$. Indeed, the computations below do not require that $v$ is a real vector field.

\begin{lemma}\label{lem-extra-term}
We have
  \begin{equation}\label{lem-extra-term.1}
  S_\theta^\perp (v)= - \frac{6\pi^2 \nu}{\|\theta \|_{\ell^2}^2} \sum_{l,\beta} v_{l,\beta} \Pi\bigg\{ \bigg[ \sum_{k,\alpha} \theta_k^2 (a_{k,\alpha} \cdot l)^2 (a_{l,\beta}\cdot (k-l)) \frac{k-l}{|k-l|^2} \bigg] e_l \bigg\}.
  \end{equation}
\end{lemma}

\begin{proof}
We give two different proofs, using respectively \eqref{Leray-proj-2} and \eqref{Leray-proj-1}.

(1) We have
  $$\nabla v(x)= \sum_{l,\beta} v_{l,\beta} \nabla \sigma_{l,\beta}(x) = 2\pi {\rm i} \sum_{l,\beta} v_{l,\beta} (a_{l,\beta} \otimes l) e_l(x). $$
Note that $\sigma_{-k,\alpha}(x)= a_{k,\alpha} e_{-k}(x)$; thus
  $$(\sigma_{-k,\alpha}\cdot\nabla v)(x) = 2\pi {\rm i} \sum_{l,\beta} v_{l,\beta} (a_{k,\alpha} \cdot l) a_{l,\beta} e_{l-k}(x). $$
By the first equality in \eqref{Leray-proj-2} and using $a_{l,\beta} \cdot l=0$, we have
  \begin{equation}\label{lem-extra-term.2}
  \aligned
  \Pi^\perp (\sigma_{-k,\alpha}\cdot\nabla v)(x)&= 2\pi {\rm i} \sum_{l,\beta} v_{l,\beta} (a_{k,\alpha} \cdot l) (a_{l,\beta} \cdot (l-k)) \frac{l-k}{|l-k|^2} e_{l-k}(x)\\
  &= - 2\pi {\rm i} \sum_{l,\beta} v_{l,\beta} (a_{k,\alpha} \cdot l) (a_{l,\beta}\cdot k) \frac{l-k}{|l-k|^2} e_{l-k}(x).
  \endaligned
  \end{equation}
As a consequence,
  $$ \aligned
  &\, \big[ \sigma_{k,\alpha}\cdot \nabla \Pi^\perp (\sigma_{-k,\alpha}\cdot\nabla v) \big](x)\\
  =& - 2\pi {\rm i} \sum_{l,\beta} v_{l,\beta} (a_{k,\alpha} \cdot l) (a_{l,\beta}\cdot k) \frac{l-k}{|l-k|^2} e_k(x) a_{k,\alpha} \cdot\nabla e_{l-k}(x) \\
  =& - (2\pi {\rm i})^2 \sum_{l,\beta} v_{l,\beta} (a_{k,\alpha} \cdot l) (a_{l,\beta}\cdot k) \frac{l-k}{|l-k|^2} (a_{k,\alpha} \cdot (l-k))  e_k(x) e_{l-k}(x) \\
  =& -4\pi^2 \sum_{l,\beta} v_{l,\beta} (a_{k,\alpha} \cdot l)^2 (a_{l,\beta}\cdot k) \frac{k-l}{|k-l|^2} e_l(x).
  \endaligned $$
This immediately gives us the desired identity since $C_\nu^2= 3\nu/2$.

(2) In the second proof we use \eqref{Leray-proj-1}. Since $v$ is divergence free, we have $\div(\sigma_{-k,\alpha}\cdot\nabla v)= (\nabla\sigma_{-k,\alpha}): (\nabla v)^\ast$, where $:$ is the inner product of matrices and $\ast$ means (real) transposition. Therefore,
  $$\aligned
  \div(\sigma_{-k,\alpha}\cdot\nabla v) &= \big[-2\pi{\rm i} (a_{k,\alpha} \otimes k) e_{-k}(x)\big] : \bigg[ 2\pi {\rm i} \sum_{l,\beta} v_{l,\beta} (a_{l,\beta} \otimes l) e_l(x) \bigg]^\ast \\
  &= 4\pi^2 \sum_{l,\beta} v_{l,\beta} \big[ (a_{k,\alpha} \otimes k): (l \otimes a_{l,\beta} )\big] e_{l-k}(x) \\
  &= 4\pi^2 \sum_{l,\beta} v_{l,\beta} (a_{k,\alpha} \cdot l) (a_{l,\beta} \cdot k) e_{l-k}(x).
  \endaligned $$
This implies
  $$\Delta^{-1} \div(\sigma_{-k,\alpha}\cdot\nabla v)= - \sum_{l,\beta} v_{l,\beta} (a_{k,\alpha} \cdot l) (a_{l,\beta} \cdot k) \frac{e_{l-k}(x)}{|l-k|^2}, $$
and thus,
  $$\Pi^\perp(\sigma_{-k,\alpha}\cdot\nabla v)= \nabla \Delta^{-1} \div(\sigma_{-k,\alpha}\cdot\nabla v)= -2\pi {\rm i} \sum_{l,\beta} v_{l,\beta} (a_{k,\alpha} \cdot l) (a_{l,\beta} \cdot k) \frac{l-k}{|l-k|^2}e_{l-k}(x). $$
This coincides with \eqref{lem-extra-term.2}. The rest of the computations are the same as those in the first proof, so we omit them.
\end{proof}

\begin{corollary}\label{cor-extra}
Denote by $\angle_{k,l}$ the angle between the vectors $k$ and $l$. We have
  $$ S_\theta^\perp (v)= - \frac{6\pi^2 \nu}{\|\theta \|_{\ell^2}^2} \sum_{l,\beta} v_{l,\beta} |l|^2 \Pi\bigg\{ \bigg[ \sum_{k} \theta_k^2 \sin^2(\angle_{k,l}) (a_{l,\beta}\cdot (k-l)) \frac{k-l}{|k-l|^2} \bigg] e_l \bigg\}. $$
\end{corollary}

\begin{proof}
Recall that $\{\frac{k}{|k|}, a_{k,1}, a_{k,2}\}$ is an ONS of $\R^3$; we have
  $$\sum_{\alpha=1}^2 (a_{k,\alpha} \cdot l)^2 = |l|^2 - \bigg(l\cdot \frac{k}{|k|}\bigg)^2 = |l|^2 \bigg(1- \frac{(k\cdot l)^2}{|k|^2 |l|^2} \bigg) = |l|^2 \sin^2(\angle_{k,l}) . $$
Thus,
  $$\sum_{k,\alpha} \theta_k^2 (a_{k,\alpha} \cdot l)^2 (a_{l,\beta}\cdot (k-l)) \frac{k-l}{|k-l|^2} = |l|^2 \sum_{k} \theta_k^2 \sin^2(\angle_{k,l}) (a_{l,\beta}\cdot (k-l)) \frac{k-l}{|k-l|^2}. $$
Substituting this equality into \eqref{lem-extra-term.1} leads to the desired result.
\end{proof}

Recall the sequence $\theta^N \in \ell^2$ defined in \eqref{theta-N-def}. The next result is a crucial step for proving the limit \eqref{alternative-limit}.

\begin{proposition}\label{prop-2}
For any fixed $l\in \Z^3_0$ and $\beta\in \{1,2\}$,
  $$\lim_{N\to \infty} \frac{1}{\|\theta^N \|_{\ell^2}^2} \sum_{k} \big(\theta^N_k \big)^2 \sin^2(\angle_{k,l}) (a_{l,\beta}\cdot (k-l)) \frac{k-l}{|k-l|^2} = \frac4{15} a_{l,\beta}. $$
\end{proposition}

Suppose we have already proved this result; we now turn to prove \eqref{alternative-limit}.

\begin{proof}[Proof of \eqref{alternative-limit}]
By Corollary \ref{cor-extra}, for any $N\geq 1$,
  $$ S_{\theta^N}^\perp (v)= - 6\pi^2 \nu \sum_{l,\beta} v_{l,\beta} |l|^2 \Pi\bigg\{ \bigg[ \frac{1}{\|\theta^N \|_{\ell^2}^2}\sum_{k} \big(\theta^N_k \big)^2 \sin^2(\angle_{k,l}) (a_{l,\beta}\cdot (k-l)) \frac{k-l}{|k-l|^2} \bigg] e_l \bigg\}.$$
Since
  $$\frac{2\nu}5 \Delta v= -\frac{8\pi^2 \nu}5 \sum_{l,\beta} v_{l,\beta} |l|^2 a_{l,\beta} e_l $$
which is divergence free, we have
  $$\aligned
  &\, S_{\theta^N}^\perp (v) -\frac{2\nu}5 \Delta v \\
  =&\, - 6\pi^2 \nu \sum_{l,\beta} v_{l,\beta} |l|^2 \Pi\bigg\{ \bigg[ \frac{1}{\|\theta^N \|_{\ell^2}^2}\sum_{k} \big(\theta^N_k \big)^2 \sin^2(\angle_{k,l}) (a_{l,\beta}\cdot (k-l)) \frac{k-l}{|k-l|^2} - \frac 4{15} a_{l,\beta} \bigg] e_l \bigg\} .
  \endaligned $$

Fix any big $M>0$. We have
  $$\bigg\| S_{\theta^N}^\perp (v) -\frac{2\nu}5 \Delta v \bigg\|_{L^2} \leq K_{M,1} + K_{M,2},$$
where
  $$\aligned
  K_{M,1} \leq C \sum_{|l|\leq M,\beta} |v_{l,\beta}|\, |l|^2 \bigg| \frac{1}{\|\theta^N \|_{\ell^2}^2}\sum_{k} \big(\theta^N_k \big)^2 \sin^2(\angle_{k,l}) (a_{l,\beta}\cdot (k-l)) \frac{k-l}{|k-l|^2} - \frac 4{15} a_{l,\beta} \bigg|
  \endaligned$$
and
  $$\aligned
  K_{M,2} &\leq C \sum_{|l|> M,\beta} |v_{l,\beta}|\, |l|^2 \bigg| \frac{1}{\|\theta^N \|_{\ell^2}^2}\sum_{k} \big(\theta^N_k \big)^2 \sin^2(\angle_{k,l}) (a_{l,\beta}\cdot (k-l)) \frac{k-l}{|k-l|^2} - \frac 4{15} a_{l,\beta} \bigg|\\
  &\leq C \sum_{|l|> M,\beta} |v_{l,\beta}|\, |l|^2 \bigg(\frac{1}{\|\theta^N \|_{\ell^2}^2}\sum_{k} \big(\theta^N_k \big)^2  + \frac 4{15} \bigg) \leq 2C \sum_{|l|> M,\beta} |v_{l,\beta}|\, |l|^2.
  \endaligned $$
Since $M$ is fixed, Proposition \ref{prop-2} implies that $K_{M,1}$ vanishes as $N\to \infty$, hence
  $$\limsup_{N\to \infty} \bigg\| S_{\theta^N}^\perp (v) -\frac{2\nu}5 \Delta v \bigg\|_{L^2} \leq 2C \sum_{|l|> M,\beta} |v_{l,\beta}|\, |l|^2. $$
As the vector field $v$ is smooth, the coefficients $v_{l,\beta}$ decrease to 0 as $|l|\to \infty$ faster than any polynomials of negative order. Thus we complete the proof by letting $M\to \infty$.
\end{proof}

Next we prove Proposition \ref{prop-2} for which we need a simple preparation.

\begin{lemma}\label{lem-differ}
Fix $l\in \Z^3_0$. For all $k\in \Z^3_0$ with $|k|$ big enough, one has
  $$\bigg|\frac{(k-l)\otimes (k-l)}{|k-l|^2} - \frac{k\otimes k}{|k|^2}\bigg| \leq 4\frac{|l|}{|k|}. $$
\end{lemma}

\begin{proof}
We have
  $$\frac{(k-l)\otimes (k-l)}{|k-l|^2} - \frac{k\otimes k}{|k|^2} = \frac{k-l}{|k-l|} \otimes \bigg(\frac{k-l}{|k-l|} - \frac{k}{|k|} \bigg) + \bigg(\frac{k-l}{|k-l|} - \frac{k}{|k|} \bigg) \otimes \frac{k}{|k|},$$
and thus
  $$\bigg|\frac{(k-l)\otimes (k-l)}{|k-l|^2} - \frac{k\otimes k}{|k|^2}\bigg| \leq 2\bigg| \frac{k-l}{|k-l|} - \frac{k}{|k|} \bigg|.$$
Next, since
  $$\frac{k-l}{|k-l|} - \frac{k}{|k|}= \bigg(\frac1{|k-l|} - \frac1{|k|}\bigg)(k-l) - \frac{l}{|k|}, $$
one has
  $$\bigg| \frac{k-l}{|k-l|} - \frac{k}{|k|} \bigg| \leq \frac{\big| |k| -|k-l| \big|}{ |k|} + \frac{|l|}{|k|} \leq 2 \frac{|l|}{|k|}. $$
Summarizing the above estimates completes the proof.
\end{proof}

Now we are ready to provide the

\begin{proof}[Proof of Proposition \ref{prop-2}]
Note that, by Lemma \ref{lem-differ},
  $$\bigg|(a_{l,\beta}\cdot (k-l)) \frac{k-l}{|k-l|^2} -(a_{l,\beta}\cdot k) \frac{k}{|k|^2} \bigg| \leq \bigg|\frac{(k-l)\otimes (k-l)}{|k-l|^2} - \frac{k\otimes k}{|k|^2}\bigg| \leq 4\frac{|l|}{|k|} . $$
Recall the definition of $\theta^N$ in \eqref{theta-N-def}; then
  $$\aligned &\, \frac{1}{\|\theta^N \|_{\ell^2}^2} \sum_{k} \big(\theta^N_k \big)^2 \sin^2(\angle_{k,l})  \bigg|(a_{l,\beta}\cdot (k-l)) \frac{k-l}{|k-l|^2} -(a_{l,\beta}\cdot k) \frac{k}{|k|^2} \bigg| \\
  \leq & \, \frac{1}{\|\theta^N \|_{\ell^2}^2} \sum_{|k|\geq N} \big(\theta^N_k \big)^2 \times 4\frac{|l|}{|k|} \leq \frac{4|l|}{N} \to 0
  \endaligned $$
as $N\to \infty$. Therefore, it is sufficient to prove
  \begin{equation}\label{key-limit}
  \lim_{N\to \infty} \frac{1}{\|\theta^N \|_{\ell^2}^2} \sum_{k} \big(\theta^N_k \big)^2 \sin^2(\angle_{k,l}) (a_{l,\beta}\cdot k) \frac{k}{|k|^2} = \frac 4{15} a_{l,\beta}.
  \end{equation}

\begin{lemma}\label{lem-proof}
Let $\theta^N$ be given as in \eqref{theta-N-def}. We have
  \begin{equation}\label{proof.1}
  \aligned
  &\, \lim_{N\to \infty} \frac{1}{\|\theta^N \|_{\ell^2}^2} \sum_{k} \big(\theta^N_k \big)^2 \sin^2(\angle_{k,l}) (a_{l,\beta}\cdot k) \frac{k}{|k|^2} \\
  =&\, \lim_{N\to \infty} \frac{1}{\|\theta^N \|_{\ell^2}^2} \int_{\{N\leq |x|\leq 2N\}} \frac1{|x|^{2\gamma}} \sin^2(\angle_{x,l}) (a_{l,\beta}\cdot x) \frac{x}{|x|^2} \,\d x.
  \endaligned
  \end{equation}
\end{lemma}

We postpone the proof of Lemma \ref{lem-proof} and continue proving Proposition \ref{prop-2}. Let $J_\beta(N)$ be the quantity on the right hand side of \eqref{proof.1}, which is a vector in $\R^3$. To compute $J_\beta(N)$, we consider the new coordinate system $(y_1, y_2, y_3)$ in which the coordinate axes are $a_{l,1}, a_{l,2}$ and $\frac{l}{|l|}$, respectively. Let $U$ be the orthogonal transformation matrix: $x=Uy$. For $i\in \{1,2,3\}$, let ${\rm e}_i\in \R^3$ be such that ${\rm e}_{i,j}= \delta_{i,j}$, $1\leq j\leq 3$. We have
  $$a_{l,i}= U {\rm e}_i\, (i=1,2)\quad \mbox{and} \quad \frac{l}{|l|} = U {\rm e}_3.$$
Now $\angle_{x,l} = \angle_{Uy,U{\rm e}_3} = \angle_{y,{\rm e}_3}$ and
  \begin{equation}\label{proof.2}
  \aligned
  J_\beta(N) &= \frac{1}{\|\theta^N \|_{\ell^2}^2} \int_{\{N\leq |y|\leq 2N\}} \frac1{|y|^{2\gamma}} \sin^2(\angle_{y,{\rm e}_3})\, (U{\rm e}_\beta \cdot Uy) \frac{Uy}{|y|^2} \,\d y \\
  &= U \bigg[ \frac{1}{\|\theta^N \|_{\ell^2}^2} \int_{\{N\leq |y|\leq 2N\}} \frac1{|y|^{2\gamma}} \sin^2(\angle_{y,{\rm e}_3})\, \frac{y_\beta y}{|y|^2} \,\d y \bigg] .
  \endaligned
  \end{equation}

We denote $\tilde J_\beta(N)$ the term in the square bracket in \eqref{proof.2}, i.e. $\tilde J_\beta(N) = U^\ast J_\beta(N) \in \R^3$. By symmetry argument, we see that
  \begin{equation}\label{proof.3}
  \tilde J_{\beta, i}(N) = \frac{1}{\|\theta^N \|_{\ell^2}^2} \int_{\{N\leq |y|\leq 2N\}} \frac1{|y|^{2\gamma}} \sin^2(\angle_{y,{\rm e}_3})\, \frac{y_\beta y_i}{|y|^2} \,\d y =0, \quad i\in \{1,2,3\}\setminus \{\beta \}.
  \end{equation}
This can also be directly computed by using the spherical coordinates below.

Next, we compute $\tilde J_{\beta ,\beta}\, (\beta=1,2)$ by changing the variables into the spherical coordinate system:
  $$\begin{cases}
  y_1= r\sin \psi \cos \varphi, \\
  y_2 = r\sin \psi \sin \varphi, \\
  y_3= r\cos \psi,
  \end{cases} \quad N\leq r\leq 2N,\, 0\leq\psi \leq\pi, 0\leq \varphi< 2\pi .$$
In this system, $\angle_{y,{\rm e}_3} = \psi$. We have
  $$\aligned
  \tilde J_{1,1}(N) &= \frac{1}{\|\theta^N \|_{\ell^2}^2} \int_N^{2N} \d r \int_0^\pi \d\psi \int_0^{2\pi} \d \varphi\, \frac1{r^{2\gamma}} (\sin^2 \psi) (\sin\psi \cos\varphi)^2 \, r^2 \sin\psi \\
  &= \frac{1}{\|\theta^N \|_{\ell^2}^2} \int_N^{2N}\frac{\d r}{r^{2\gamma -2}} \int_0^\pi \sin^5 \psi\, \d\psi \int_0^{2\pi} \cos^2\varphi\,  \d \varphi .
  \endaligned $$
Note that $\int_0^{2\pi} \cos^2\varphi\,  \d \varphi = \int_0^{2\pi} \frac12 (1+ \cos 2\varphi)\,  \d \varphi = \pi $ and
  $$\aligned
  \int_0^\pi \sin^5 \psi\, \d \psi &= - \int_0^\pi (1-\cos^2\psi )^2 \, \d \cos\psi = - \int_0^\pi \big(1-2\cos^2\psi + \cos^4\psi \big) \, \d \cos\psi \\
  &= - \bigg(\cos\psi -\frac23 \cos^3\psi + \frac15 \cos^5\psi \bigg)\bigg|_0^{\pi} = \frac{16}{15}.
  \endaligned $$
Thus
  \begin{equation}\label{proof.4}
  \tilde J_{1,1}(N) = \frac{16}{15} \pi \times \frac{1}{\|\theta^N \|_{\ell^2}^2} \int_N^{2N}\frac{\d r}{r^{2\gamma -2}} .
  \end{equation}
Following the proof of Lemma \ref{lem-proof} (it is much simpler here since the function $g$ can be taken identically 1), one can show
  \begin{equation}\label{proof.5}
  \bigg|\sum_{k} \big(\theta^N_k \big)^2 - \int_{\{N\leq |x|\leq 2N\}} \frac{\d x}{|x|^{2\gamma}} \bigg|\leq \frac{C}{N} \big\|\theta^N \big\|_{\ell^2}^2
  \end{equation}
for some constant $C>0$. Equivalently,
  $$\bigg| \big\|\theta^N \big\|_{\ell^2}^2 - 4\pi \int_N^{ 2N} \frac{\d r}{r^{2\gamma -2}} \bigg|\leq \frac{C}{N} \big\|\theta^N \big\|_{\ell^2}^2, $$
which implies
  $$\bigg| \frac1{4\pi} - \frac{1}{\|\theta^N \|_{\ell^2}^2} \int_N^{ 2N} \frac{\d r}{r^{2\gamma -2}} \bigg|\leq \frac{C}{N }.$$
Recalling \eqref{proof.4}, we obtain $\lim_{N\to \infty} \tilde J_{1,1}(N) = \frac 4{15}$, which, combined with \eqref{proof.3}, implies
  $$\lim_{N\to \infty} \tilde J_1(N) = \frac 4{15} {\rm e}_1.$$
Therefore, by \eqref{proof.2},
  $$\lim_{N\to \infty} J_1(N) = \lim_{N\to \infty} U \tilde J_1(N) = \frac 4{15} U {\rm e}_1 =\frac 4{15} a_{l,1} .$$

Similarly,
  $$\aligned
  \tilde J_{2,2}(N)&= \frac{1}{\|\theta^N \|_{\ell^2}^2} \int_N^{2N} \d r \int_0^\pi \d \psi \int_0^{2\pi} \d \varphi\, \frac1{r^{2\gamma}} (\sin^2 \psi) (\sin\psi \sin\varphi)^2 \, r^2 \sin\psi = \tilde J_{1,1}(N) \to \frac 4{15},
  \endaligned $$
and thus $\lim_{N\to \infty} \tilde J_2(N) = \frac 4{15} {\rm e}_2$. As a result,
  $$\lim_{N\to \infty} J_2(N) = \lim_{N\to \infty} U \tilde J_2(N) = \frac 4{15} U {\rm e}_2 =\frac 4{15} a_{l,2} .$$
Combining these two results with \eqref{proof.1}, we obtain \eqref{key-limit}.
\end{proof}

Now we provide the

\begin{proof}[Proof of Lemma \ref{lem-proof}]
We define the function
  $$g(x) = \sin^2(\angle_{x,l}) (a_{l,\beta}\cdot x) \frac{x}{|x|^2}, \quad x\in \R^3,\, x\neq 0. $$
Clearly, $\|g\|_\infty \leq 1$. We shall prove that
  \begin{equation}\label{proof-lem.1}
  \aligned
  \bigg|\sum_{k} \big(\theta^N_k \big)^2 g(k) -\int_{\{N\leq |x|\leq 2N\}} \frac{g(x)}{|x|^{2\gamma}}  \,\d x \bigg| \leq \frac{C}N \big\|\theta^N \big\|_{\ell^2}^2.
  \endaligned
  \end{equation}
Let $\square(k)$ be the unit cube centered at $k\in \Z^3$ such that all sides have length 1 and are parallel to the axes. Note that for all $k,l\in \Z^3$, $k\neq l$, the interiors of $\square(k)$ and $\square(l)$ are disjoint. Let $S_N = \bigcup_{N\leq |k|\leq 2N} \square(k)$; then,
  $$\bigg|\sum_{k} \big(\theta^N_k \big)^2 g(k) - \int_{S_N} \frac{g(x)}{|x|^{2\gamma}} \,\d x\bigg| \leq \sum_{N\leq |k|\leq 2N} \int_{\square(k)} \bigg|\frac{g(k)}{|k|^{2\gamma}} - \frac{g(x)}{|x|^{2\gamma}} \bigg|\, \d x.$$
It holds that, for all $|k|$ big enough and $x\in \square(k)$,
  $$\bigg|\frac{g(k)}{|k|^{2\gamma}} - \frac{g(x)}{|x|^{2\gamma}} \bigg| \leq \bigg|\frac{1}{|k|^{2\gamma}} - \frac{1}{|x|^{2\gamma}} \bigg| + \frac{|g(k) -g(x)|}{|x|^{2\gamma}} \leq C\bigg(\frac{1}{|k|^{2\gamma+1}} + \frac{|g(k) -g(x)|}{|k|^{2\gamma}}\bigg) .$$
Next,
  $$\aligned
  |g(k) -g(x)| &\leq |\sin^2(\angle_{k,l}) -\sin^2(\angle_{x,l})| + \bigg| (a_{l,\beta}\cdot k) \frac{k}{|k|^2} - (a_{l,\beta}\cdot x) \frac{x}{|x|^2}\bigg| \\
  &\leq 2|\sin(\angle_{k,l}) -\sin(\angle_{x,l})| + \bigg| \frac{k\otimes k}{|k|^2} - \frac{x\otimes x}{|x|^2}\bigg| \\
  &\leq 2| \angle_{k,l} -\angle_{x,l}| + 2 \bigg| \frac{k}{|k|} - \frac{x}{|x|}\bigg|.
  \endaligned $$
Since $|x-k|\leq 1$ and $|k|\geq N \gg 1$, one has
  $$|g(k) -g(x)| \leq \frac{C}{|k|}. $$
Summarizing the above discussions, we obtain
  $$\bigg|\sum_{k} \big(\theta^N_k \big)^2 g(k) - \int_{S_N} \frac{g(x)}{|x|^{2\gamma}} \,\d x\bigg| \leq \sum_{N\leq |k|\leq 2N} \int_{\square(k)} \frac{C}{|k|^{2\gamma+1}}\,\d x \leq \frac CN \sum_{N\leq |k|\leq 2N} \frac{1}{|k|^{2\gamma}} = \frac CN \big\|\theta^N \big\|_{\ell^2}^2. $$

Note that there is a small difference between the sets $\{N\leq |x|\leq 2N\}$ and $S_N$, but, in the same way, one can show that
  $$\bigg|\int_{\{N\leq |x|\leq 2N\}} \frac{g(x)}{|x|^{2\gamma}} \,\d x - \int_{S_N} \frac{g(x)}{|x|^{2\gamma}} \,\d x\bigg| \leq \frac CN \big\|\theta^N \big\|_{\ell^2}^2.$$
Indeed, for any $x\in \square(k)$ with $N\leq |k| \leq 2N$, one has $N-1 \leq |x| \leq 2N+1$. Therefore,
  $$S_N = \bigcup_{N\leq |k| \leq 2N} \square(k) \subset \{N-1 \leq |x| \leq 2N+1 \} =: T_N. $$
One also has
  $$R_N:= \{N+1 \leq |x| \leq 2N-1 \} \subset S_N.$$
Denote by $A\Delta B$ the symmetric difference of sets $A,B\subset \R^3$; then,
  $$\aligned
  &\, \bigg|\int_{\{N\leq |x|\leq 2N\}} \frac{g(x)}{|x|^{2\gamma}} \,\d x - \int_{S_N} \frac{g(x)}{|x|^{2\gamma}} \,\d x\bigg| = \bigg|\int_{S_N\Delta \{N\leq |x|\leq 2N\}} \frac{g(x)}{|x|^{2\gamma}} \,\d x  \bigg| \\
  \leq &\, \int_{S_N\Delta \{N\leq |x|\leq 2N\}} \frac{1}{|x|^{2\gamma}} \,\d x \leq \int_{T_N\setminus R_N} \frac{1}{|x|^{2\gamma}} \,\d x \leq \frac{C}{N^{2\gamma-2}} \leq \frac{C}{N} \big\|\theta^N \big\|_{\ell^2}^2,
  \endaligned $$
where the last step follows from
  $$\aligned
  \big\|\theta^N \big\|_{\ell^2}^2 &= \sum_{N\leq |k| \leq 2N} \frac1{|k|^{2\gamma}} \geq \frac1{(2N)^{2\gamma}}\, \#\{k\in \Z^3_0: N\leq |k| \leq 2N \} \geq \frac{C}{N^{2\gamma -3}}.
  \endaligned $$
The proof is complete.
\end{proof}

\begin{remark}\label{rem-generic-theta}
Assume $\gamma\in [0,3/2]$. For any $N\in \N$, we define
  $$\theta^N_k = \frac1{|k|^\gamma} {\bf 1}_{\{|k|\leq N\}}, \quad k\in \Z^3_0.$$
Then $\|\theta^N\|_{\ell^\infty} =1$ and $\|\theta^N\|_{\ell^2} \to \infty$ as $N\to \infty$. Thus the sequence $\{\theta^N \}_{N\in \N}$ satisfies the property \eqref{theta-N}. With suitable modifications of the proofs in this section, we can still prove Theorem \ref{prop-extra-term}. Indeed, the arguments above the proof of Proposition \ref{prop-2} remain the same. To prove Proposition \ref{prop-2}, we fix $M\in \N$; then for all $N> M$,
  $$\aligned &\, \frac{1}{\|\theta^N \|_{\ell^2}^2} \sum_{k} \big(\theta^N_k \big)^2 \sin^2(\angle_{k,l})  \bigg|(a_{l,\beta}\cdot (k-l)) \frac{k-l}{|k-l|^2} -(a_{l,\beta}\cdot k) \frac{k}{|k|^2} \bigg| \\
  \leq & \, \frac{1}{\|\theta^N \|_{\ell^2}^2} \sum_{|k|\leq M} 2 \big(\theta^N_k \big)^2 + \frac{1}{\|\theta^N \|_{\ell^2}^2} \sum_{|k|> M} \big(\theta^N_k \big)^2 \times 4\frac{|l|}{|k|} \\
  \leq &\, C_M \frac{\|\theta^N \|_{\ell^\infty}^2}{\|\theta^N \|_{\ell^2}^2} + 4\frac{|l|}{M}.
  \endaligned $$
First letting $N\to \infty$ and then $M\to \infty$ we see that it is sufficient to prove the limit \eqref{key-limit}. In the subsequent proofs, similar modifications work as well and we can complete the proof of Theorem \ref{prop-extra-term}.
\end{remark}

\section{Appendix 2: the difficulty with the advection noise}

In this part we do some formal computations to illustrate why we cannot deal with 3D Navier-Stokes equations \eqref{stoch NS} with the full advection noise. Using our vector fields $\{\sigma_{k,\alpha}: k\in \Z^3_0, \alpha =1,2\}$, the equations can be written as
  \begin{equation*}
  \d \xi + \L_u \xi\,\d t = \Delta \xi\,\d t + \frac{C_\nu}{\|\theta\|_{\ell^2}} \sum_{k,\alpha} \theta_k \L_{\sigma_{k,\alpha}} \xi \circ \d W^{k,\alpha}_t,
  \end{equation*}
where, as usual, $u$ is related to $\xi$ via the Biot-Savart law. It has the It\^o formulation
  \begin{equation*}
  \aligned
  \d \xi + \L_{u} \xi\,\d t &= \Delta \xi\,\d t + \frac{C_\nu}{\|\theta \|_{\ell^2}} \sum_{k,\alpha} \theta_k \L_{\sigma_{k,\alpha}} \xi \, \d W^{k,\alpha}_t + \frac{C_\nu^2}{\|\theta \|_{\ell^2}^2} \sum_{k,\alpha} \theta_k^2 \L_{\sigma_{k,\alpha}} \big( \L_{\sigma_{-k,\alpha}} \xi \big) \, \d t.
  \endaligned
  \end{equation*}
By Proposition \ref{prop-key-identity} below, this equation can be reduced to
  \begin{equation}\label{3D-NSE-advection-Ito}
  \d \xi + \L_{u} \xi\,\d t = (1+\nu) \Delta \xi\,\d t + \frac{C_\nu}{\|\theta \|_{\ell^2}} \sum_{k,\alpha} \theta_k \L_{\sigma_{k,\alpha}} \xi \, \d W^{k,\alpha}_t.
  \end{equation}

\begin{proposition}\label{prop-key-identity}
It holds that
  $$\sum_{k,\alpha} \theta_k^2 \L_{\sigma_{k,\alpha}} \big( \L_{\sigma_{-k,\alpha}} \xi \big) = \frac23\|\theta \|_{\ell^2} \Delta \xi. $$
\end{proposition}

\begin{proof}
First, for any $k\in \Z^3_0$, we have
  \begin{equation}\label{lem-1.2}
  \xi \cdot \nabla \sigma_{k,\alpha}= 2\pi{\rm i} (\xi \cdot k) \sigma_{k,\alpha}, \quad \alpha=1,2.
  \end{equation}
Thus,
  \begin{equation}\label{lem-1.2.0}
  \L_{\sigma_{k,\alpha}} \xi = \sigma_{k,\alpha}\cdot \nabla \xi - 2\pi{\rm i} (k\cdot \xi) \sigma_{k,\alpha}, \quad \alpha=1,2.
  \end{equation}
Next we prove that for any $k\in \Z^3_0$ and $\alpha=1,2$,
  \begin{equation}\label{lem-1.2.1}
  \L_{\sigma_{k,\alpha}} \big(\L_{\sigma_{-k,\alpha}} \xi \big)= {\rm Tr}\big[(a_{k,\alpha} \otimes a_{k,\alpha}) \nabla^2 \xi \big] .
  \end{equation}
The desired equality follows immediately from this fact and \eqref{proof.0}.

We have
  $$\L_{\sigma_{k,\alpha}} \big(\L_{\sigma_{-k,\alpha}} \xi \big) =  \sigma_{k,\alpha}\cdot \nabla \big(\L_{\sigma_{-k,\alpha}} \xi\big) - \big(\L_{\sigma_{-k,\alpha}} \xi\big)\cdot \nabla \sigma_{k,\alpha} =: I_1 -I_2 .$$
By \eqref{lem-1.2.0},
  $$I_1= \sigma_{k,\alpha}\cdot \nabla \big( \sigma_{-k,\alpha}\cdot \nabla \xi + 2\pi{\rm i} (k\cdot \xi) \sigma_{-k,\alpha}\big). $$
The definition \eqref{vector-fields} of $\sigma_{k,\alpha}$ leads to
  \begin{equation}\label{proof.11}
  \sigma_{k,\alpha}\cdot \nabla \sigma_{k,\alpha}= \sigma_{k,\alpha}\cdot \nabla \sigma_{-k,\alpha} =0, \quad k\in \Z^3_0.
  \end{equation}
Therefore,
  $$\aligned
  I_1 &= {\rm Tr}\big[(\sigma_{k,\alpha} \otimes \sigma_{-k,\alpha}) \nabla^2 \xi \big]+ 2\pi{\rm i} \big[\sigma_{k,\alpha}\cdot \nabla (k\cdot \xi) \big] \sigma_{-k,\alpha} \\
  &= {\rm Tr}\big[(a_{k,\alpha} \otimes a_{k,\alpha}) \nabla^2 \xi \big]+ 2\pi{\rm i} \big[k\cdot ( a_{k,\alpha}\cdot \nabla\xi) \big] a_{k,\alpha}.
  \endaligned $$
Next, by \ref{lem-1.2.0} and \eqref{proof.11},
  $$I_2 = \big( \sigma_{-k,\alpha}\cdot \nabla \xi + 2\pi{\rm i} (k\cdot \xi) \sigma_{-k,\alpha} \big) \cdot \nabla \sigma_{k,\alpha} = (\sigma_{-k,\alpha}\cdot \nabla \xi) \cdot \nabla \sigma_{k,\alpha}. $$
Replacing $\xi$ in \eqref{lem-1.2} by $\sigma_{-k,\alpha}\cdot \nabla \xi$ yields
  $$I_2= 2\pi{\rm i} \big((\sigma_{-k,\alpha}\cdot \nabla \xi) \cdot k \big)\, \sigma_{k,\alpha}= 2\pi{\rm i} \big[k\cdot ( a_{k,\alpha}\cdot \nabla\xi) \big] a_{k,\alpha}. $$
Summarizing the above computations we obtain the equality \eqref{lem-1.2.1}.
\end{proof}

We want to find an a priori estimate for the solution to \eqref{3D-NSE-advection-Ito} with some heuristic computations below. By the It\^o formula,
  \begin{equation}\label{SNSE-vort.2}
  \aligned
  \d \|\xi \|_{L^2}^2 &= -2\<\xi, \L_u \xi\>_{L^2}\,\d t - 2(1+\nu) \| \nabla \xi\|_{L^2}^2 \,\d t + \frac{ 2 C_\nu}{\|\theta\|_{\ell^2}} \sum_{k, \alpha} \theta_k \big\<\xi, \L_{\sigma_{k,\alpha}} \xi \big\>_{L^2} \, \d W^{k,\alpha}_t \\
  &\quad + \frac{2C_\nu^2}{\|\theta \|_{\ell^2}^2} \sum_{k, \alpha} \theta_k^2 \big\| \L_{\sigma_{k,\alpha}} \xi \big\|_{L^2}^2 \,\d t.
  \endaligned
  \end{equation}
First, it is not difficult (cf. the proof of \eqref{estim-nonlinear}) to show that
  \begin{equation}\label{estim-nonlinear.1}
  |\<\xi, \L_u \xi\>_{L^2}| \leq \frac12 \|\nabla\xi \|_{L^2}^{2} + C \|\xi \|_{L^2}^{6}.
  \end{equation}
Next, we denote
  $$\d M(t): = \frac{2C_\nu}{\|\theta\|_{\ell^2}} \sum_{k, \alpha} \theta_k \big\<\xi, \L_{\sigma_{k,\alpha}} \xi \big\>_{L^2} \, \d W^{k,\alpha}_t= - \frac{2C_\nu}{\|\theta\|_{\ell^2}} \sum_{k,\alpha} \theta_k \big\<\xi, \xi\cdot \nabla \sigma_{k,\alpha} \big\>_{L^2} \, \d W^{k,\alpha}_t $$
the martingale part and
  $$J(t)= \frac{2C_\nu^2}{\|\theta \|_{\ell^2}^2} \sum_{k,\alpha} \theta_k^2 \big\| \L_{\sigma_{k,\alpha}} \xi \big\|_{L^2}^2 =\frac{3\nu}{\|\theta \|_{\ell^2}^2} \sum_{k,\alpha} \theta_k^2 \big\| \L_{\sigma_{k,\alpha}} \xi \big\|_{L^2}^2 .$$
Then, since $\<\xi, \Delta\xi\>_{L^2}= - \|\nabla\xi \|_{L^2}^2$, we obtain from \eqref{SNSE-vort.2} and \eqref{estim-nonlinear.1} that
  \begin{equation}\label{estim-vort}
  \d \|\xi \|_{L^2}^2 \leq -(1+2\nu) \|\nabla\xi \|_{L^2}^2 \,\d t+ C \|\xi \|_{L^2}^{6}\,\d t + \d M(t) + J(t) \,\d t.
  \end{equation}
Now we compute the term $J(t)$.

\begin{lemma}\label{lem-quadratic-vort}
It holds that
  $$J(t)= 2\nu \|\nabla \xi\|_{L^2}^2 +  4\nu\pi^2 \frac{\|\theta \|_{h^1}^2}{\|\theta \|_{\ell^2}^2} \|\xi \|_{L^2}^2,$$
where
  $$\|\theta \|_{h^1}^2 = \sum_{k\in \Z_0^3}\theta_k^2 |k|^2.$$
\end{lemma}

\begin{proof}
We split $J(t)$ as $J(t)= \sum_{i=1}^3 J_i(t)$, where
  $$\aligned
  J_1(t)&= \frac{3\nu}{\|\theta \|_{\ell^2}^2} \sum_{k, \alpha} \theta_k^2 \big\| \sigma_{k,\alpha}\cdot\nabla \xi \big\|_{L^2}^2, \quad J_2(t)= \frac{3\nu}{\|\theta \|_{\ell^2}^2} \sum_{k, \alpha} \theta_k^2 \big\| \xi\cdot\nabla\sigma_{k,\alpha} \big\|_{L^2}^2, \\
  J_3(t)&= - \frac{3\nu}{\|\theta \|_{\ell^2}^2} \sum_{k, \alpha} \theta_k^2 \big(\<\sigma_{k,\alpha}\cdot\nabla \xi, \xi\cdot\nabla\sigma_{-k,\alpha} \>_{L^2} +\<\sigma_{-k,\alpha}\cdot\nabla \xi, \xi\cdot\nabla\sigma_{k,\alpha} \>_{L^2} \big).
  \endaligned$$
Similarly as the proof of \eqref{proof.0}, we have
  $$J_1(t)= 2\nu \|\nabla \xi\|_{L^2}^2. $$
Next, by \eqref{lem-1.2},
  $$\big\| \xi\cdot\nabla\sigma_{k,\alpha} \big\|_{L^2}^2 = 4\pi^2 \int_{\T^3} \big| (\xi\cdot k) \sigma_{k,\alpha} \big|^2\,\d x = 4\pi^2 \int_{\T^3} (\xi\cdot k)^2 \,\d x.  $$
Thus,
  $$\aligned
  J_2(t) &= \frac{3\nu}{\|\theta \|_{\ell^2}^2} \sum_{k\in \Z^3_0} \theta_k^2 \times 4\pi^2 \int_{\T^3} (\xi\cdot k)^2\,\d x  = \frac{12\nu\pi^2}{\|\theta \|_{\ell^2}^2} \sum_{k\in \Z^3_0} \theta_k^2 \int_{\T^3} (\xi\cdot k)^2\,\d x.
  \endaligned $$
Note that $(\xi\cdot k)^2 = \sum_{i, j=1}^3 k_i k_j \xi_i \xi_j$ and (cf. the computations below \eqref{proof.0})
  $$\sum_{k\in \Z^3_0} \theta_k^2 k_i k_j= \begin{cases}
  0, & i \neq j; \\
  \frac13 \sum_{k\in \Z^3_0} \theta_k^2 |k|^2= \frac13 \|\theta \|_{h^1}^2, & i= j.
  \end{cases} $$
Therefore,
  $$J_2(t)= \frac{12\nu\pi^2}{\|\theta \|_{\ell^2}^2} \sum_{i=1}^3 \frac13 \|\theta \|_{h^1}^2 \int_{\T^3} \xi_i^2\,\d x= 4\nu\pi^2 \frac{\|\theta \|_{h^1}^2}{\|\theta \|_{\ell^2}^2} \|\xi \|_{L^2}^2. $$

Finally, by \eqref{lem-1.2} and the definition of the vector fields $\sigma_{k,\alpha}$, we have
  $$ \<\sigma_{k,\alpha}\cdot\nabla \xi, \xi\cdot\nabla\sigma_{-k,\alpha} \>_{L^2} = -2\pi{\rm i} \int_{\T^3} (\xi \cdot k) (a_{k,\alpha}\cdot\nabla \xi) \cdot a_{k,\alpha} \,\d x.$$
In the same way,
  $$\<\sigma_{-k,\alpha}\cdot\nabla \xi, \xi\cdot\nabla\sigma_{k,\alpha} \>_{L^2}= 2\pi{\rm i} \int_{\T^3} (\xi \cdot k) (a_{k,\alpha}\cdot\nabla \xi) \cdot a_{k,\alpha} \,\d x. $$
Hence $J_3(t)$ vanishes. Summarizing these arguments we complete the proof.
\end{proof}

Therefore, the inequality \eqref{estim-vort} reduces to
  $$\d \|\xi \|_{L^2}^2 \leq - \|\nabla\xi \|_{L^2}^2 \,\d t+ C \|\xi \|_{L^2}^{6}\,\d t + \d M(t) + 4\nu\pi^2 \frac{\|\theta \|_{h^1}^2}{\|\theta \|_{\ell^2}^2} \|\xi \|_{L^2}^2 \,\d t. $$
The ratio $\frac{\|\theta \|_{h^1}^2}{\|\theta \|_{\ell^2}^2}$ spoils the a priori estimate, since the sequence $\{\theta^N \}_{N\geq 1}$ we take in our limit process has always the property
  $$\lim_{N\to \infty} \frac{\|\theta^N \|_{h^1}^2}{\|\theta^N \|_{\ell^2}^2} =\infty. $$

\section{Appendix 3: an incomplete attempt to motivate transport noise} \label{appendix-3}

We advise the reader that the argument given in this section is a sort of
cartoon based on imagination, and a potentially rigorous scaling limit behind
it would be presumably much more intricate than what is explained, or maybe
even impossible.

A fact, rigorous in several function spaces, is that given two vector fields
$A,B$ in $\mathbb{R}^{3}$, the condition%
\begin{equation}
\mathcal{L}_{A}B=\Pi\left(  A\cdot\nabla B\right)  \label{reduction}%
\end{equation}
is equivalent to
\[
B\cdot\nabla A=\nabla q
\]
for some scalar function $q$; the particular case when $\nabla q=0$ is implied
by a ``2D structure''
\begin{equation}
B\left(  x\right)  =b\left(  x\right)  e,\quad A\left(  x\right)  =A\left(
\pi_{e^{\perp}}x\right)  \label{2D structure}%
\end{equation}
where $e$ is a given unitary vector, $b\left(  x\right)  $ is a scalar
function on $\mathbb{R}^{3}$ (hence the vector field $B$ points always in the
direction $e$) and the improper notation $A\left(  x\right)  =A\left(
\pi_{e^{\perp}}x\right)  $ means that $A$ depends only on the projection of
$x$ on the plane orthogonal to $e$ (namely $A$ is independent of the
coordinate along $e$; this implies that the directional derivative of $A$ in
the direction $e$ is zero, which is precisely $B\cdot\nabla A=0$). What we
describe below is a sort of \textit{local 2D structure}, with different
orientations $e$ at different points, in which the identity (\ref{reduction})
could be approximately satisfied.

Assume to observe a fluid where the vorticity field $\xi$ is made of two
components%
\[
\xi=\xi_{L}+\xi_{S}%
\]
where the large-scale component $\xi_{L}$ is the sum of slowly varying
smoothed vortex filaments $\xi_{L}^{i}$%
\[
\xi_{L}=\sum_{i}\xi_{L}^{i}%
\]
and the small-scale component $\xi_{S}$ is a fast-varying field. By smoothed
vortex filament we mean a vortex structure strongly concentrated along a
vortex line;\ in the spirit of this cartoon we do not give any precise
definition, but vortex filaments, although extremely difficult to define and
describe, are commonly observed structures in complex fluids (see \cite{Vincent Meneguzzi}). We
need to qualify the filaments as smoothed because viscosity does not allow for
idealized filaments concentrated over lines. Corresponding to the vorticity
fields there are velocity fields obtained by Biot-Savart law, $u=u_{L}+u_{S}$.

Consider a point $x_{0}$ close to the core of a smoothed vortex filament
$\xi_{L}^{i}$, consider a neighbourhood $\mathcal{U}\left(  x_{0}\right)  $ of
$x_{0}$ and imagine a blow-up, a scaling such that we observe $\mathcal{U}%
\left(  x_{0}\right)  $ as if it were the full space. If the vortex filaments
are sufficiently thin, separated, regular and slowly moving compared to the
fast component $u_{S}$, in $\mathcal{U}\left(  x_{0}\right)  $ (which now
looks as the entire space) the vorticity is very close to zero everywhere
except along the line spanned by the vector $e=\xi_{L}^{i}\left(
x_{0}\right)  $; moreover, we may think to consider the full system on a time
scale where the large-scale objects $\mathcal{U}\left(  x_{0}\right)  $,
$\xi_{L}^{i}\left(  x_{0}\right)  $ etc. do not change while the small-scale
objects $\xi_{S},u_{S}$ change. The local picture of the small-scale fluid
$u_{S}$ in $\mathcal{U}\left(  x_{0}\right)  $ is thus of a 3D fluid subject
to a constant strong rotation around the vector $e$. If such a fluid, namely
$u_{S}|_{\mathcal{U}\left(  x_{0}\right)  }$, would be isolated from any other
input and interaction, it would become approximatively averaged in the
direction $e$, like the field $A$ in (\ref{2D structure}). This has been
rigorously proved in several works, see for instance
\cite{BabinMahalovNikolaenko} (see also \cite{FlaMah} in a stochastic framework).
Obviously we do not mean that the global field $u_{S}$ is almost
two-dimensional: only at local level it has a tendency to average in the
direction of $\xi_{L}^{i}\left(  x_{0}\right)  $; this vector changes
orientation from a small region to another. When this happens, we have
$\xi_{L}^{i}\left(  x_{0}\right)  \cdot\nabla u_{S}\left(  x_{0}\right)
\sim0$. We have argued in the proximity of a vortex core; far from filaments
$\xi_{L}^{i}\left(  x_{0}\right)  \cdot\nabla u_{S}\left(  x_{0}\right)  $ is
small just because $\xi_{L}$ is almost zero by itself. We deduce that
everywhere%
\begin{equation}
\xi_{L}\left(  x\right)  \cdot\nabla u_{S}\left(  x\right)  \sim
0.\label{stratif}%
\end{equation}
We ignore whether it is possible to establish a more rigorous derivation of
such a fact by a proper scaling limit and maybe an argument similar to the
concept of local equilibrium in the statistical mechanics of particle systems,
where the local convergence to equilibrium is replaced by the
``vertical averaging'' property described above.

Let us derive a consequence of (\ref{stratif}). Given a decomposition
\[
\xi\left(  0\right)  =\xi_{L}\left(  0\right)  +\xi_{S}\left(  0\right)
\]
of an initial condition $\xi\left(  0\right)  $, if the system%

\begin{align*}
\partial_{t}\xi_{L}+\mathcal{L}_{u_{L}}\xi_{L}+\mathcal{L}_{u_{S}}\xi_{L}  &
=\Delta\xi_{L}\\
\partial_{t}\xi_{S}+\mathcal{L}_{u_{S}}\xi_{S}+\mathcal{L}_{u_{L}}\xi_{S}  &
=\Delta\xi_{S}%
\end{align*}
with initial condition $\left(  \xi_{L}\left(  0\right)  ,\xi_{S}\left(
0\right)  \right)  $ has a solution, then $\xi=\xi_{L}+\xi_{S}$ is a solution
of the full 3D Navier-Stokes equations, solution decomposed in the two
\textquotedblleft scales\textquotedblright\ $\xi_{L}$ and $\xi_{S}$. Consider
the first equation, for the large scales. We have%
\[
\mathcal{L}_{u_{S}}\xi_{L}=u_{S}\cdot\nabla\xi_{L}-\xi_{L}\cdot\nabla u_{S}.
\]
We may also write
\[
\mathcal{L}_{u_{S}}\xi_{L}=\Pi\left(  u_{S}\cdot\nabla\xi_{L}\right)
-\Pi\left(  \xi_{L}\cdot\nabla u_{S}\right)
\]
since $\Pi\left(  \mathcal{L}_{u_{S}}\xi_{L}\right)  =\mathcal{L}_{u_{S}}%
\xi_{L}$ (but this is not true separately for the two addends). The equation
for the large scales then is%
\[
\partial_{t}\xi_{L}+\mathcal{L}_{u_{L}}\xi_{L}+\Pi\left(  u_{S}\cdot\nabla
\xi_{L}\right)  =\Delta\xi_{L}+\Pi\left(  \xi_{L}\cdot\nabla u_{S}\right)  .
\]
Assume we may apply the arguments described above. We get (approximately)\ the
equation
\[
\partial_{t}\xi_{L}+\mathcal{L}_{u_{L}}\xi_{L}+\Pi\left(  u_{S}\cdot\nabla
\xi_{L}\right)  =\Delta\xi_{L}.
\]

The model considered in this work corresponds to the idealization when $u_{S}$
is replaced by a white noise in time, idealization reminiscent of stochastic
reduction techniques like those more carefully developed in \cite{MajdaTV}. To
be fair, let us notice that the isotropic noise considered in our work is
incompatible with the orthogonality conditions (\ref{stratif}), making the
above justification still incomplete even at a very heuristic ground.

\bigskip

\noindent \textbf{Acknowledgements.} Marek Capi\'{n}ski gave a talk in Ludwig Arnold group in Bremen around 1987,
conjecturing that stochastic transport in parabolic PDEs could have a similar
stabilizing effect as the skew-symmetric linear state dependent noise used by
Arnold, Crauel and Wishtutz in their theory of stabilization by noise
\cite{ArnoldCW, Arnold}. The first author, attending that talk, was
permanently inspired by that conjecture, which however is still unproven,
although related results exist in many directions (see \cite{Constantin} and
references therein). We are grateful also to Zdzis{\l}aw Brze\'{z}niak for several
discussions on Capi\'{n}ski's suggestion. See \cite{BCF} for a first attempt to
use transport noise in 2D Navier-Stokes equations. The result here is not a solution of that
problem but it is based on a similar intuition.

The second author is grateful to the National Key R\&D Program of China (No. 2020YFA0712700), the National Natural Science Foundation of China (Nos. 11688101, 11931004, 12090014) and the Youth Innovation Promotion Association, CAS (2017003).


\begin{thebibliography}{9}

\bibitem{Arnold} L. Arnold. Stabilization by noise revisited. \emph{Z. Angew. Math.
Mech.} 70 (1990), no. 7, 235--246.

\bibitem{ArnoldCW} L. Arnold, H. Crauel, V. Wihstutz. Stabilization of
linear systems by noise. \textit{SIAM J. Control Optim.} \textbf{21} (1983),
451--461.

\bibitem{BabinMahalovNikolaenko} A. Babin, A. Mahalov, B. Nicolaenko. Global
splitting, integrability and regularity of 3D Euler and Navier-Stokes
equations for uniformly rotating fluids. \emph{Europ. J. Mech. B/Fluids} \textbf{15} (1996), 291--300.

\bibitem{BBF} D. Barbato, H. Bessaih, B. Ferrario. On a stochastic
Leray-$\alpha$ model of Euler equations. \emph{Stoch. Proc. Appl.} \textbf{124} (2014), no. 1, 199--219.

\bibitem{Billingsley} P. Billingsley. Convergence of Probability Measures. Second edition. Wiley Series in Probability and Statistics: Probability and Statistics. A Wiley-Interscience Publication. \emph{John Wiley \& Sons, Inc., New York,} 1999.

\bibitem{BCF} Z. Brze\'{z}niak, M. Capi\'{n}ski, F. Flandoli. Stochastic Navier-Stokes equations with multiplicative noise. \emph{Stochastic Anal. Appl.} \textbf{10} (1992), no. 5, 523--532.

\bibitem{ButMyt} O. Butkovski, L. Mytnik. Regularization by noise and flows
of solutions for a stochastic heat equation. \textit{Ann. Probab.} \textbf{47} (2019), 165--212.

\bibitem{CFF} E. Chiodaroli, E. Feireisl, F. Flandoli. Ill posedness for the
full Euler system driven by multiplicative white noise. to appear on \emph{Indiana Univ. Math. J.}, see arXiv:1904.07977.

\bibitem{Constantin} P. Constantin, A. Kiselev, L. Ryzhik, A. Zlato\v{s}.
Diffusion and mixing in fluid flow. \textit{Ann. of Math.} \textbf{168}
(2008), 643--674.

\bibitem{DapDeb} G. Da Prato, A. Debussche. Ergodicity for the 3D stochastic
Navier-Stokes equations. \textit{J. Math. Pures Appl.} (9) 82 (2003), no. 8,
877--947.

\bibitem{DaPrato Fla} G. Da Prato, F. Flandoli. Pathwise uniqueness for a
class of SDE in Hilbert spaces and applications. \textit{J. Funct. Anal.}
\textbf{259} (2010), 243--267.

\bibitem{DFPR} G. Da Prato, F. Flandoli, E. Priola, M. R\"{o}ckner. Strong
uniqueness for stochastic evolution equations in Hilbert spaces perturbed by
a bounded measurable drift. \textit{Ann. Probab.} \textbf{41} (2013),
3306--3344.

\bibitem{DFRV} G. Da Prato, F. Flandoli, M. R\"{o}ckner, A. Yu.
Veretennikov. Strong uniqueness for SDEs in Hilbert spaces with nonregular
drift. \textit{Ann. Probab.} \textbf{44} (2016), 1985--2023.

\bibitem{Davie} A. M. Davie. Uniqueness of solutions of stochastic
differential equations, Int. Math. Res. Not. 24, Article ID rnm 124, 26 p.
(2007).

\bibitem {DebTsu} A. Debussche, Y. Tsutsumi. 1D quintic nonlinear
Schr\"{o}dinger equation with white noise dispersion. \emph{J. Math. Pures Appl.} \textbf{96}
(2011), no. 4, 363--376.

\bibitem{DFV} F. Delarue, F. Flandoli, D. Vincenzi. Noise prevents collapse of
vlasov-poisson point charges. \emph{Comm. Pures Appl. Math.} \textbf{67} (2014), 1700--1736.

\bibitem{Feff} C. L. Fefferman. Existence and smoothness of the Navier-Stokes
equations, the millennium prize problems, Clay Math. Inst., Cambridge 2006, 57--67.

\bibitem{FGL} F. Flandoli, L. Galeati, D. Luo. Scaling limit of stochastic 2D Euler equations with transport noises to the deterministic Navier--Stokes equations. \emph{J. Evol. Equ.} (2020), http://link.springer.com/article/10.1007/s00028-020-00592-z.

\bibitem{FGP} F. Flandoli, M. Gubinelli, E. Priola. Well posedness of the
transport equation by stochastic perturbation. \emph{Invent. Math.} \textbf{180}
(2010), 1--53.

\bibitem{FGP2} F. Flandoli, M. Gubinelli, E. Priola. Full well-posedness of
point vortex dynamics corresponding to stochastic 2D Euler equations. \emph{Stoch.
Proc. Appl.} \textbf{121} (2011), no. 7, 1445--1463.


\bibitem{FlaLuoAoP} F. Flandoli, D. Luo. Convergence of transport noise to Ornstein-Uhlenbeck for 2D Euler equations under the enstrophy measure. \emph{Ann. Probab.} \textbf{48} (2020), no. 1, 264--295.

\bibitem{FlaMah} F. Flandoli, A. Mahalov. Stochastic three-dimensional
rotating Navier--Stokes equations: averaging, convergence and regularity.
\emph{Arch. Rational Mech. Anal.} \textbf{205} (2012), no. 1, 195--237.

\bibitem{FlaRom CKN} F. Flandoli, M. Romito. Partial regularity for the
stochastic Navier-Stokes equations. \emph{Trans. Amer. Math. Soc.} 354 (2002),
2207--2241.

\bibitem{FlaRom Markov} F. Flandoli, M. Romito. Markov selections for the 3D
stochastic Navier-Stokes equations. \textit{Probab. Theory Related Fields}
140 (2008), no. 3-4, 407--458.

\bibitem{Galeati} L. Galeati. On the convergence of stochastic transport equations to a deterministic parabolic one. \emph{Stoch. Partial Differ. Equ. Anal. Comput.} \textbf{8} (2020), no. 4, 833--868.

\bibitem{GassiatGess} P. Gassiat, B. Gess. Regularization by noise for
stochastic Hamilton-Jacobi equations. \emph{Probab. Theory Relat. Fields} \textbf{173} (2019), 1063--1098.

\bibitem{GessMau} B. Gess, M. Maurelli. Well-posedness by noise for scalar
conservation laws. \emph{Comm. Partial Diff. Eq.} \textbf{43} (2018), no. 12, 1702--1736.

\bibitem{Gyongy and co} I. Gy\"{o}ngy. Existence and uniqueness results for
semilinear stochastic partial differential equations. \textit{Stochastic
Process. Appl.} 73 (1998), no. 2, 271--299.

\bibitem{HZZ} M. Hofmanov\'{a}, R. Zhu, X. Zhu. Non-uniqueness in law of stochastic 3D Navier--Stokes equations. arXiv:1912.11841v1.

\bibitem{Holm} D. D. Holm. Variational principles for stochastic fluid
dynamics. \emph{Proc. Royal Soc.} A \textbf{471} (2015), 20140963.

\bibitem{Iyer} G. Iyer, X. Xu, A. Zlatos. Convection induced singularity
suppression in the Keller-Siegel and other Non-liner PDEs, arXiv:1908.01941.

\bibitem{Kunita} H. Kunita. \textit{Stochastic differential equations and
stochastic flows of diffeomorphisms}, Ecole d'\'{e}t\'{e} de probabilit\'{e}s
de Saint-Flour, XII---1982, 143-303, Lecture Notes in Math. \textbf{1097},
Springer, Berlin, 1984.

\bibitem{Kry-Roeck} N. V. Krylov, M. R\"{o}ckner. Strong solutions of
stochastic equations with singular time dependent drift. \textit{Probab.
Theory Related Fields} \textbf{131} (2005), 154--196.

\bibitem{Kurtz} T. Kurtz. The Yamada-Watanabe-Engelbert theorem for general stochastic equations and inequalities. \emph{Electron. J. Probab.} \textbf{12} (2007), 951--965.

\bibitem{MajdaTV} A. J. Majda, I. Timofeyev, E. Vanden-Eijnden. A
mathematical framework for stochastic climate models. \emph{Comm. Pure Appl. Math.}
\textbf{54} (2001), 891--974.

\bibitem{MikRoz} R. Mikulevicius, B. L. Rozovskii. Global $L^2$-solutions of stochastic Navier-Stokes equations. \textit{Ann.
Probab.} 33 (2005), no. 1, 137--176.

\bibitem{RL} B. L. Rozovsky, S. V. Lototsky. Stochastic evolution systems. Linear theory and applications to non-linear filtering. Second edition. Probability Theory and Stochastic Modelling, 89. \emph{Springer, Cham}, 2018.

\bibitem{Simon} J. Simon. Compact sets in the space $L^p(0,T; B)$. \emph{Ann. Mat. Pura Appl.} \textbf{146} (1987), 65--96.

\bibitem{Tao} T. Tao. Finite time blowup for an averaged three-dimensional
Navier-Stokes equation. \textit{J. Amer. Math. Soc.} \textbf{29} (2016),
601--674.

\bibitem{Temam} R. Temam. Navier-Stokes equations and nonlinear functional analysis. Second edition. CBMS-NSF Regional Conference Series in Applied Mathematics, 66. Society for Industrial and Applied Mathematics (SIAM), Philadelphia, PA, 1995.

\bibitem{Veret} Yu. A. Veretennikov. On strong solution and explicit
formulas for solutions of stochastic integral equations. \emph{Math. USSR Sb.}
\textbf{39} (1981), 387--403.

\bibitem{Vincent Meneguzzi} A. Vincent, M. Meneguzzi. The spatial structure
and statistical properties of homogeneous turbulence. \emph{J. Fluid Mech.} \textbf{%
225} (1991), 1--20.

\end{thebibliography}
\end{document}